\pdfoutput=1

\documentclass[
	bigMargin = false,
	font = palatino,
	final
	]{MyClass}

\usepackage{pdflscape}
\usepackage{myMathHeader}
\MyNewMathOperator{\projectiveSpace}	{command={\mathbb{P}}, display={$\projectiveSpace(V)$}, description={Projective space of a vector space $V$}}
\let\img\relax
\MyNewMathOperator{\img}			{operator={Range}, display={$\img f$}, description={Image of the map $f$}}

\newcommand{\barpartial}{{\bar\difp}}
\newcommand{\bfC}{{\mathbb C}}
\newcommand{\bfR}{{\mathbb R}}

\newcommand{\bfP}{{\mathbf P}}

\newcommand{\barj}{{\overline j}}

\newcommand{\bark}{{\overline k}}
\newcommand{\barl}{{\overline \ell}}

\newcommand{\envalias}[2]{\newenvironment{#1}{\begin{#2}}{\end{#2}}}
\envalias{definition}{defn}
\envalias{proposition}{prop}
\envalias{theorem}{thm}

\Title{Cartan Geometry and Infinite-Dimensional Kempf--Ness Theory}

\Abstract{
	We pioneer the development of a rigorous infinite-dimensional 
framework for the Kempf--Ness theorem, addressing the significant 
challenge posed by the absence of a complexification for the 
symmetry group in infinite dimensions, \eg, the diffeomorphism 
group. We propose a novel approach, based on Cartan 
bundles, to generalize Kempf--Ness theory to infinite dimensions, 
invoking the fundamental role played by the Maurer--Cartan form.
This approach allows us 
to define and study objects essential for the Kempf--Ness theorem, 
such as the complex model for orbits and the Kempf--Ness function, 
as well as establishing its convexity properties and defining a 
generalized Futaki character.
We show how our framework can be applied to the study of various 
problems in K\"ahler geometry, deformation quantization, and 
gauge theory.
}



\Institution{sjtud}{
	School of Mathematical Sciences, CMA-Shanghai, Shanghai 
Jiao Tong University, 800 Dongchuan Road, 200240 China}

\Institution{tsinghua}{
	Yau Mathematical Sciences Center,
	Tsinghua University,
	Beijing 100084, China}
	
\Institution{sjtur}{School of Mathematical Sciences, CMA-Shanghai,
and Ministry of Education Laboratory of Scientific Computing (MOE-LSC),
Shanghai Jiao Tong University, 800 Dongchuan Road, 200240 China, and 
Section de Math\'ematiques, Ecole Polytechnique F\'ed\'erale de 
Lausanne, 1500 Lausanne, Switzerland.}
	
\Author[
	affiliation = sjtud,
	email = {diez@sjtu.edu.cn, research@tobiasdiez.de},
	funding = {Partially supported by National Natural Science Foundation of China grant 12350410357.}
	]{diez}{Tobias Diez}

\Author[
	affiliation = tsinghua,
	email = futaki@tsinghua.edu.cn,
]{futaki}{Akito Futaki}

\Author[
	affiliation = sjtur,
	email = {ratiu@sjtu.edu.cn, tudor.ratiu@epfl.ch},
	funding = {Partially supported by National Natural Science Foundation of China grant 12250710125 and NCCR SwissMAP grant of the Swiss National Science Foundation.}
	]{ratiu}{Tudor S. Ratiu}

\begin{document}

\MakeTitle

\tableofcontents

\listoftodos

\section{Introduction}
The fruitful interplay between Hamiltonian systems with symmetry 
and complex geometry is of paramount importance in symplectic 
geometry. A particularly powerful tool in connecting these areas 
is the Kempf--Ness theorem \parencite{KempfNess1979} which 
describes the equivalence between the notions of quotient in 
symplectic and algebraic geometry. The theorem states that 
the symplectic quotient of a Hamiltonian action by a compact 
Lie group is isomorphic to the GIT quotient of the associated 
action of the complexified group. 

While being landmarks in their own right, these rigorous results 
about finite-dimensional systems expand their full strength as a 
conceptual framework for the study of geometric partial 
differential equations. Often, a difficult system of PDEs contains
non-evolutionary equations that can be formulated as a level-set 
constraint of a momentum map associated with an infinite-dimensional 
Lie group acting on an infinite-dimensional symplectic manifold (this
is the case in, \eg, electromagnetism, Yang--Mills theory, and, 
in some sense, general relativity). When this is the case, 
the finite-dimensional techniques surrounding the Kempf--Ness 
theorem serve as a blueprint to come up with fundamental 
conjectures about obstructions and stability of solutions to 
the original system of PDEs. Examples include the work of 
Atiyah and Bott \parencite{AtiyahBott1983} on Yang--Mills 
connections on a Riemann surface, the Donaldson--Uhlenbeck--Yau 
correspondence \parencite{Donaldson1985, UhlenbeckYau1986} 
relating stable holomorphic vector bundles and Hermitian 
Yang--Mills connections, the Kobayashi--Hitchin correspondence 
\parencite{Kobayashi1982,Hitchin1979} and the recent 
resolution of the Yau--Tian--Donaldson conjecture \parencite{ChenDonaldsonSun2015,Tian2015}.
All these examples are quite different in nature, but they 
all share the same abstract framework grounded in 
infinite-dimensional symplectic geometry.

In view of the wide success of this \emph{conceptual} picture, 
it is perhaps astonishing that no \emph{rigorous} 
infinite-dimensional framework is available yet.
In this paper, we start the development of a general 
theory of the Kempf--Ness theorem in infinite dimensions 
with the goal to encompass the above-mentioned examples 
as specific cases. Passing to infinite dimensions, however, 
is quite challenging because the symmetry group often does
not admit a complexification. For example, the diffeomorphism 
group does not complexify. This major obstacle makes it 
unclear what the correct notion of stability and what 
the GIT quotient should be. 

In this paper, we propose a solution to these problems using 
the framework of Cartan bundles. Our starting point is the observation 
that, in finite dimensions, the Maurer--Cartan form essentially 
determines the Lie group. Roughly speaking, a simply connected 
manifold endowed with a Lie-algebra valued \( 1 \)-form satisfying 
the Maurer--Cartan equation (and possessing a few other natural 
properties) is locally a Lie group, see 
\parencite[Theorem~8.7]{Sharpe1997} for the precise statement.
The Maurer--Cartan form is hence a fundamental object in the 
theory of Lie groups and this viewpoint is the basis of 
its extension to Cartan bundles. The framework of Cartan bundles 
translates, without much effort, to infinite dimensions. Only the 
(local) integration of a Cartan geometry to a Lie group is not 
available in infinite dimensions, although various partial 
integration results related to Cartan forms have recently been 
established in infinite dimensions, see 
\parencite{MichorMichor2024,GloecknerNeeb2013,Neeb2006}.
Our main idea is to replace the complexification of the symmetry 
group in the GIT theory by an appropriate Cartan bundle.

The language of Cartan bundles allows us to rigorously 
define and study objects that are usually connected to 
the complexified action and that are essential for the 
Kempf--Ness theorem. For example, we define the notion of 
a Cartan model for the orbit through a point in a 
symplectic manifold endowed with a Hamiltonian action 
by a Lie group. This is just a Cartan version for the 
important notion of the orbit of the complexified action 
in the finite-dimensional case. Moreover, we introduce 
the Kempf--Ness function and establish its convexity 
properties along geodesics, see 
\cref{prop:cartan:kempfNess,prop:cartan:kempfNess:positive}. 
We also define a generalized Futaki character as a 
character on the stabilizer of the Cartan action. We show 
that the generalized Futaki character is constant in 
a certain sense and that it obstructs the existence of 
zeros of the momentum map, see \cref{prop:cartan:constFutaki}.
An interesting feature of our theory is that it frees the 
GIT theory from its embedding in complex geometry and 
opens up the possibility to study the Kempf--Ness theorem 
in a more general context. For example, in 
\cref{ex:cartan:informationGeometry} we show that our 
theory can be applied to information geometry, where the 
group of diffeomorphisms acts on the space of probability 
measures --- a setting usually not associated with complex 
geometry.

As we discuss in the second part of the paper, our general 
results have immediate applications to various geometric 
PDEs. For example, \textcite{Fujiki1992,Donaldson1997} 
remarked that the existence problem of a constant scalar 
curvature K\"ahler (cscK for short) metric is equivalent to 
finding the zeros of a momentum map associated with the 
action of the Hamiltonian diffeomorphism group on an 
infinite-dimensional space. Moreover, \textcite{Donaldson1999a} 
also argued that a certain principal bundle 
\( \SectionSpaceAbb{P} \) over the space of K\"ahler metrics 
is a natural candidate for a complexification of the Hamiltonian 
diffeomorphism group. We support this claim by showing that 
the bundle \( \SectionSpaceAbb{P} \) carries a natural flat 
Cartan connection; so (locally) it would be a Lie group if 
not for the infinite-dimensional nature of the problem.
(This observation was actually the starting point of our 
investigation.) Applied to the cscK problem, our general 
theory recovers the Mabuchi K-energy as the Kempf--Ness function.
Convexity of the Mabuchi K-energy along geodesics is a 
well-known result, see \parencite{Mabuchi1987a}.
Moreover, the generalized Futaki character of the general 
theory yields the classical Futaki character 
\parencite{Futaki1983}, which obstructs the existence 
of constant scalar curvature K\"ahler metrics.
We emphasize that the novelty is not in the results 
themselves, or that they \emph{formally} fit in a GIT 
setting, but in the conceptual and \emph{rigorous 
infinite-dimensional framework} that allows us to 
derive them in a unified way with minimal assumptions.
To illustrate the power of our theory, we also discuss 
applications to various other geometric PDEs, where we 
recover known results from the literature indicated below as well as obtain new conclusions.
\begin{enumerate}
	\item Perturbation of the scalar curvature by Chern forms: \parencite{Bando2006,Futaki2006,Futaki2008}. 
	\item Z-critical K\"ahler metrics: \parencite{Dervan2023,DervanHallam2023}.
	\item Symplectic connections with applications to deformation quantization: \parencite{Fuente-Gravy2016,FutakiOno2018,FutakiFuenteGravy2019}.
	\item Hermitian Yang--Mills connections: \parencite{Donaldson1985,UhlenbeckYau1986}.
	\item Z-critical Yang--Mills connections: \parencite{DervanMcCarthySektnan2020,CollinsYau2021}. 
\end{enumerate}
In each case, one first gives a momentum map interpretation of the PDE under study and then constructs an appropriate Cartan bundle.
Our general theory then provides a package of results, such as the Kempf--Ness function, the generalized Futaki character, and the uniqueness of solutions modulo automorphisms.
Some applications of our theory to concrete problems are often technically challenging and require a careful analysis (see, \eg, the recent study of the LYZ equation in \parencite{CollinsYau2021}).

An overview of our results and their implications in various geometric settings is  given in \cref{tab:applications}.
This gives only a brief summary. It is not intended to provide a complete overview of the vast literature.
More detailed discussions of the applications are given in \cref{sec:kaehler,sec:gauge}.

\paragraph*{Acknowledgements}
The authors would like to thank Klas Modin and Cornelia Vizman for helpful discussions.

\begin{landscape}
	\thispagestyle{plain}
	\begin{table}
	\vspace*{-3em}
	\hspace*{-5em}
	\centering
	\setlength{\tabcolsep}{2ex} 
	\renewcommand{\arraystretch}{2.4} 
	\begin{tabular}{p{4cm}|p{3cm}p{3cm}p{3cm}p{3cm}p{3cm}p{3cm}}
		\toprule
		General result
			& cscK metrics
			& perturbed cscK metrics
			& Z-critical K\"ahler metrics
			& Symplectic connections
			& Hermitian YM connections
			& Z-critical YM connections
			\tabularnewline
		\midrule

		Generalized Futaki character \newline \Cref{eq:futaki:generalizedFutakiInvariant}
			& Futaki character \parencite{Futaki1983}
			& \parencite{Futaki2006}
			& \parencite{Dervan2023}
			& \parencite{Fuente-Gravy2016}
			& \parencite{FutakiMorita1985}
			& \parencite{CollinsYau2021} though not explicit
			\tabularnewline
		
	
		Kempf--Ness function \newline \cref{prop:cartan:kempfNess}
			& Mabuchi K-energy \parencite{Mabuchi1987}
			& New
			& \parencite{Dervan2023}
			& \parencite{FutakiFuenteGravy2019}
			& Donaldson functional \parencite{Donaldson1985,UhlenbeckYau1986}
			& \parencite{CollinsYau2021}
			\tabularnewline
	

		Stability* \newline \cref{defn:cartan:stability}
			& K-stability \newline \parencite{Tian1997,Donaldson2002}
			& New
			& \parencite{Dervan2023}
			& \parencite{FutakiFuenteGravy2019}
			& Slope stability of vector bundles
			& \parencite{CollinsYau2021,DervanMcCarthySektnan2020}
			\tabularnewline

		Uniqueness of zeros of the momentum map in a complex orbit up to the action of the stabilizer \newline \cref{auto}* 
			& \parencite{BermanBerndtsson2017}
			& New
			& New
			& \parencite{FutakiFuenteGravy2019}
			& Essential uniqueness of Hermitian YM connections \parencite{Donaldson1985} 
			& \parencite{CollinsYau2021} though not explicit
			\tabularnewline
	
		Extremal elements \newline \Cref{eq:cartan:extremalElements:extremalCondition}
			& extremal K\"ahler vector field \parencite{Futaki1988,FutakiMabuchi1995}
			& \parencite{Futaki2008}
			& \parencite{Dervan2023}
			& \parencite{FutakiOno2018}
			& New
			& New
			\tabularnewline
		\bottomrule
	\end{tabular}
	\caption{Applications of the general theory to various geometric problems. \\
	{
		\footnotesize
		*: In our general theory, stability and uniqueness of zeros of the momentum map are taken relative to \emph{smooth} geodesic rays. In applications, existence and uniqueness of sufficiently smooth geodesic rays is often a difficult problem and it may be advantageous to reformulate it in terms of algebraic-geometric objects such as test configurations.
	}}
	\label{tab:applications}
	\end{table}
\end{landscape}

\section{Infinite-dimensional GIT}
\label{sec:general}
\subsection{Cartan bundles}
\Textcite{Futaki1983} introduced a character on the Lie algebra of 
holomorphic vector fields that is constant as a function of the 
K\"ahler form in a given K\"ahler class and that obstructs the 
existence of a constant scalar curvature K\"ahler form in that 
class. One of our goals is to generalize this philosophy and, in 
the general setting of a Hamiltonian action, construct a character 
on the stabilizer of the complexified action that is constant in 
a certain sense and that obstructs the existence of zeros of the 
momentum map.

Let \( (M, \omega) \) be a symplectic manifold endowed with a Hamiltonian 
action by a Lie group \( G \). Let \( m \in M \).
Assume, for the moment, that \( G \) has a complexification \( G^c \)
 acting on \(M\).
Let \(\mathfrak{g}\) be the Lie algebra of \(G\), \(\mathfrak{g}_\mathbb{C}\)
its complexification, the Lie algebra  of \( G^c \). 
Morally speaking, for every \( \zeta \) in the 
stabilizer \( (\LieA{g}_\C)_m \), we associate a \( G \)-invariant 
map \( \tilde{F}_\zeta: G^c \to \R \) and show that the induced map 
on the quotient \( F_\zeta: G \backslash G^c \to \R \)
is constant. The map \( (\LieA{g}_\C)_m \ni 
\zeta \mapsto F_\zeta \in \sFunctionSpace(G \backslash G^c, \R) \) is then the 
\textit{generalized Futaki character}. As will be discussed in 
\cref{sec:kaehler}, this definition recovers the classical Futaki 
character by identifying \( G \backslash  G^c \) with the space 
of K\"ahler forms in a given K\"ahler class. The problem with 
the picture described above is that the group of interest \( G \) is the Lie group
of symplectomorphisms, which  may 
not admit a complexification \( G^c \).

Hence, the main issues we have to face is to define the 
generalized Futaki character in the absence of a complexification
\( G^c \) of the Lie group \( G \). More precisely, 
we need an appropriate replacement for the bundle 
\( G^c \to G \backslash  G^c \). This is achieved using the 
theory of Cartan bundles, which are generalizations of 
Klein bundles \( G \to H \backslash  G \), where \(H\) is a 
closed subgroup of \(G\). The main merit, 
from our point of view, is that the Cartan
theory generalizes without much effort to infinite dimensions, 
while the (local) integration of a Cartan geometry to a Klein 
geometry is not available in infinite dimensions. As we will 
see, the framework of Cartan connections also provides a 
natural formalization of the constructions of \textcite{Donaldson1999a} 
in the example of K\"ahler geometry.

Recall from \parencite{Sharpe1997} that a \textit{Klein pair} 
\( (\LieA{a}, \LieA{g}) \) consists of a Lie algebra \( \LieA{a} \) 
and a Lie subalgebra \( \LieA{g} \) of \( \LieA{a} \).
Moreover, we suppose that there exists a Lie group \( G \) that 
integrates \( \LieA{g} \) and that there exists a representation 
\( \AdAction \) of \( G \) on \( \LieA{a} \) that restricts to 
the adjoint representation of \( G \) on \( \LieA{g} \).
In infinite dimensions, we assume that \( \LieA{g} \) is closed and, 
as a topological vector space, has a closed
complement in \( \LieA{a} \). We will refer to this setting by 
saying that \( (\LieA{a}, \LieA{g}) \) is a \textit{Klein pair with 
group} \( G \). Crucially, no integrability assumption is made 
for \( \LieA{a} \). A \emphDef{Cartan geometry} modeled on 
the Klein pair \( (\LieA{a}, \LieA{g}) \) with group \( G \) on 
a manifold \( B \) is a principal \( G \)-bundle \( \pi: P \to B \) 
together with a \( G \)-equivariant \( 1 \)-form 
\( \theta \in \DiffFormSpace^1(P, \LieA{a}) \) that 
induces an isomorphism \( \theta_p: \TBundle_p P \to \LieA{a} \) 
for every \( p \in P \) and that satisfies \( \theta(\xi^*) = \xi \) 
for all \( \xi \in \LieA{g} \).
Here, \( \xi^* \in  \VectorFieldSpace(P)\) denotes
the infinitesimal generator, or fundamental vector field, defined by
\(\xi \). The \( 1 \)-form \( \theta \) 
is called a \emphDef{Cartan connection}. Contrary to the usual 
treatment in the literature, it will be convenient to work with 
a left principal \( G \)-bundle \( \pi: P \to B \) instead of a right 
principal bundle, so that the equivariance property of \( \theta \) 
reads \( \theta_{g \cdot p}(g \ldot X) = \AdAction_{g} \theta_p(X) \) 
for all \( p \in P \), \( X \in \TBundle_p P \), and \( g \in G \).
The \emphDef{curvature} of \( \theta \) is the 
\( 2 \)-form\footnotemark{} \( \Omega \defeq \dif \theta - \frac{1}{2} 
\wedgeLie{\theta}{\theta} \in \DiffFormSpace^2(P, \LieA{a}) \).
\footnotetext{Note the sign in front of the quadratic term, which 
is the opposite of the usual convention for \emph{right} principal 
bundles. If \( \alpha \) is an \( \LieA{a} \)-valued 
\(p\)-form and \( \beta \) an \( \LieA{a} \)-valued \(q\)-form on 
\(P\), then \( \wedgeLie{\alpha}{\beta} \) denotes the  
\(\LieA{a} \)-valued  \( (p+q) \)-form on \( P \) obtained as 
in the usual theory of \( \R \)-valued forms, except that 
multiplication of real numbers is replaced by the Lie bracket.}
A Cartan connection \( \theta \) is said to be \emphDef{torsion-free} 
if \( \Omega \) takes values in \( \LieA{g} \). A standard result in 
the theory of Cartan connections is that \( \theta \) induces a 
bundle isomorphism \( \TBundle B \to P \times_G (\LieA{a} \slash \LieA{g}) \); 
see \parencite[Theorem~3.15]{Sharpe1997}. The reader should keep the example
\( A = G^c \to G \backslash G^c = B \), with \( \theta \) being the 
Maurer--Cartan form, in mind. In fact, every Cartan geometry with
vanishing curvature is locally isomorphic to a homogenous space
in finite dimensions; see \parencite[Theorem~5.1]{Sharpe1997}.
In infinite dimensions, this may no longer hold, as \(\mathfrak{a} \) 
may not integrate to a suitable Lie group \(A\).

\begin{defn}
\label{def:cartan:orbitModel}
Let \( M \) be a manifold and \( G \) a Lie 
group acting on \( M \). 	
A Cartan geometry \( (P, \theta) \) modeled on the Klein pair 
\( (\LieA{a}, \LieA{g}) \) with group \( G \)
 is said to be an \emphDef{orbit} through the point \( m \in M \) 
if the following holds:
\begin{thmenumerate}
\item There exists a smooth \( G \)-equivariant map \( \chi: P \to M \) 
such that \( m \) is in the image of \( \chi \).
\item There exists a Lie subalgebra \( \LieA{a}_m \subseteq \LieA{a} \) 
and a right Lie algebra action of \( \LieA{a}_m \) 
on \( P \), denoted by \( \zeta^*_p = p \ldot \zeta \)
for 
\( \zeta \in \LieA{a}_m \) and \( p \in P \), which commutes 
with the \( G \)-action and which satisfies \( \tangent \chi(\zeta^*) = 0 \).
\qedhere
\end{thmenumerate}
\end{defn}
Given a Cartan model \( (P, \theta) \) which is an orbit through \( m \),
 we get an induced fiberwise linear bundle map 
 \( \rho: \LieA{a} \times P \to \chi^* \TBundle M \) over \( P \) 
 by setting
\begin{equation}
	\label{eq:cartan:infAction}
	\rho(\xi, p) \equiv \rho_p(\xi) \defeq 
	\tangent_p \chi \bigl( \theta^{-1}_p(\xi) \bigr) \in 
	\TBundle_{\chi(p)} M.
\end{equation}
Thinking of \( \rho \) as the infinitesimal action of \( \LieA{a} \) 
on \( M \), we slightly abuse notation and write 
\( \rho(\xi, p) \equiv \xi \ldot \chi(p) \) although 
\( \rho(\xi, p) \) might really depend on \( p \) and not 
only on \( \chi(p) \). Note that
\begin{equation}
	\bigl(\AdAction_g \xi\bigr) \ldot \chi(g \cdot p) = 
	\rho\bigl(\AdAction_g \xi, g \cdot p\bigr) = 
	g \ldot \rho\bigl(\xi, p\bigr) = 
	g \ldot \bigl(\xi \ldot \chi(p)\bigr)
\end{equation}
for every \( g \in G \), as one would expect from an infinitesimal 
action (but note that it only holds with respect to \( G \)).
The standard example we have in mind is the following.
\begin{example}
	\label{ex:cartan:orbit:homogenousSpace}
Let \( A \) be a Lie group acting on the manifold \( M \) and 
\( G \subseteq A \) a principal Lie subgroup, \ie, a Lie subgroup such that the quotient \( G \backslash A \) is a manifold and \( A \to G \backslash A \) is a principal \( G \)-bundle (this is automatic in finite dimensions).
Then \( A \to G \backslash A \) endowed with the right Maurer--Cartan 
form \( \theta_a(\xi \ldot a) = \xi \) is a Cartan model for the 
\( A \)-orbit through every point \( m \in M \) with \( \chi(a) = 
a \cdot m \) and \( \LieA{a}_m \) being the Lie algebra of the 
stabilizer \( A_m \) acting by right translations on \( A \).
For all \( \xi \in \LieA{a} \), we have \( \rho(\xi, a) = 
\tangent_a \chi (\xi \ldot a) = \xi \ldot (a \cdot m) \).
\end{example}

The following provides the appropriate generalization of an action 
of \( G^c \) on a symplectic manifold \( M \), now rephrased in terms 
of a Cartan geometry.
Recall that a Lie algebra action \( \Upsilon_m \) 
of \(\LieA{g}\) on a manifold \( M \) endowed with an almost complex 
structure \( j \) extends to an 
\( \I \)-\( j \)-complex-linear map \( \Upsilon^c_m: \LieA{g}_\C 
\to \TBundle_m M\) given by 
\((\xi_1 + \I \xi_2) \mapsto (\xi_1 + \I \xi_2) \ldot m 
\defeq \xi_1 \ldot m + j \, (\xi_2 \ldot m)\) for \( \xi_1 ,
\xi _2 \in  \LieA{g} \).
The image \( \Upsilon^c_m(\LieA{g}_\C) \) is denoted by \( \LieA{g}_\C \ldot m \).
\begin{defn}
	\label{def:cartan:complexOrbitModel}
Let \( M \) be a manifold, \( G \) a Lie group acting on \( M \),
and \( j \) an almost complex structure on \( M \).
A Cartan geometry \( (P, \theta) \) modeled on the Klein pair 
\( (\LieA{a}, \LieA{g}) \) with group \( G \) is said to be 
\emphDef{a model for the complex orbit} through the point 
\( m \in M \) if it is an orbit through \( m \) and
\begin{equation}
	\label{eq:cartan:complexOrbitModel}
\img \tangent_p \chi = \set{\xi \ldot \chi(p) \given \xi \in \LieA{a}} 
= \LieA{g}_\C \ldot \chi(p) = \img \Upsilon^c_{\chi(p)}
\end{equation}
for all \( p \in P \).
\end{defn}
Note that \( \LieA{a} \) might be strictly smaller than \( \LieA{g}_\C \) 
and, in fact, may not even be a complex vector space.
The important property is that its action \( \rho \) on \( M \) yields the 
same distribution as the complexified action \( \Upsilon^c \).
\begin{example}
	\label{ex:cartan:complexifiedGroup}
If the complexification \( G^c \) of \( G \) exists and acts holomorphically 
on \( M \), then \( G^c \to G \backslash G^c \) endowed with the right 
Maurer--Cartan form is a model for the complex orbit through every point 
\( m \in M \). In this case, the underlying Klein pair is given by the 
embedding of \( \LieA{g} \) into \( \LieA{g}_\C \) as the real part.
\end{example}
\begin{example}[Real GIT]
	\label{ex:cartan:realGIT}
Following \textcite{HeinznerSchwarz2007}, consider a compact Lie group 
\( U \) whose complexification \( U^c \) acts holomorphically on 
a K\"ahler manifold \( (M, \omega, j) \). 
Moreover, let \( A \subseteq U^c \) be a closed Lie subgroup such 
that the Cartan decomposition \( U^c = \exp(\I \LieA{u}) \, U \) 
induces a diffeomorphism \( A = \exp(\LieA{p}) \, G \), where 
\( G = U \intersect A \) and \( \LieA{p} \subseteq \I \LieA{u} \) 
is a \( \AdAction_G \)-stable subspace. Then \( A \to G \backslash A \) 
endowed with the Maurer--Cartan form is a Cartan bundle and the action 
induced via the inclusion \( A \subseteq U^c \) realizes it as an orbit 
through any point \( m \in M \). While the Cartan geometry of \( U^c \) 
yields a model for the \emph{complex} \( U \)-orbit, the Cartan geometry 
of \( A \) may no longer be a model for the \emph{complex} \( G \)-orbit.   
\end{example}
\begin{defn}
	\label{defn:cartan:duality}
A duality of the Lie algebra \( \LieA{g} \) is a vector space 
\( \LieA{g}^* \) together with a non-degenerate pairing 
\( \kappa: \LieA{g}^* \times \LieA{g} \to \R \).
A duality of the Klein pair \( (\LieA{a}, \LieA{g}) \) is 
a map \( \kappa_\LieA{a}: \LieA{g}^* \times \LieA{a} \to \R \) 
that is \( \AdAction_G \)-invariant and that vanishes on 
\( \LieA{g}^* \times \LieA{g} \).
\end{defn}
Note that \( \LieA{g}^* \) is free to be chosen; it does not 
have to be the topological dual space of \( \LieA{g} \).
In many examples, one has \( \LieA{g}^* = \LieA{g} \) 
with \( \kappa \) being the Killing form (in the finite-dimensional 
case) or some form of the \( \LTwoFunctionSpace \)-paring (in the 
infinite-dimensional case).
If \( \LieA{a} = \LieA{g}_\C \) and \( \LieA{g} \subseteq \LieA{g}_\C \) 
is embedded as the real part, then a natural choice for the pairing 
\( \kappa_\LieA{a} \) is the imaginary part \( \kappa_\LieA{a} = \Im \kappa_\C \)  of the extension 
of \( \kappa \) to a Hermitian inner product 
\( \kappa_\C: \LieA{g}_\C \times \LieA{g}_\C \to \C \)
on \( \LieA{g}_\C \), namely
\begin{equation}
\label{kappa_C}
\kappa_\C (\xi_1 + \I \xi_2, \eta_1 + \I \eta_2) \defeq 
\kappa (\xi_1, \eta_1) + \I \kappa(\xi_2, \eta_1) - 
\I \kappa(\xi_1, \eta_2) + \kappa(\xi_2, \eta_2).
\end{equation}

\begin{remark}\label{sign1}Note also that if a positive definite Hermitian form $h$ is $\C$-anti-linear in the first component and $\C$-linear 
in the second component then $h$ has a Riemannian metric $g$ as its real part and a K\"ahler form $\omega$ as its imaginary part,
\ie, $h = g + \I\omega$.
\end{remark}

\subsection{Equivariance of the momentum map and the Calabi operator}

Let \(G\) be a Lie group and 
\( (M, \omega) \) a symplectic manifold with a 
\( G \)-Hamiltonian action whose momentum map 
\( J: M \to \LieA{g}^* \) is defined with respect to 
some non-degenerate pairing \( \kappa: \LieA{g}^* 
\times \LieA{g} \to \R \). That is, \( J\) satisfies
\begin{equation}\label{momentum1}
\omega_m(\xi.m,X) + \kappa (\tangent_m J(X),\xi) = 0
\end{equation}
for all \( m \in M \), \( X \in \TBundle_m M \), and \( \xi \in \LieA{g} \).
Consider a Cartan model  \( (P, \theta) \) for the orbit through 
a chosen point \( m \in M \) with modeling map \( \chi: P \to M \) 
and infinitesimal action \( \rho \) as in \cref{def:cartan:orbitModel}.
For every \( p \in P \), we define the \emphDef{Calabi operator}
\begin{equation}
	\label{eq:cartan:calabiOperator}
	C_p \defeq \tangent_{\chi(p)} J \circ \rho_p: \LieA{a} \to \LieA{g}^*, \quad 
\text{\ie}, \quad C_p (\xi) = \tangent_{\chi(p)} J \bigl(\xi \ldot \chi(p)\bigr), 
\; \; \xi \in  \LieA{a}.
\end{equation}
In the K\"ahler example discussed in \cref{sec:kaehler}, the Calabi operator coincides with the operator used by Calabi in his study of extremal K\"ahler metrics \parencite{Calabi1985}.

Note that if \( J \) is (infinitesimally) equivariant, then \( C_p(\xi) 
= -\CoadAction_\xi J(\chi(p)) \) for all \( \xi \in \LieA{g} \).
This suggests the following definition.
\begin{defn}
	\label{def:cartan:momentumMapEquivariance}
We say that \( J \) is \emphDef{\( \LieA{a} \)-equivariant} relative 
to a pairing \( \kappa_{\LieA{a}}: \LieA{g}^* \times \LieA{a} \to \R \) 
of the Klein pair \( (\LieA{a}, \LieA{g}) \) if 
\begin{equation}
	\label{eq:cartan:momentumMapNonEquivariance:calabi}
\kappa_{\LieA{a}}(C_p \xi, \eta) - \kappa_{\LieA{a}}(C_p \eta, \xi) 
= - \kappa_{\LieA{a}}\bigl(J(\chi(p)), \commutator{\xi}{\eta}\bigr), 
\qquad p \in P, \; \xi, \eta \in \LieA{a}
\end{equation}
and if the non-equivariance \( 1 \)-cocycle \( \sigma: G \to \LieA{g}^* \) 
of \( J \) satisfies
\begin{equation}
	\label{eq:cartan:momentumMapNonEquivariance:cocycle}
\kappa_{\LieA{a}} \bigl(\sigma(g), \xi\bigr) = 0 
\end{equation}
for all \( g \in G \) and \( \xi \in \LieA{a} \).
\end{defn}
\begin{prop}
	\label{prop:cartan:momentumMapEquivarianceComplexCase}
Let \( \LieA{g} \) be a Lie algebra and \( \LieA{m} \subseteq \LieA{g} \) 
a subspace. Then on the Klein pair 
\( \LieA{g} \subseteq \LieA{g} \oplus \I \LieA{m} \equiv \LieA{a} \) 
define the duality \( \kappa_\LieA{a}: \LieA{g}^* \times \LieA{a} \to \R \)
\begin{equation}
	\kappa_\LieA{a}(\mu, \xi_1 + \I \xi_2) \defeq - \kappa(\mu, \xi_2), 
	\qquad \mu \in \LieA{g}^*, \; \;\xi_1 \in \LieA{g}, \; \;\xi_2 \in \LieA{m}.
\end{equation}
Let \( (P, \theta) \) be a Cartan geometry relative to the Klein 
pair \( (\LieA{a}, \LieA{g}) \), which is an orbit through 
\( m\) in an almost complex manifold \((M,j) \) where \( j \) is an almost complex structure on \( M \) compatible with 
the symplectic structure \( \omega \).
Assume that the infinitesimal action \( \rho \) 
is complex linear in the sense that
	\begin{equation}
		\rho(\xi_1 + \I \xi_2, p) = \rho(\xi_1, p) + j \rho(\xi_2, p).
	\end{equation}
Then the momentum map \( J: M \to \LieA{g}^* \) is \( \LieA{a} \)-equivariant 
relative to \( \kappa_\LieA{a} \) if and only if the non-equivariance 
\( 1 \)-cocycle \( \sigma: G \to \LieA{g}^* \) of \( J \) satisfies
	\begin{equation}
		\kappa\bigl(\sigma(g), \xi\bigr) = 0
	\end{equation}
	for all \( g \in G \) and \( \xi \in \LieA{m} \).
\end{prop}
\begin{proof}
Recall that the non-equivariance \( 2 \)-cocycle 
\( \Sigma: \LieA{g} \times \LieA{g} \to \R \) of the 
momentum map \( J \) is defined by
	\begin{equation}\begin{split}
	\label{eq:normedsquared:nonequiv_first}
	\Sigma(\xi, \eta) &\defeq
	\kappa\left(\tangent_e \sigma (\xi), \eta \right) =
	\kappa(J(m), \commutator{\xi}{\eta}) +
	\omega_{m}(\xi \ldot m, \eta \ldot m) \\
	&=\kappa(J(m), \commutator{\xi}{\eta}) + 
	\kappa(\tangent_m J(\xi \ldot m), \eta).
	\end{split}
	\end{equation}
Let \( \xi = \xi_1 + \I \xi_2 \) and \( \eta = \eta_1 + \I \eta_2 \) 
with \( \xi_1, \eta_1 \in \LieA{g} \) and \( \xi_2, \eta_2 \in \LieA{m} \).
Then
\begin{equation}\begin{split}
	\kappa_\LieA{a} &\bigl(C_p \xi, \eta\bigr) + 
	\kappa_\LieA{a} \bigl(J\bigl(\chi(p)\bigr), \commutator{\xi}{\eta}\bigr) \\
	&= -\kappa \bigl(C_p (\xi_1 + \I \xi_2), \eta_2\bigr) - 
	\kappa \bigl(J\bigl(\chi(p)\bigr), \commutator{\xi_1}{\eta_2}\bigr) - 
	\kappa \bigl(J\bigl(\chi(p)\bigr), \commutator{\xi_2}{\eta_1}\bigr) \\
	&= -\Sigma(\xi_1, \eta_2) -\kappa \bigl(C_p (\I \xi_2), \eta_2\bigr) + 
	\Sigma(\eta_1, \xi_2) - \kappa(C_p \eta_1, \xi_2).
	\end{split}\end{equation}
On the other hand, by the complex linearity of the action \( \rho \), 
we find
\begin{equation}\label{eq:cartan:momentumMapNonEquivariance:CpOnI}\begin{split}
\kappa \bigl(C_p (\I \xi_2), \eta_2\bigr)
&= \kappa \bigl(\tangent_{\chi(p)} J(j \xi_2 \ldot \chi(p)), \eta_2\bigr) \\
&= - \omega(\eta_2 \ldot \chi(p), j \xi_2 \ldot \chi(p)) \\
&= \omega(j \eta_2 \ldot \chi(p), \xi_2 \ldot \chi(p)) \\
&= \kappa \bigl(C_p (\I\eta_2), \xi_2\bigr).
\end{split}\end{equation}
Hence, in summary,
\begin{equation}
\kappa_\LieA{a} \bigl(C_p \xi, \eta\bigr) - 
\kappa_\LieA{a} \bigl(C_p \eta, \xi\bigr) + 
\kappa_\LieA{a} \bigl(J\bigl(\chi(p)\bigr), \commutator{\xi}{\eta}\bigr)
= -\Sigma(\xi_1, \eta_2) + \Sigma(\eta_1, \xi_2).
\end{equation}
Since \( \eta_2 \) and \( \xi_2 \) are elements of \( \LieA{m} \), 
the right-hand side vanishes if the non-equivariance \( 1 \)-cocycle 
\( \sigma: G \to \LieA{g}^* \) of \( J \) satisfies 
\( \kappa\bigl(\sigma(g), \xi\bigr) = 0 \) for all \( g \in G \) 
and \( \xi \in \LieA{m} \).
The converse implication is immediate from 
\cref{eq:cartan:momentumMapNonEquivariance:cocycle}.
\end{proof}
Note that if \( \LieA{a} = \LieA{g}_\C \) with \( \kappa_{\LieA{a}} 
= \Im \kappa_\C \) as described above, then this is just saying that 
\( J \) has to be equivariant.
Thus, here we are harvesting the benefits of allowing \( \LieA{a} \) 
to be strictly smaller than \( \LieA{g}_\C \) in order to allow 
for non-equivariant momentum maps. In fact, as we will see below, 
this flexibility is needed in the example of K\"ahler geometry.

\subsection{Generalized Futaki invariant}
Let \( (M, \omega) \) be a symplectic manifold with a 
\( G \)-Hamiltonian action possessing a momentum map 
\( J: M \to \LieA{g}^* \) with respect to some non-degenerate
 pairing \( \kappa: \LieA{g}^* \times \LieA{g} \to \R \).
For every \( \zeta \in \LieA{g}_m \), the function
\begin{equation}
F_\zeta: G \ni g \mapsto \kappa\bigl(J(g \cdot m), \AdAction_g \zeta\bigr)
\in \R 
\end{equation}
is constant if \( J \) and \( \kappa \) are equivariant, namely, 
\(F_\zeta = \kappa(J(m), \zeta) \).
Moreover, \( F: \LieA{g}_m \ni \zeta \mapsto F_\zeta \in  
\R \) is a Lie algebra homomorphism since, for any 
\( \xi , \eta \in \LieA{g}_m\), we have 
\( F \bigl(\commutator{\xi}{\eta}\bigr) = 
- \omega_{m}(\xi \ldot m, \eta \ldot m) = 0 \) using equivariance
of \( J \). Thus, to every orbit \( G \cdot m \) we 
can assign the Lie algebra character 
\( F: \LieA{g}_m \to \R \). 
The generalized Futaki invariant is the extension of this character 
to a character on the stabilizer of a Cartan model for the orbit 
through \( m \).

To define it, choose a Cartan model \( (P, \theta) \) for the orbit 
through \( m \) using a map \( \chi: P \to M \) as in 
\cref{def:cartan:orbitModel}. Moreover, let 
\( \kappa_{\LieA{a}}: \LieA{g}^* \times \LieA{a} \to \R \) 
be a duality of the Klein pair \( (\LieA{a}, \LieA{g}) \) 
underlying \( P \) in the sense of \cref{defn:cartan:duality}.
For every \( \zeta \in \LieA{a}_m \), the value \( \theta_p(p \ldot \zeta) 
\in \LieA{a} \) depends equivariantly on \( p \in P \).
Thus, if \( J \) is equivariant, then the function
\begin{equation}
	\label{eq:futaki:generalizedFutakiInvariant}
	\tilde{F}_\zeta: P \to \R, \qquad p \mapsto 
	\kappa_{\LieA{a}}\Bigl(J\bigl(\chi(p)\bigr), 
	\theta_p(p \ldot \zeta)\Bigr)
\end{equation}
is \( G \)-invariant and hence descends to a function 
\( F_\zeta: B \to \R \).
We call \(F_\zeta \) the (generalized) \emphDef{Futaki invariant} 
associated with \( \zeta \in \LieA{a}_m \).
If \( J \) is not equivariant, then we 
have to assume that \( J \) is \( \LieA{a} \)-equivariant\footnotemark{}.
\footnotetext{Actually, it would be enough to require that the 
non-equivariance \( 1 \)-cocycle 
\( \sigma: G \to \LieA{g}^* \) vanishes when paired with 
elements of the form \( \theta_p (p \ldot \zeta) \in \LieA{a} \) 
with varying \( p \in P \).}

The classical Futaki invariant in K\"ahler geometry is constant 
as a function of the K\"ahler 
form in a given K\"ahler class. 
As we discuss now, the generalized Futaki invariant defined for a general 
Cartan model is also constant under natural assumptions.
This general result, in turn, sheds some light on what geometric 
properties lead to the fact that the classical Futaki invariant 
only depends on the K\"ahler class.
Note, in particular, that we do not need to assume that the Cartan 
bundle is a model for the \emph{complex} orbit.
In fact, the result does not even make any reference to complex geometry.
\begin{thm}
	\label{prop:cartan:constFutaki}
Let \( (M, \omega) \) be a symplectic manifold and let \( G \) be 
a Lie group acting symplectically on \( M \)  with momentum map 
\( J: M \to \LieA{g}^* \). Let 
\( (P, \theta) \) be a Cartan model for the orbit through 
\( m \in M \) and \( \kappa_{\LieA{a}} \) a duality of the underlying 
Klein pair \( (\LieA{a}, \LieA{g}) \).
Assume that \( J \) is \( \LieA{a} \)-equivariant and that \( \theta \) 
is torsion-free.
Choose \( \zeta \in \LieA{a}_m \). If \( \difLie_{\zeta^*} \theta \) 
takes values in \( \LieA{g} \), then 
the generalized Futaki invariant \( F_\zeta \) is locally constant,
\ie, \( F_\zeta \) is constant on the
connected components of \(B\).
\end{thm}

Conceptually, this theorem is a generalization of the basic fact 
that if a closed real-valued 
\( 1 \)-form \( \theta \) is invariant under an \( H \)-action, 
then for each \( \zeta \in \LieA{h} \) the function 
\( \theta(\zeta^*) \) is constant.
In our setting, the \( \LieA{a} \slash \LieA{g} \) part of the 
form \( \theta \) is invariant under the stabilizer \( \LieA{a}_m \), 
but \( \theta \) is not necessarily closed.
The resulting variation of \( \theta(\zeta^*) \) is compensated 
by the variation of the map \( J \circ \chi \).
See also \cref{Bourg}.

In the special case when \(P\) is the complexification \(G^c\) 
of a finite-dimensional compact Lie group \(G\) acting 
holomorphically on a finite-dimensional K\"ahler manifold 
(\cf \cref{ex:cartan:complexifiedGroup}), 
\cref{prop:cartan:constFutaki} recovers the result of 
\textcite[Proposition~6]{Wang2004}.

\begin{proof}
Before starting the proof we recall a formula
from Cartan geometry (see \parencite{Sharpe1997}). Let \((P,\theta)\) be a 
Cartan geometry modeled on the Klein pair \( (\LieA{a}, \LieA{g}) \) 
with group \( G \) on a manifold \( B \), \(F\) be a manifold, 
and \(\rho :G \times  F \to  F\) a smooth left action. Then
\(G\) acts freely on the left on  \( P \times F \) by
\(g \cdot (p,f)\defeq (g \cdot p, \rho(g)f)\), so we can form the
fiber bundle \(P \times_G F \to  B\) with typical fiber
\(F\). Let \( \nabla_\xi s \) denote the universal covariant derivative 
induced by \( \theta \) in the direction \( \xi \in \LieA{a} \) of a 
section \( s \) of  the fiber bundle \( P \times_G F \to B \); 
by definition,
\begin{equation}
\label{def_universal_covariant_derivative}
\nabla_\xi s (b) \defeq (\iota_p \circ \tangent_p \tilde{s})
\bigl(\theta^{-1}_p(\xi)\bigr),\quad \text{for all} \quad 
b \in B, \;\xi \in \LieA{a},
\end{equation}  
where \( p \in \pi^{-1}(b)\), 
\( \iota_p(f) = \equivClass{p,f} \), and \( \tilde{s}: P \to F \) 
is the \(G\)-equivariant map uniquely corresponding to \( s \), 
\ie, \(s(\pi(p)) = [p,\tilde{s}(p)]\).

We apply~\eqref{def_universal_covariant_derivative} in the following 
particular case: \( F = \LieA{g}^* \),  \( \rho(g) = 
\operatorname{Ad}^*_{g^{-1}} + \sigma(g) \), where 
\( \sigma:G \to \LieA{g}^* \) is the non-equivariance 
1-cocycle of the momentum map \( J:M \to \LieA{g}^* \). 
Then \( J \circ \chi:P \to  \LieA{g}^* \)
is \(G\)-equivariant (relative to \( \rho \)) so it plays the role of 
\(\tilde{s}\) in the previous considerations, \ie, it defines a
section \( s_J:B \to P \times_G \LieA{g}^* \) by
\(s_J(\pi(p))=[p,(J \circ \chi)(p)]\). Thus, letting \(b= \pi(p)\),
we get for every \( \xi \in \LieA{a} \)
\begin{equation}
\begin{split}
\nabla_\xi s_J (b)
= \left(\iota_p \circ \tangent_{\chi(p)} J \circ 
\tangent_p \chi\right) \bigl(\theta^{-1}_p(\xi)\bigr)
= \left(\iota_p \circ C_{p}\right) (\xi),
\end{split}\end{equation}
using~\eqref{eq:cartan:infAction} and the definition of the Calabi 
operator~\eqref{eq:cartan:calabiOperator}.

On the other hand, for \( \zeta \in \LieA{a}_m \), let \( G_\zeta \) 
be the section of \( P \times_G \LieA{a} \to  B\) corresponding 
to the \( G \)-equivariant map \( \tilde{G}_\zeta: P\ni p \mapsto 
\theta_p(p\ldot \zeta) \in\LieA{a} \),
\ie,  \( G_\zeta (\pi(p)) = [p, \tilde{G}_\zeta(p)] \). We apply
again~\eqref{def_universal_covariant_derivative} and get
\begin{equation}\begin{split}
\nabla_\xi G_\zeta (b)
&= \iota_p\left(\tangent_p \tilde{G}_\zeta\bigl(\theta_p^{-1}(\xi)\bigr)\right)
= \iota_p\left(\theta_p^{-1}(\xi)(\tilde{G}_\zeta)
\right) \\
&= \iota_p \bigl(\theta_p^{-1}(\xi)(\theta(\zeta^*))\bigr) \in  
P \times_G \LieA{a},
\end{split}\end{equation}
where \( \theta_p^{-1}(\xi)(\tilde{G}_\zeta) \) denotes the action 
of the derivation \( \theta_p^{-1}(\xi) \in \TBundle_p P \) on the 
\( \LieA{a} \)-valued function \( \tilde{G}_\zeta \).
Viewing \( p \mapsto \theta_p^{-1}(\xi) \) as a vector field on \( P \),
we have
\begin{equation}\begin{split}
\theta^{-1}(\xi)\bigl(\theta(\zeta^*)\bigr)
&= \dif \theta(\theta^{-1}(\xi), \zeta^*) +
\theta\bigl(\commutator{\theta^{-1}(\xi)}{\zeta^*}\bigr) + 
\zeta^*\bigl(\theta(\theta^{-1}(\xi))\bigr)\\
&= \Omega(\theta^{-1}(\xi), \zeta^*) + \commutator{\xi}{\theta(\zeta^*)} + 
\bigl(\difLie_{\zeta^*} \theta\bigr) (\theta^{-1}(\xi)) \in  \LieA{a} .
\end{split}\end{equation}
Thus, under the standing assumption that the torsion vanishes (\ie, 
the curvature \( \Omega \defeq \dif \theta - \frac{1}{2} 
\wedgeLie{\theta}{\theta} \in \DiffFormSpace^2(P, \LieA{a}) \)
takes values in \(\LieA{g}\)) and that 
\( \difLie_{\zeta^*} \theta \) takes values in \( \LieA{g} \), we find
\begin{equation}
	\label{eq:cartan:constFutaki:covszeta}
\nabla_\xi  G_\zeta (b) \mod \LieA{g} 
= \iota_p \commutator*{\xi}{\theta_p(p \ldot \zeta)}.
\end{equation}
Introducing \( \tilde{\kappa}_{\LieA{a}} : (P\times_G \LieA{g}^*)
\times_M (P\times_G \LieA{a}) \ni ([p, \mu],[p, \eta])
\mapsto \kappa_{\LieA{a}} (\mu, \eta) \in  \R \), the 
definition~\eqref{eq:futaki:generalizedFutakiInvariant} 
of \( \tilde{F}_\zeta \) takes the form
\begin{equation}
	\tilde{F}_\zeta = \kappa_{\LieA{a}} (J \circ \chi, 
	\tilde{G}_\zeta )\; \;\text{ and hence }\; \;
	F_\zeta = \tilde{\kappa}_{\LieA{a}} (s_J, G_\zeta).
\end{equation}
Using the non-equivariance 
property~\eqref{eq:cartan:momentumMapNonEquivariance:calabi} of 
\( J \), we get for \(b= \pi(p)\),
\begin{equation}\begin{split}
\nabla_\xi F_\zeta(b)
&= \tilde{\kappa}_{\LieA{a}} \Bigl(\nabla_\xi s_J(b), G_\zeta(b) \Bigr) 
+ \tilde{\kappa}_{\LieA{a}} \Bigl(s_J(b), 
\nabla_\xi G_\zeta(b) \Bigr) \\
&= \kappa_{\LieA{a}} \Bigl(C_{p} \xi, 
\theta_p(p \ldot \zeta) \Bigr) + 
\kappa_{\LieA{a}} \Bigl(J\bigl(\chi(p)\bigr), 
\commutator{\xi}{\theta_p(p \ldot \zeta)} \Bigr) \\
&= \kappa_{\LieA{a}} \Bigl(C_p \theta_p(p \ldot \zeta), \xi \Bigr).
\end{split}\end{equation}
On the other hand, the definition of \( C_p \) implies 
\( C_p \theta_p(p \ldot \zeta) = 
\tangent_{\chi(p)} J \bigl( \tangent_p \chi(p \ldot \zeta) \bigr) \).
Since \( \zeta \in \LieA{a}_m \), we have 
\( \tangent_p \chi(p \ldot \zeta) = 0 \) and, hence, 
\( \nabla_\xi F_\zeta = 0 \) for all \( \xi \in \LieA{a} \).

Since \( \theta \) is an isomorphism, \( \nabla_\xi F_\zeta = 0\) for 
all \( \xi \) implies that \( \tangent \tilde{F}_\zeta = 0 \) 
by~\eqref{def_universal_covariant_derivative}. But \( \tilde{F}_\zeta \) is 
also \(G\)-invariant, so \( F_\zeta:M \to  \mathbb{R} \) is (locally) 
constant. 
\end{proof}

Thus, for connected \( B \), the generalized Futaki invariant 
yields a map \( F: \zeta \ni\LieA{a}_m \mapsto F_\zeta \in  \R \), 
for every \( m \in M \) if \( \difLie_{\zeta^*} \theta \) takes 
values in \( \LieA{g} \) for all \( \zeta \in \LieA{a}_m \).
\begin{prop}
	\label{prop:cartan:futakiIsCharacter}
In the setting of \cref{prop:cartan:constFutaki}, assume, in addition, 
that \( B \) is connected and that \( \difLie_{\zeta^*} \theta \) takes 
values in \( \LieA{g} \) for all \( \zeta \in \LieA{a}_m \).
Then the map \( F: \LieA{a}_m \to \R, \zeta \mapsto F_\zeta \) 
is a Lie algebra character.
\end{prop}
\begin{proof}
Let \( \zeta, \eta \in \LieA{a}_m \). Using
\begin{equation}
	\dif \theta \bigl(\zeta^*, \eta^*\bigr) = 
	\difLie_{\zeta^*} \bigl(\theta(\eta^*)\bigr) - 
	\difLie_{\eta^*} \bigl(\theta(\zeta^*)\bigr) - 
	\theta\bigl(\commutator{\zeta^*}{\eta^*}\bigr)
\end{equation}
and
\begin{equation}
	\difLie_{\zeta^*} \bigl(\theta(\eta^*)\bigr) = 
	\bigl(\difLie_{\zeta^*} \theta\bigr)(\eta^*) + 
	\theta\bigl(\commutator{\zeta^*}{\eta^*}\bigr)
\end{equation}
we find
\begin{equation}\begin{split}
	\Omega(\zeta^*, \eta^*)
	&= \dif \theta \bigl(\zeta^*, \eta^*\bigr) - 
	\commutator*{\theta(\zeta^*)}{\theta(\eta^*)}\\
	&= \bigl(\difLie_{\zeta^*} \theta\bigr)(\eta^*) - 
	\bigl(\difLie_{\eta^*} \theta\bigr)(\zeta^*) + 
	\theta\bigl(\commutator{\zeta^*}{\eta^*}\bigr) - 
	\commutator*{\theta(\zeta^*)}{\theta(\eta^*)},
\end{split}\end{equation}
where \( \Omega \) is the curvature of \( \theta \).
Since \( \Omega \) and \( \difLie_{\zeta^*} \theta \) take 
values in \( \LieA{g} \), we conclude
\begin{equation}
	\theta\bigl(\commutator*{\zeta^*}{\eta^*}\bigr) \mod \LieA{g} = 
	\commutator*{\theta(\zeta^*)}{\theta(\eta^*)} \mod \LieA{g}.
\end{equation}
So, for every \( p \in P \), we have from~\eqref{eq:futaki:generalizedFutakiInvariant},
\begin{equation}\begin{split}
	\tilde{F}_{\commutator{\zeta}{\eta}}(p)
	&= \kappa_{\LieA{a}}\Bigl(J(\chi(p)), \theta_p\bigl(p \ldot 
	\commutator{\zeta}{\eta}\bigr)\Bigr)
			\\
	&= \kappa_{\LieA{a}}\Bigl(J(\chi(p)), \theta_p
	\bigl(\commutator{\zeta^*}{\eta^*}_p\bigr)\Bigr)
			\\
	&= \kappa_{\LieA{a}}\Bigl(J(\chi(p)), 
	\commutator*{\theta_p(p \ldot \zeta)}{\theta_p(p \ldot \eta)}\Bigr)
			\\
	&= - \kappa_{\LieA{a}}\Bigl(C_p \theta_p(p \ldot \zeta), 
	\theta_p(p \ldot \eta)\Bigr) + 
	\kappa_{\LieA{a}}\Bigl(C_p \theta_p(p \ldot \eta), 
	\theta_p(p \ldot \zeta)\Bigr),
	\end{split}\end{equation}
where the last equality uses the \( \LieA{a} \)-equivariance 
of \( J \), \ie, \cref{eq:cartan:momentumMapNonEquivariance:calabi}.
Now, as in the proof of \cref{prop:cartan:constFutaki}, we conclude 
that \( C_p \theta_p(p \ldot \zeta) \) and 
\( C_p \theta_p(p \ldot \eta) \) vanish, because 
\( \zeta, \eta \in \LieA{a}_m \).
Hence, \( F_{\commutator{\zeta}{\eta}} = 0 \) as claimed.
\end{proof}
\begin{remark}
Note that we only required \( (P, \theta) \) to be an orbit of \( M \), and not necessarily a \emph{complex} orbit, \ie, we did not make use of the 
property \( \img \tangent_p \chi = \LieA{g}_\C \ldot \chi(p) \).
\end{remark}
\begin{remark}
	With natural modifications, the above construction can be carried out in the slightly more general setting, where the structure group of the Cartan bundle is no longer equal to the group that acts on the symplectic manifold \( M \).
	For example, one may allow for a Cartan \( G \)-bundle with a Lie group homomorphism \( \varPhi: G \to U \), where \( U \) acts symplectically on \( M \).
	Such an extension is necessary to cover the example of real GIT discussed in \cref{ex:cartan:realGIT}.
	Another advantage of such a generalization is that one would obtain constructions that are functorial with respect to morphisms of Cartan geometries.
	We leave the details to the interested reader.
\end{remark}

\subsection{Kempf--Ness function}
\label{sec:kempfNess}

We start by recalling the standard definition of the Kempf--Ness function.
Let \( G \) be a compact Lie group with complexification \( G^c \).
Assume that \( G^c \) acts linearly on a finite-dimensional complex 
vector space \( V \) endowed with a \( G \)-invariant Hermitian form 
\( V \times V \to \C \). Let \( M \) be a smooth \( G^c \)-invariant 
submanifold of the projective space \( \projectiveSpace V \).
Then the induced \( G \)-action on \( M \) is Hamiltonian with 
momentum map \( J: M \to \LieA{g}^* \).
For every \( \equivClass{v} \in M \), the lifted 
\emphDef{Kempf--Ness function} \( \Psi_{\equivClass{v}}: G^c \to \R \) 
is defined by
\begin{equation}
	\Psi_{\equivClass{v}}(a) \defeq \frac{1}{2} \log \norm{a \cdot v}^2 
	- \frac{1}{2} \log \norm{v}^2.
\end{equation}
A direct calculation shows that 
\( \tangent_a \Psi_{\equivClass{v}} (\zeta \ldot a) 
= \Im \kappa_\C \bigl(J(a \cdot m), \zeta\bigr) \) 
for all \( a \in G^c \) and \( \zeta \in \LieA{g}_\C \), 
where \( J \) is the momentum map for the Fubini--Study symplectic 
structure (extended complex-linearly to \( \LieA{g}_\C \)), 
see \parencite[Lemma~8.3]{GeorgoulasRobbinEtAl2018}.
In fact, this relation uniquely characterizes the Kempf--Ness 
function up to a constant (if \( G^c \) is connected), and is 
used to define the Kempf--Ness function for more general K\"ahler 
manifolds \( M \) endowed with a \( G^c \)-action.
We can employ the same idea to define a Kempf--Ness function 
for Cartan geometries.
\begin{thm}
	\label{prop:cartan:kempfNess}
Let \( (M, \omega) \) be a symplectic manifold and let \( G \) be a 
Lie group acting symplectically on \( M \)  with momentum map 
\( J: M \to \LieA{g}^* \). Let \( (P, \theta) \) be a Cartan model 
for an orbit through \( m \in M \) via the map \( \chi: P \to M \), 
with Klein pair \( (\LieA{a}, \LieA{g}) \) and duality \( \kappa_\LieA{a} \).
Assume that \( J \) is \( \LieA{a} \)-equivariant.
Then the following holds:
\begin{thmenumerate}
\item
The \( 1 \)-form \( \alpha \) on \( P \) defined by
\begin{equation}
	\label{eq:cartan:kempfNess:alpha}
\alpha_p(X) = \kappa_\LieA{a} \bigl(J(\chi(p)), \theta_p(X)\bigr), 
\qquad p \in P, X \in \TBundle_p P,
\end{equation}
is basic and hence induces a \( 1 \)-form \( \check{\alpha} \) on \( B \).
\item
We have
\begin{equation}
	\dif \alpha = - \kappa_\LieA{a} \bigl(J \circ \chi, \Omega \bigr),
\end{equation}
where \( \Omega \defeq \dif \theta - \frac{1}{2} 
\wedgeLie{\theta}{\theta} \in \DiffFormSpace^2(P, \LieA{a}) \) 
is the curvature of \( \theta \).
In particular, if \( \theta \) is torsion-free, then 
\( \dif \alpha = 0 = \dif \check{\alpha} \).
\item
\label{i:cartan:kempfNess:existence}
If \( B \) is simply-connected and \( \theta \) is torsion-free, 
then there exists a function \( \Psi_m: B \to \R \) such that 
\( \check{\alpha} = \dif \Psi_m \).
			\qedhere
\end{thmenumerate}
\end{thm}
The function \( \Psi_m \) is unique up to a constant and we call 
it the \emphDef{Kempf--Ness function} at \( m \).
\begin{proof}
Let \( \alpha \) be the \( 1 \)-form defined above.
Since the non-equivariance cocycle of \( J \) vanishes when 
paired with \( \LieA{a} \), the \( G \)-invariance of \( \alpha \) 
follows from the \( G \)-equivariance of \( \theta \) 
and \( \chi \). Moreover, for all \( \xi \in \LieA{g} \), 
we have \( \alpha_p (\xi \ldot p) = 
\kappa_{\LieA{a}} \bigl(J(\chi(p)), \theta_p (\xi \ldot p)\bigr) 
= 0 \) since \( \theta_p (\xi \ldot p) = \xi \) and 
\( \kappa_{\LieA{a}} \) vanishes on \( \LieA{g} \).
So \( \alpha \) is basic and thus descends to a 
\( 1 \)-form \( \check{\alpha} \) on \( B \).

For the second part, abbreviate \( s \defeq J \circ \chi: 
P \to \LieA{g}^* \). Then \( \alpha_p(X) = 
\kappa_{\LieA{a}}\bigl(s(p), \theta_p(X)\bigr) \), 
and it thus suffices to calculate \( \dif s \) and 
\( \dif \theta \). For the first, we have
\begin{equation}\begin{split}
\tangent_p s \bigl(\theta^{-1}_p (\xi)\bigr)
&= \tangent_{\chi(p)} J \bigl(\tangent_p \chi (\theta^{-1}_p (\xi))\bigr)
= \tangent_{\chi(p)} J \bigl(\xi \ldot \chi(p)\bigr) = C_p \xi
\end{split}\end{equation}
for every \( \xi \in \LieA{a} \). For the second, 
\( \dif \theta = \Omega + \frac{1}{2} \wedgeLie{\theta}{\theta} \), 
where \( \Omega \) is the curvature of \( \theta \).	
Let \( X_1, X_2 \in \TBundle_p P \) and choose 
\( \xi_1, \xi_2 \in \LieA{a} \) such that 
\( \theta_p(X_i) = \xi_i \) for \( i = 1, 2 \). Then
\begin{equation}
\dif \theta (X_1, X_2)
= \Omega (X_1, X_2) + \commutator{\xi_1}{\xi_2}.
\end{equation}
Combining these two calculations, we find
\begin{equation}\begin{split}
(\dif &\alpha)_p (X_1, X_2) - 
\kappa_{\LieA{a}} \bigl(J(\chi(p)), \Omega_p (X_1, X_2)\bigr) \\
&= \kappa_{\LieA{a}} (\dif s \wedge \theta)_p (X_1, X_2) 
+ \kappa_{\LieA{a}}\bigl(s(p), (\dif \theta - 
\Omega)_p(X_1, X_2)\bigr) \\
&= \kappa_{\LieA{a}} \bigl(C_p \xi_1, \xi_2\bigr) - 
\kappa_{\LieA{a}} \bigl(C_p \xi_2, \xi_1\bigr) + 
\kappa_{\LieA{a}} \bigl(J(\chi(p)), 
\commutator{\xi_1}{\xi_2}\bigr) \\
&= 0 
\end{split}\end{equation}
due to the \( \LieA{a} \)-equivariance of \( J \).
\end{proof}
\begin{prop}
\label{prop:cartan:kempfNess:positive}
In the setting of \cref{i:cartan:kempfNess:existence}, the following holds:
\begin{thmenumerate}
\item
A point \( b \in B \) is a critical point of \( \Psi_m \) if and 
only if for some (and hence for all) \( p \in \pi^{-1}(b) \) the 
functional \( J(\chi(p)) \) vanishes when paired with elements 
of \( \LieA{a} \).
\item
For every \( \xi \in \LieA{a} \), let \( \gamma_\xi: [0, \delta] \to P \) 
be the integral curve of the vector field \( \theta^{-1}(\xi) \).
Assume that \( \pi\bigl(\gamma_\xi(0) \bigr) \) is a critical point 
of \( \Psi_m \). Then
\begin{equation}
\difFracAt[2]{}{t}{t=0} 
\Psi_m \Bigl(\pi\bigl(\gamma_\xi(t)\bigr)\Bigr)
= \kappa_{\LieA{a}} \bigl(C_p \xi, \xi\bigr).
\end{equation}
\item
\label{prop:cartan:kempfNessConvexity}
In addition, assume that we are in the setting of \cref{prop:cartan:momentumMapEquivarianceComplexCase} so that
we consider an orbit through a point \(m\) in an almost complex manifold \( (M,j) \) and 
that \( \commutator{\LieA{g}}{\I\LieA{m}} \subseteq \I \LieA{m} \). 
We also assume that 
the almost complex structure \( j \) is compatible with the symplectic form \( \omega \), \ie, \( g(\cdot, \cdot) = \omega(\cdot, j \cdot) \) is a Riemannian metric.
Then:
\begin{equation}\label{eq:cartan:kempfNess:positive}
\difFracAt[2]{}{t}{t=0} 
\Psi_m \Bigl(\pi\bigl(\gamma_\xi(t)\bigr)\Bigr)
= - \kappa \bigl(C_p (\I \Im \xi), \Im \xi\bigr)
= \norm*{(\Im \xi) \ldot \chi(p)}^2,
\end{equation}
where the norm is taken with respect to \( g \).
		
Moreover, if \( \xi \in \I \LieA{m} \), then \( \Psi_m \) is convex 
along \( \pi \circ \gamma_\xi \):
\begin{equation}
\label{eq:cartan:kempfNess:convexityAlongGeodesics}
\difFracAt[2]{}{t}{t} 
\Psi_m \Bigl(\pi\bigl(\gamma_\xi(t)\bigr)\Bigr) 
= \norm*{\xi \ldot \chi\bigl(\gamma_\xi(t)\bigr)}^2.
				\qedhere
\end{equation}
\end{thmenumerate}
\end{prop}
In the classical case \( P = G^c \to B \defeq G^c \slash G \), every 
tangent vector at the identity coset \( \equivClass{e} \) is 
identified with an element of \( \I \LieA{g} \).
So the condition that \( \xi \) has to be imaginary in the last 
statement is no real restriction and the proposition entails 
that \( \Psi_m \) is convex along the geodesics in 
\( G^c \slash G \) in arbitrary directions.
In \cref{sec:cartan:geodesics}, we will discuss the geodesics 
in the base of a Cartan geometry in more detail and will see 
that the Kempf--Ness function is convex along geodesics.
\begin{proof}
The first part follows from the fact that 
\( \dif \Psi_m = \check{\alpha} \) and from the 
formula~\eqref{eq:cartan:kempfNess:alpha} 
for \( \alpha \).

For the second part, let \( \xi \in \LieA{a} \) and let 
\( \gamma_\xi \) be an integral curve of \( \theta^{-1}(\xi) \).
Then
\begin{equation}\label{eq:cartan:kempfNess:firstDerivativeGeneral}
\begin{split}
\difFrac{}{t} \Psi_m \Bigl(\pi\bigl(\gamma_\xi(t)\bigr)\Bigr)
&= \tangent_{\pi(\gamma_\xi(t))} \Psi_m \bigl(
\tangent_{\gamma_\xi(t)} \pi (\dot{\gamma}_\xi(t))\bigr)\\
&= \alpha_{\gamma_\xi(t)} \bigl(\dot{\gamma}_\xi(t)\bigr)\\
&= \kappa_{\LieA{a}} \Bigl(J\bigl(\chi(\gamma_\xi(t))\bigr), 
\theta_{\gamma_\xi(t)} \bigl(\dot{\gamma}_\xi(t)\bigr)\Bigr)\\
&= \kappa_{\LieA{a}} \bigl(J(\chi(\gamma_\xi(t))), \xi\bigr).
\end{split}\end{equation}
Using the definition~\eqref{eq:cartan:calabiOperator} of the 
Calabi operator, we conclude
\begin{equation}
\label{eq:cartan:kempfNess:secondDerivativeGeneral}
\difFracAt[2]{}{t}{t} 
\Psi_m \Bigl(\pi\bigl(\gamma_\xi(t)\bigr)\Bigr)
= \kappa_{\LieA{a}} \bigl(C_{\gamma_\xi(t)} \xi, \xi\bigr).
\end{equation}
For the last part, assume that \( \gamma_\xi(0) = p \) projects 
onto a critical point of \( \Psi_m \) and that we are in the 
setting of \cref{prop:cartan:momentumMapEquivarianceComplexCase}.
Decomposing \( \xi = \xi_1 + \I \xi_2 \) with \( \xi_1 \in \LieA{g} \) 
and \( \xi_2 \in \LieA{m} \), we have
\begin{equation}
\kappa_{\LieA{a}} \bigl(C_p \xi, \xi\bigr)
= - \kappa\bigl(C_p \xi_1, \xi_2\bigr) - \kappa\bigl(C_p (\I \xi_2), \xi_2\bigr)
\end{equation}
Then, using the first part of the proposition, \( J(\chi(p)) \) 
vanishes when paired with \( \LieA{m} \) under \( \kappa \).
This implies that the first summand vanishes since 
\( C_p \xi_1 = - \CoadAction_{\xi_1} J(\chi(p)) \) and the 
commutator \( \commutator{\xi_1}{\xi_2} \) lies in 
\( \LieA{m} \) by assumption.
As we have seen in~\eqref{eq:cartan:momentumMapNonEquivariance:CpOnI},
\( \kappa\bigl( C_p (\I \xi_2), \xi_2\bigr) = 
- \omega\bigl(\xi_2 \ldot \chi(p), j \xi_2 \ldot \chi(p)\bigr)
= -\norm*{\xi_2 \ldot \chi(p)}^2 \), which proves the second equality 
in~\eqref{eq:cartan:kempfNess:positive}, \cf also 
\parencite[Eq.~3.32]{DiezRatiuSymplecticConnections}.
Finally,~\eqref{eq:cartan:kempfNess:convexityAlongGeodesics} 
follows directly from~\eqref{eq:cartan:kempfNess:secondDerivativeGeneral} 
by similar arguments.
\end{proof}
\begin{remark}\label{Bourg}
If the closed \( 1 \)-form \( \alpha \) defined \cref{prop:cartan:kempfNess}, 
satisfies \(\difLie_{\zeta^*} \alpha = 0\) for all 
\(\zeta \in \LieA{a}_m\) then one directly obtains a different 
proof of \cref{prop:cartan:constFutaki}.
This line of thought goes back at least to the work of Bourguignon. 
The invariance of \( \alpha \) holds in the K\"ahler examples 
due to the properties of the curvature.
However, in our abstract setting, 
\(\difLie_{\zeta^*} \alpha = 0\) follows instead from  
\cref{prop:cartan:constFutaki,prop:cartan:kempfNess}.
\end{remark}
\begin{prop}\label{auto}
In the setting of \cref{prop:cartan:momentumMapEquivarianceComplexCase}, 
assume that the Lie algebra action of \( \LieA{a}_m \) is integrated 
to an action of a regular\footnotemark{} Lie group \( Z \) on 
\( P \) and that \( \ker \tangent_p \chi = p \ldot \LieA{a}_m \).
\footnotetext{Roughly speaking, a Lie group \( G \) is called 
\emphDef{regular} if every curve in its Lie algebra integrates 
to an evolution curve in \( G \). We refer to 
\parencite[Section~38]{KrieglMichor1997} 
or \parencite[Section~III]{Neeb2006} for a precise definition.}
	Let \( \gamma_\xi: [0, 1] \to P \) be an integral curve of \( \theta^{-1}(\xi) \) for \( \xi \in \I \LieA{m} \).
	If \( \gamma_\xi(0) \) and \( \gamma_\xi(1) \) are zeros of \( J \circ \chi \), then there exists \( z \in Z \) such that \( \gamma_\xi(0) = \gamma_\xi(1) \cdot z \). 
\end{prop}
\begin{proof}
Since \( J \circ \chi \) vanishes at the end points of 
\( \gamma_\xi \), \cref{eq:cartan:kempfNess:firstDerivativeGeneral} 
implies that \( \difFracAt{}{t}{t=0} \Psi_m \Bigl(
\pi\bigl(\gamma_\xi(t)\bigr)\Bigr) = 0 = 
\difFracAt{}{t}{t=1} \Psi_m \Bigl(\pi\bigl(
\gamma_\xi(t)\bigr)\Bigr) \).
Moreover, \( \Psi_m \) is convex along \( \pi \circ \gamma_\xi \) 
by \cref{prop:cartan:kempfNessConvexity}.
But a convex function with vanishing derivatives at the end points 
has vanishing derivative everywhere.
Hence, using~\eqref{eq:cartan:kempfNess:convexityAlongGeodesics}, 
we get \( \xi \ldot \chi\bigl(\gamma_\xi(t) \bigr) = 0 \) for 
all \( t \in [0, 1] \). Equivalently,
\begin{equationcd}
0 = \xi \ldot \chi\bigl(\gamma_\xi(t)\bigr)
= \tangent_{\gamma_\xi(t)} \chi \Bigl(\theta^{-1}_{\gamma_\xi(t)} (\xi)\Bigr).
\end{equationcd}
By assumption, there exists a curve \( \zeta(t) \in \LieA{a}_m \) 
with \( \gamma_\xi(t) \ldot \zeta(t) = 
\theta^{-1}_{\gamma_\xi(t)} (\xi) \).
Since \( Z \) is a regular Lie group, we can find a curve 
\( z(t) \in Z \) starting at the identity with 
\( \zeta(t) \ldot z(t) = -\dot{z}(t) \). Then
\begin{equation}\begin{split}
\difFracAt{}{t}{t}  
\bigl( \gamma_\xi(t) \cdot z(t) \bigr)
&= \dot{\gamma}_\xi(t) \ldot z(t) + 
\gamma_\xi(t) \ldot \dot{z}(t)\\
&= \theta^{-1}_{\gamma_\xi(t)} (\xi) \ldot z(t) - 
\gamma_\xi(t) \ldot \zeta(t) \ldot z(t)\\
&= \theta^{-1}_{\gamma_\xi(t)} (\xi) \ldot z(t) - 
\theta^{-1}_{\gamma_\xi(t)} (\xi) \ldot z(t)\\
&= 0.
\end{split}\end{equation}
Hence \( \gamma_\xi(0) = \gamma_\xi(1) \cdot z(1) \).
\end{proof}

\subsection{Geodesics}
\label{sec:cartan:geodesics}

A Klein pair \( (\LieA{a}, \LieA{g}) \) with group \( G \) is called 
\emphDef{reductive} if there exists a \( G \)-invariant topological
complement \( \LieA{m} \) of \( \LieA{g} \) in \( \LieA{a} \), \ie, 
\( \LieA{a} = \LieA{g} \oplus \LieA{m} \) as a topological vector space 
and \( \AdAction_g \LieA{m} \subseteq \LieA{m} \) for all \( g \in G \).
In particular, \( \commutator{\LieA{g}}{\LieA{m}} \subseteq \LieA{m} \).
For a Cartan geometry \( (P, \theta) \) modeled on a reductive Klein 
pair \( (\LieA{a}, \LieA{g}) \), the Cartan connection \( \theta \) 
decomposes as \( \theta = \theta_\LieA{g} + \theta_\LieA{m} \) 
where \( \theta_\LieA{g} \) and \( \theta_\LieA{m} \) take values 
in \( \LieA{g} \) and \( \LieA{m} \), respectively.
One easily checks that \( \theta_\LieA{g} \) is an ordinary 
principal connection on the \( G \)-bundle \( P \) and 
\( \theta_\LieA{m} \) is a \( G \)-equivariant and horizontal 
\( 1 \)-form on \( P \) (with values in \( \LieA{m} \)).
Thus, we may regard \( \theta_\LieA{m} \) also as a \( 1 \)-form 
on the base manifold \( B \) of \( P \) with values in 
\( P \times_G \LieA{m} \). Moreover, the bundle isomorphism 
\( \TBundle B \isomorph P \times_G (\LieA{a} \slash \LieA{g}) \) 
mentioned above takes the simple form 
\begin{equation}
\label{eq:cartan:reductiveBundleIso}
\TBundle B \ni (b, Z) \mapsto \equivClass*{p, \theta_\LieA{m}(\hat{Z}_p)} 
\in P \times_G \LieA{m},
\end{equation}
where \( p \in \pi^{-1}(b) \) and \( \hat{Z}_p \in \TBundle_P P \) 
is any lift of \( Z \in \TBundle_b B \). Note that this is just 
\( \theta_\LieA{m} \in \DiffFormSpace^1(B, P \times_G \LieA{m}) \) 
in disguise. Due to this property, \( \theta_\LieA{m} \) is called 
the \emphDef{solder form}.

If \( B \) is finite-dimensional (equivalently, \( \LieA{m} \) 
is finite-dimensional) and the representation 
\( \AdAction: G \to \GLGroup(\LieA{m}) \) is injective, 
then we can identify \( P \) with the frame bundle of \( B \) 
and \( \theta_\LieA{m} \) with the tautological form, see 
\parencite[Appendix~A.2]{Sharpe1997}.
In fact, a choice of basis of \( \LieA{m} \) yields a 
correspondence between local sections of \( P \) and local 
frames of \( B \) via~\eqref{eq:cartan:reductiveBundleIso}.
Without the injectivity assumption on the representation, 
multiple sections of \( P \) may correspond to the same frame 
of \( B \). It is not obvious how to generalize this identification 
to the infinite-dimensional setting, and, in fact, even the notion 
of a frame bundle is not clear in this setting.
Changing the perspective, we may however use this identification 
as the raison d'\^etre for the following definition.
\begin{defn}
Let \( B \) be an infinite-dimensional manifold and \( G \) 
an infinite-dimensional Lie group.
A \emphDef{\( G \)-structure} on \( B \) is a Cartan geometry 
\( P \to B \) modeled on a reductive Klein pair with group \( G \).
\end{defn}
Many classical concepts from the finite-dimensional theory 
of \( G \)-structures carry over to this setting when reformulated 
in terms of Cartan geometries.
In particular, the principal connection \( \theta_{\LieA{g}} \) 
induces a connection on every vector bundle associated with 
\( P \) (in infinite dimensions, it is perhaps best to formulate 
this in terms of the so-called connector, see 
\parencite[Section~37.26]{KrieglMichor1997}).
In particular, we obtain an affine connection on the tangent 
bundle \( \TBundle B \) via the bundle 
isomorphism~\eqref{eq:cartan:reductiveBundleIso}, and thus 
it makes sense to talk about geodesics in \( B \).
Explicitly, the \emphDef{covariant derivative} 
\( \nabla_{\dot{\gamma}} X \) of a vector field \( X \) 
along a curve \( \gamma: I \to B \) is defined by
\begin{equation}
\label{eq:cartan:covariantDerivative}
\theta_{\LieA{m}} \Bigl( \nabla_{\dot{\gamma}} X (t) \Bigr) 
= \iota_p \Bigl( \tangent_p \tilde{X} (\hat{\gamma}_p) \Bigr), 
\end{equation}
where \( p \in \pi^{-1}(\gamma(t)) \) and \( \hat{\gamma}_p \in 
\TBundle_p P \) is the horizontal lift of 
\( \dot{\gamma}(t) \in \TBundle_{\gamma(t)} B \).
Moreover, \( \tilde{X}: \gamma^* P \to \LieA{m} \) satisfying 
\( \equivClass*{p, \tilde{X}(t, p)} = \theta_{\LieA{m}} (X(t)) \) 
is the \( G \)-equivariant map corresponding to \( X \) under 
the isomorphism~\eqref{eq:cartan:reductiveBundleIso}.
We call \( \gamma \) a \emphDef{geodesic} if 
\( \nabla_{\dot{\gamma}} \dot{\gamma} = 0 \).
As in the finite-dimensional case, we have a close relationship 
between geodesics and the flows of the \( \theta \)-constant 
vector fields on \( P \).
\begin{prop}
\label{prop:cartan:geodesicsAsProjections}
Let \( \pi: P \to B \) be a \( G \)-structure on \( B \).
Assume that, for every \( \xi \in \LieA{m} \), the horizontal 
vector field \( \theta^{-1}(\xi) \) on \( P \) has a smooth local 
flow \( \flow^\xi_t \).
Then, for every \( \equivClass{p, \xi} \in P \times_G \LieA{m} \), 
the curve \( t \mapsto \pi \bigl(\flow^\xi_t(p)\bigr) \) is a geodesic 
in \( B \). Conversely, if \( G \) is regular, then every geodesic 
in \( B \) is of this form for a unique 
\( \equivClass{p, \xi} \in P \times_G \LieA{m} \).
\end{prop}
\begin{proof}
The proof is analogous to the finite-dimensional case, see 
\parencite[Proposition~III.6.3]{KobayashiNomizu1963} or 
\parencite[Proposition~2.1.22]{RudolphSchmidt2014}.

For \( \xi \in \LieA{m} \) and \( p \in P \), let 
\( \hat{\gamma}(t) = \flow^\xi_t(p) \) and \( \gamma(t) 
= \pi \bigl(\hat{\gamma}(t)\bigr) \) for \( t \in I \subseteq \R \) 
an interval containing \( 0 \). Note that \( \gamma \) depends 
only on the equivalence class \( \equivClass{p, \xi} \) and not 
on the choice of \( p \) and \( \xi \).
Moreover, \( \hat{\gamma} \) is a horizontal lift of \( \gamma \)
with \( \difFrac{}{t} \hat{\gamma}(t) = 
\theta^{-1}_{\hat{\gamma}(t)}(\xi) \).
Hence, the map \( \gamma^* P \to \LieA{m} \) corresponding to 
\( \dot{\gamma} \) is constant along horizontal curves.
Thus, by~\eqref{eq:cartan:covariantDerivative}, 
\( \nabla_{\dot{\gamma}} \dot{\gamma} = 0 \), so 
\( \gamma \) is a geodesic.

Conversely, let \( \gamma: I \to B \) be a geodesic defined 
on an interval \( I \subseteq \R \) containing \( 0 \).
Since \( G \) is regular, for every \( p \in \pi^{-1}(\gamma(0)) \), 
there exists a unique horizontal lift \( \hat{\gamma}: I \to P \) 
of \( \gamma \) with \( \hat{\gamma}(0) = p \), see 
\parencite[Theorem~39.1]{KrieglMichor1997}.
Since \( \gamma \) is a geodesic, we have 
\( \theta_{\LieA{m}}(\dot{\gamma}(t)) = 
\equivClass{\hat{\gamma}(t), \xi} \) for some constant 
\( \xi \in \LieA{m} \). Thus, \( \hat{\gamma} \) is an 
integral curve of \( \theta^{-1}(\xi) \).
\end{proof}

If \( P \to B \) is a Cartan geometry, not necessarily modeled 
on a reductive Klein pair, then we can still consider the 
projection of integral curves \( \gamma_\xi \) of the vector fields 
\( \theta^{-1}(\xi) \) for \( \xi \in \LieA{a} \).
Such curves are called \emphDef{generalized geodesics} 
or generalized circles, see \parencite[Definition~5.4.16]{Sharpe1997}.
Reformulated in this terminology, \cref{prop:cartan:kempfNess} 
expresses that the Kempf--Ness function is convex along generalized 
geodesics.

For a (generalized) geodesic ray \( \gamma_\xi \), we define 
its slope by
\begin{equation}
	\lim_{t \to \infty} \frac{\Psi_m(\pi(\gamma_\xi(t)))}{t}.
\end{equation}
In the K\"ahler setting, this definition appears in 
\Textcite[page 32, Conjecture/Question 12, (2)]{Donaldson1999a}. 
In view of the work of \Textcite{PhongSturm2007}, a test 
configuration gives rise to a geodesic and
the slope for this geodesic coincides with the Donaldson-Futaki 
invariant together with the Lelong
number of the central fiber. When the central fiber is reduced 
the Lelong number is $0$. So we may regard the slope as the 
``Mumford weight''. Thus, instead of ``slope'', one could call
it ``weight'', or ``Mumford weight'', or ``GIT weight''.

\begin{defn}
	\label{defn:cartan:stability}
A Cartan model for the complex orbit through \( m \in M \) 
is said to be stable (resp. semistable) if the slope of any 
geodesic ray is positive (resp. non-negative). A Cartan model for 
the complex orbit through \( m \in M \) is said to be 
unstable if it is not semistable.
\end{defn}

Since geodesics cannot be compactified, it is not straightforward 
how to define test configurations in our general framework. 
Instead, we might consider Tits buildings. A more comprehensive 
exploration of this idea is left for future work.
\smallskip

In many cases, the base manifold \( B \) of a Cartan geometry carries a 
natural Riemannian metric, so it is natural to ask whether the 
geodesics of this metric are the same as the Cartan geodesics.
Let \( (P, \theta) \) be a Cartan geometry modeled on a reductive 
Klein pair \( (\LieA{a}, \LieA{g}) \) with splitting \( \LieA{a} 
= \LieA{g} \oplus \LieA{m} \). We assume that \( \LieA{m} \) 
carries an \( \AdAction_G \)-invariant, positive-definite, 
symmetric bilinear form \( \Xi_{\LieA{m}}: \LieA{m} \times \LieA{m} 
\to \R \). These properties ensure that \( \Xi_{\LieA{m}} \) 
induces a (weak)\footnotemark{} Riemannian metric \( \Xi \) on 
\( B \) via the isomorphism \( \TBundle B \isomorph P \times_G 
\LieA{m} \) induced by the Cartan connection:
\footnotetext{The metric is weak in the sense that the induced 
musical map \( \sharp: \TBundle_b B \to \TBundle^*_b B \) is only 
injective and not necessarily an isomorphism. One complication 
for weak metrics is that the Koszul formula only yields uniqueness 
of the Levi-Civita connection, but not its existence.}
\begin{equation}
\label{eq:cartan:riemannianMetric}
\Xi_b (Z_1, Z_2) = \Xi_{\LieA{m}} \bigl(\theta_{\LieA{m}}(\hat{Z}_1), 
\theta_{\LieA{m}}(\hat{Z}_2)\bigr),
\end{equation}
where \( b \in B \), \( Z_1, Z_2 \in \TBundle_b B \),  
\( p \in \pi^{-1}(b) \), and \( \hat{Z}_i \in \TBundle_p P \) 
is a lift of \( Z_i \). The Cartan covariant derivative defined 
in~\eqref{eq:cartan:covariantDerivative} has, in general, a 
non-vanishing torsion and thus does not coincide with the 
Levi-Civita connection of \( \Xi \).
However, the geodesics of the Cartan connection and the Levi-Civita 
connection in fact coincide under natural assumptions.
\begin{prop}
\label{prop:cartan:riemannianConnection}
Let \( (P, \theta) \) be a Cartan geometry modeled on a reductive 
Klein pair \( (\LieA{a}, \LieA{g}) \) with splitting \( \LieA{a} 
= \LieA{g} \oplus \LieA{m} \).
Assume that \( \LieA{m} \) carries an \( \AdAction_G \)-invariant, 
positive-definite, symmetric bilinear form \( \Xi_{\LieA{m}}: 
\LieA{m} \times \LieA{m} \to \R \) and let \( \Xi \) be the 
induced Riemannian metric on \( B \).
Suppose that there exists a \( 1 \)-form \( \alpha \) on \( B \) 
with values in \( P \times_G \LieA{g} \) such that
\begin{equation}
\Omega_{\LieA{m}} + 
\frac{1}{2} \wedgeLie{\theta_{\LieA{m}}}{\theta_{\LieA{m}}}_{\LieA{m}} 
= \wedgeLie{\alpha}{\theta_{\LieA{m}}},
\end{equation}
where the subscript in the second term indicates that we take the \( \LieA{m} \)-component of the commutator.
Then the Levi-Civita connection \( \nabla^{\textrm{LC}} \) of 
\( \Xi \) is given by
\begin{equation}
\label{eq:cartan:riemannianConnection}
\theta_{\LieA{m}} \bigl(\nabla^{\textrm{LC}}_X Y\bigr)
= \theta_{\LieA{m}}\bigl(\nabla^{\theta}_X Y\bigr) - 
\commutator{\alpha(X)}{\theta_{\LieA{m}}(Y)},
\end{equation}
where \( \nabla^{\theta} \) is the Cartan covariant derivative 
associated with \( \theta \) as defined 
in~\eqref{eq:cartan:covariantDerivative}.
In particular, if \( \theta \) is torsion-free, then 
\begin{equation}
\label{eq:cartan:riemannianConnectionTorsionFree}
\theta_{\LieA{m}} \bigl(\nabla^{\textrm{LC}}_X Y\bigr) 
= \theta_{\LieA{m}} \bigl(\nabla^{\theta}_X Y\bigr) 
- \frac{1}{2}\commutator{\theta_{\LieA{m}}(X)}{\theta_{\LieA{m}}(Y)}_{\LieA{m}}
\end{equation}
and the geodesics of \( \nabla^{\textrm{LC}} \) and
\( \nabla^{\theta} \) coincide.
\end{prop}
\begin{proof}
The \( \LieA{m} \)-component of the curvature \( \Omega \) of 
\( \theta \) is given by
\begin{equation}
\Omega_{\LieA{m}} = \dif \theta_{\LieA{m}} - 
\wedgeLie{\theta_{\LieA{g}}}{\theta_{\LieA{m}}} -
\frac{1}{2} \wedgeLie{\theta_{\LieA{m}}}{\theta_{\LieA{m}}}_{\LieA{m}},
\end{equation}
where \( \tau = \dif \theta_{\LieA{m}} - 
\wedgeLie{\theta_{\LieA{g}}}{\theta_{\LieA{m}}} \) 
is the torsion of the affine connection induced by 
\( \theta_{\LieA{g}} \), see, \eg, 
\parencite[Proposition~2.1.19]{RudolphSchmidt2014}.
By assumption, there exists a \( 1 \)-form \( \alpha \) on \( B \) 
with values in \( P \times_G \LieA{g} \) such that 
\( \tau = \wedgeLie{\alpha}{\theta_{\LieA{m}}} \).
Consider the Cartan connection
\begin{equation}
	\bar{\theta} = \theta_{\LieA{g}} + \alpha + \theta_{\LieA{m}}.
\end{equation}
It has torsion \( \bar{\Omega}_{\LieA{m}} = -
\frac{1}{2} \wedgeLie{\theta_{\LieA{m}}}{\theta_{\LieA{m}}}_{\LieA{m}} \).
Moreover, the torsion \( \bar{\tau} \) of the affine connection 
\( \nabla^{\bar{\theta}} \) induced by \( \bar{\theta}_{\LieA{g}} 
= \theta_{\LieA{g}} + \alpha \) is zero.
Unraveling the definitions, we find that \( \nabla^{\bar{\theta}}_X Y =
\nabla^{\theta}_X Y - \commutator{\alpha(X)}{\theta_{\LieA{m}}(Y)} \).
Moreover, for vector fields \( X, Y, Z \) on \( B \) with 
\( \bar{\theta} \)-horizontal lifts \( \hat{X}, \hat{Y}, \hat{Z} \) 
to \( P \), we have
\begin{equation}\begin{split}
Z \Xi(X, Y)
&= \hat{Z} \Xi_{\LieA{m}} \bigl(\theta_{\LieA{m}}(\hat{X}), 
\theta_{\LieA{m}}(\hat{Y})\bigr) \\
&= \Xi_{\LieA{m}} \bigl(\hat{Z} (\theta_{\LieA{m}}(\hat{X})), 
\theta_{\LieA{m}}(\hat{Y})\bigr) + 
\Xi_{\LieA{m}} \bigl(\theta_{\LieA{m}}(\hat{X}), 
\hat{Z} (\theta_{\LieA{m}}(\hat{Y}))\bigr)\\
&= \Xi_{\LieA{m}} \bigl(\theta_{\LieA{m}}(\nabla^{\bar{\theta}}_Z X), \theta_{\LieA{m}}(\hat{Y})\bigr) + 
\Xi_{\LieA{m}} \bigl(\theta_{\LieA{m}}(\hat{X}), 
\nabla^{\bar{\theta}}_Z \theta_{\LieA{m}}(\hat{Y})\bigr)\\
&= \Xi \bigl(\nabla^{\bar{\theta}}_Z X, Y\bigr) + 
\Xi\bigl(X, \nabla^{\bar{\theta}}_Z Y\bigr).
\end{split}\end{equation}
Hence, \( \nabla^{\bar{\theta}} \) is metric, and thus it is the 
Levi-Civita connection of \( \Xi \).

Finally, if \( \Omega_{\LieA{m}} = 0 \), then we can choose 
\( \alpha = \frac{1}{2} \theta_{\LieA{m}} \) and 
obtain~\eqref{eq:cartan:riemannianConnectionTorsionFree}.
In this case, \( \nabla^{\textrm{LC}}_X X = \nabla^{\theta}_X X \) 
so that both connections have the same geodesics.
\end{proof}
Note that this discussion of geodesics does not at all invoke 
the role of the Cartan bundle as a local model for a (complexified) 
orbit. In fact, it also gives a slightly different geometric 
interpretation of some constructions in 
\parencite{Modin2017,KhesinMisiolekModin2021}.
Let us first discuss the case of a homogenous space, and then 
apply it to Wasserstein geometry.
\begin{example}[Homogenous space]
Let \( (\LieA{a}, \LieA{g}) \) be a Klein pair with group \( G \) 
and assume that \( \LieA{a} \) integrates to a Lie group \( A \) 
with an exponential map.
Furthermore, suppose that \( G \) is a principal Lie subgroup 
of \( A \), \ie, the left action of \( G \) on \( A \) yields a 
principal \( G \)-bundle \( A \to G \backslash A \) (this is no 
longer automatic in the infinite-dimensional setting, 
see \parencite[Problem~IX.3]{Neeb2006}).
Then right \( A \)-invariant Cartan connections on 
\( A \to G \backslash A \) are in bijective correspondence 
with \( \AdAction_G \)-equivariant isomorphisms 
\( \lambda: \LieA{a} \to \LieA{a} \) and 
\( \check{\lambda}: \LieA{a} \slash \LieA{g} \to 
\LieA{a} \slash \LieA{g} \) fitting into the commutative diagram
\begin{equationcd}
\label{eq:cartan:homogenousSpaceExactSequence}
0 \to[r] &\LieA{g} \to[r] \to[d, "\id"] &\LieA{a} \to[r] 
\to[d, "\lambda"] &\LieA{a} \slash \LieA{g} \to[r] 
\to[d, "\check{\lambda}"] &0\\
0 \to[r] &\LieA{g} \to[r] &\LieA{a} \to[r] 
&\LieA{a} \slash \LieA{g} \to[r] &0.
\end{equationcd}
The associated Cartan connection \( \theta \) is given 
by \( \theta_a(\xi \ldot a) = \lambda(\xi) \) for 
\( a \in A \) and \( \xi \in \LieA{a} \).
For \( \xi \in \LieA{a} \), the constant vector field 
\( \theta^{-1}(\xi) \) on \( A \) has the form 
\( \theta^{-1}(\xi)_a = \lambda^{-1}(\xi) \ldot a \) 
and the integral curve starting at \( a \) is given by 
the \( 1 \)-parameter subgroup 
\( t \mapsto \exp(t \lambda^{-1}(\xi)) \ldot a \).
The projection of this curve to \( G \backslash A \) is a 
generalized geodesic. If \( \lambda \) is the identity, 
then \( \theta \) is the Maurer--Cartan form of \( A \) 
and \( \theta^{-1}(\xi) \) is the right invariant vector 
field on \( A \) corresponding to \( \xi \in \LieA{a} \).

Assume now that we have a reductive decomposition 
\( \LieA{a} = \LieA{g} \oplus \LieA{m} \).
Then \( \lambda \) is necessarily of the form 
\( \lambda(\eta + \xi) = \eta + \tilde{\lambda}(\xi) 
+ \check{\lambda}(\xi) \), \( \eta \in \LieA{g} \), 
\( \xi \in \LieA{m} \), for some \( \AdAction_G \)-equivariant 
linear map \( \tilde{\lambda}: \LieA{m} \to \LieA{g} \) 
and an \( \AdAction_G \)-equivariant isomorphism 
\( \check{\lambda}: \LieA{m} \to \LieA{m} \).
Then the unique geodesic starting at the identity coset in 
the direction \( \equivClass{e, \xi} \in A \times_G \LieA{m} \) 
is given by \( t \mapsto 
\equivClass*{\exp\Bigl(t \check{\lambda}^{-1}(\xi) - 
t \tilde{\lambda}\bigl(\check{\lambda}^{-1}(\xi)\bigr)\Bigr)} \).
\end{example}
\begin{example}[Information Geometry]\label{ex:cartan:informationGeometry}
Let \( (M, g) \) be a compact Riemannian manifold of 
dimension \( n \). Denote the volume form of \( g \) 
by \( \mu \) and the space of smooth volume forms on 
\( M \) with total volume \( 1 \) by \( \SectionSpaceAbb{B} \).
The diffeomorphism group 
\( \SectionSpaceAbb{P} = \DiffGroup(M) \) fibers by 
the push-forward map over \( \SectionSpaceAbb{B} \), that is,
\( \pi:\phi\ni \DiffGroup(M) \mapsto\phi_* \mu = (\phi^{-1})^* \mu
\in  \mathcal{B} \).
The structure group of this principal bundle is the 
group \( \DiffGroup(M, \mu) \) of 
volume-preserving diffeomorphism and 
\( \psi \in \DiffGroup(M, \mu) \) acts on 
\( \SectionSpaceAbb{P} \) by 
\( \phi \mapsto \phi \circ \psi^{-1} \).
It is tempting to use the Maurer--Cartan form on 
\( \SectionSpaceAbb{P} \) together with the 
Helmholtz--Hodge decomposition of vector fields 
on \( M \) into divergence-free and gradient 
parts to define a Cartan connection on 
\( \SectionSpaceAbb{P} \). However, only metric-preserving 
diffeomorphisms leave the Helmholtz decomposition invariant; 
this renders the construction of a Cartan connection 
along these lines difficult.

Instead, we start with the following weighted Helmholtz 
decomposition of vector fields on \( M \). For every 
\( \phi \in \DiffGroup(M) \) and a vector field \( X \) 
on \( M \), we can decompose \( \phi_* X \) into a 
divergence-free part \( X^{\phi_* \mu-\divergence} \) 
and a gradient part \( X^{\grad} \) with respect to the 
volume form \( \phi_* \mu \) as
\begin{equation}
	\phi_* X = X^{\phi_* \mu-\divergence} + X^{\grad}, 
	\quad \divergence_{\phi_* \mu} X^{\phi_* \mu-\divergence} = 0.
\end{equation}
Equivalently, we can write \( X = \phi^* X^{\phi_* \mu-\divergence} 
+ \phi^* X^{\grad} \), where \( \phi^* X^{\phi_* \mu-\divergence} \) 
is now divergence-free with respect to \( \mu \).
In other words, every vector field \( X \) uniquely decomposes 
as follows:
\begin{equation}
\label{eq:cartan:helmholtzDecomposition}
X = X^{\divergence} + \phi^* \grad f, \qquad 
X^{\divergence} \in \VectorFieldSpace(M, \mu), \; \; 
f \in \sFunctionSpace_0(M).
\end{equation}
Note that both \( X^{\divergence} \) and \( f \) also 
depend on \( \phi \). In fact, if we regard both objects 
as functions of \( X \) and \( \phi \), then we get the 
following equivariance properties under the action 
of \( \psi \in \DiffGroup(M, \mu) \):
\begin{equation}
	X^{\divergence}(\psi_* X, \phi \circ \psi^{-1}) 
	= \psi_* X^{\divergence}(X, \phi), \qquad 
	f(\psi_* X, \phi \circ \psi^{-1}) = f(X, \phi).
\end{equation}
Using this decomposition, we define a Cartan connection 
on \( \SectionSpaceAbb{P} \) by
\begin{equation}
	\theta_\phi(\tangent \phi \circ X) = (X^{\divergence}, f), 
	\quad \phi \in \SectionSpaceAbb{P}, \; \;
	X = X^{\divergence} + \phi^* \grad f \in \VectorFieldSpace(M),
\end{equation}
relative to the Klein pair \( \VectorFieldSpace(M, \mu) 
\subset \VectorFieldSpace(M, \mu) \oplus \sFunctionSpace_0(M) \).
The equivariance properties above show that \( \theta \) is 
indeed equivariant if we let \( \DiffGroup(M, \mu) \) act 
trivially on \( \sFunctionSpace_0(M) \) (and by push-forward 
on \( \VectorFieldSpace(M, \mu) \)).
In particular, this is now a reductive Cartan geometry.

Note that \( \theta^{-1}_\phi(0, f) = \grad f \circ \phi \).
The geodesics of the Cartan connection are hence curves 
\( \mu(t) = \bigl(\flow^{\grad f}_t\bigr)_* \mu \) for a 
time-independent \( f \in \sFunctionSpace_0(M) \).
By passing to the densities \( \rho(t) = \mu(t) \slash \mu \), 
the geodesic equation can be cast into the more familiar 
form of the continuity equation
\begin{equation}
	\partial_t \rho + \divergence_\mu(\rho \grad f) = 0.
\end{equation}

Above, we have absorbed the pull-back by \( \phi \) in the 
decomposition~\eqref{eq:cartan:helmholtzDecomposition} in 
the Cartan connection itself. Alternatively, we can keep 
the pull-back explicit and try to define a different 
Cartan connection as follows. Consider again the Klein 
pair \( \VectorFieldSpace(M, \mu) \subset 
\VectorFieldSpace(M, \mu) \oplus \sFunctionSpace_0(M) \),
but this time with the action of \( \DiffGroup(M, \mu) \) 
by push-forward on both components.
\Textcite{Otto2001} showed that the tangent space to 
\( \SectionSpaceAbb{B} \) at \( \nu \) can naturally be 
identified with \( \sFunctionSpace_0(M) \) using the 
parametrization\footnotemark{}\footnotetext{Expressed 
relative to the density function 
\( \rho = \nu \slash \mu \), 
this parametrization is written as
\begin{equation}
f \mapsto - \frac{1}{\rho} \nabla_i \bigl( \rho \nabla^i f \bigr).
\end{equation}}
\begin{equation}
\label{eq:cartan:ottoParametrization}
\sFunctionSpace_0(M) \ni f \mapsto - \difLie_{\grad f} \nu 
\in \TBundle_\nu \SectionSpaceAbb{B}.
\end{equation}
Note that \( \tangent_\phi (\tangent \phi \circ X) =
- \difLie_{\grad f} (\phi_* \mu) = 
- \difLie_{\grad f} \nu \) if \( X = \phi^* \grad f \) 
and \( \nu = \phi_* \mu \).
Using this identity, we can reinterpret the 
identification~\eqref{eq:cartan:ottoParametrization} as 
the \( \sFunctionSpace_0(M) \)-component of a Cartan 
connection \( \bar{\theta} \) on \( \SectionSpaceAbb{P} \) 
by setting
\begin{equation}
\left(\bar{\theta}_{\sFunctionSpace_0(M)}\right)_\phi(
\tangent \phi \circ X) = 
f \circ \phi, \qquad \phi \in \SectionSpaceAbb{P}, \; \; 
X = X^{\divergence} + \phi^* \grad f \in \VectorFieldSpace(M).
\end{equation}
It would be now natural to complete this to a Cartan connection 
by choosing the \( \VectorFieldSpace(M, \mu) \)-component in 
such a way that the geodesics coincide with the Wasserstein geodesics.
Following \parencite[Lemma~4]{Lott2008}, the Levi-Civita 
connection \( \bar{\nabla} \) on \( \SectionSpaceAbb{B} \) 
corresponding to the Wasserstein metric is given, under the 
identification~\eqref{eq:cartan:ottoParametrization}, by the map
\begin{equation}
\bar{\nabla}_g: f \mapsto G_\nu \dif^*_\nu 
\bigl((\nabla^j g) \cdot (\nabla_i \nabla_j f) \dif x^i \bigr),
\end{equation}
where \( \dif^*_\nu \) is the adjoint of \( \dif \) in
\( \LTwoFunctionSpace(M, \nu) \), \( G_\nu \) is the Green 
operator of the Laplacian \( \dif^*_\nu \dif \) and 
\( g \in \sFunctionSpace_0(M) \) gives the direction 
of the derivative. Hence, in order for this affine connection 
to come from a Cartan connection, one would need write 
\( \bar{\nabla}_g (f) \) as the application of a divergence-free 
vector field on \( f \). In particular, \( \bar{\nabla}_g \) 
would need to be a first-order differential operator.
However, the appearance of the Green operator in the expression 
of \( \bar{\nabla}_g \) makes this impossible and one only 
obtains a \emph{pseudo}-differential operator of order \( 1 \).
	
\emph{Thus, there is no Cartan connection \( \bar{\theta} \) 
on \( \SectionSpaceAbb{P} \) with the \( \sFunctionSpace_0(M) \)-component given above, whose geodesics coincide 
with the Wasserstein geodesics.} The structure group 
\( \DiffGroup(M, \mu) \) of the Cartan bundle 
\( \SectionSpaceAbb{P} \) is too small to serve as the 
holonomy group of the Wasserstein metric.
In order to fit the Wasserstein metric into the framework 
of Cartan geometry, one would need to enlarge the structure 
group of the Cartan bundle to a Lie group whose Lie algebra 
contains the pseudo-differential operator \( \bar{\nabla}_g \).
This may be possible, based on the work 
\parencite{AdamsRatiuSchmid1985}, and it is an interesting 
direction for future research.

\end{example}
\subsection{Extremal elements}
\label{sec:cartan:extremalElements}

In K\"ahler geometry, the extremal vector field is closely 
related to the Futaki invariant.
We will now show that the same is true in the context 
of Cartan geometries.
Our treatment of the material here is inspired by, and 
follows in spirit, the original treatment of \textcite{FutakiMabuchi1995}.

Let \( (P, \theta) \) be a Cartan geometry modeled on a 
reductive Klein pair \( (\LieA{a}, \LieA{g}) \) with group
\( G \). We assume that \( \LieA{a} \) carries an 
\( \AdAction_G \)-invariant bilinear form 
\( \Xi: \LieA{a} \times \LieA{a} \to \R \) which vanishes 
on \( \LieA{g} \) and is (weakly) non-degenerate on 
\( \LieA{a} \slash \LieA{g} \).
These properties ensure that \( \Xi \) produces 
a (weak) Riemannian metric \( \Xi \) on \( B \) 
via the isomorphism 
\( \TBundle B \isomorph P \times_G (\LieA{a} \slash \LieA{g}) \) 
induced by the Cartan connection.
We can pull back this metric along the action of \( \LieA{a}_m \) 
to obtain, for every \( p \in P \), a bilinear form
\begin{equation}
	\label{eq:cartan:extremalElements:pullbackForm}
\Xi_p: \LieA{a}_m \times \LieA{a}_m \to \R, \quad 
(\zeta_1, \zeta_2) \mapsto 
\Xi\bigl(\theta_p(p \ldot \zeta_1), \theta_p(p \ldot \zeta_2)\bigr).
\end{equation}
The equivariance properties ensure that \( \Xi_p \) 
depends only on the point \(b \defeq \pi(p) \in B \); we will write 
\( \Xi_b \) for \( \Xi_p \) in this case. 
\begin{lemma}\label{extremal1}
For \( \zeta_1, \zeta_2 \in \LieA{a}_m \), the covariant derivative 
of the function \( b \mapsto \Xi_b(\zeta_1, \zeta_2) \) 
is given by
\begin{equation}
\nabla_\xi \bigl(\Xi_{(\cdot)}(\zeta_1, \zeta_2)\bigr)(b) = 
\Xi\bigl(\commutator{\xi}{\theta_p(p \ldot \zeta_1)}, 
\theta_p(p \ldot \zeta_2)\bigr)
+ \Xi\bigl(\theta_p(p \ldot \zeta_1), 
\commutator{\xi}{\theta_p(p \ldot \zeta_2)}\bigr)
\end{equation}
for \( \xi \in \LieA{a} \) and \( p \in \pi^{-1}(b) \).
In particular, if \( \Xi \) is \( \adAction_{\LieA{a}} \)-invariant 
and \( B \) path-connected, then \( \Xi_b(\zeta_1, \zeta_2) \) 
is independent of \( b \in B \).
\end{lemma}
\begin{proof}
Let \( \zeta_1, \zeta_2 \in \LieA{a}_m \) and 
\( \xi \in \LieA{a} \). Then, abbreviating 
\( s_\zeta(p) = \theta(p \ldot \zeta) \in \LieA{a} \) and using~\eqref{eq:cartan:constFutaki:covszeta}, we find
\begin{equation}\begin{split}
\nabla_\xi \bigl(\Xi_{(\cdot)}(\zeta_1, \zeta_2)\bigr)(b)
&= \Xi \bigl(\nabla_\xi s_{\zeta_1}(p), s_{\zeta_2}(p)\bigr) 
+ \Xi \bigl(s_{\zeta_1}(p), \nabla_\xi s_{\zeta_2}(p)\bigr)\\
&= \Xi \bigl(\commutator{\xi}{s_{\zeta_1}(p)}, s_{\zeta_2}(p)\bigr) 
+ \Xi \bigl(s_{\zeta_1}(p), \commutator{\xi}{s_{\zeta_2}(p)}\bigr),
\end{split}\end{equation}
because \( \Xi \) vanishes on \( \LieA{g} \).
The second statement follows from the first by the path-connectedness 
of \( B \).
\end{proof}
If the assumptions of the lemma are satisfied, then we obtain a 
constant bilinear form \( \Xi_m \equiv \Xi_b: 
\LieA{a}_m \times \LieA{a}_m \to \R \).
Assuming that \( \Xi_m \) is strongly non-degenerate, 
we can define an extremal element of \( \LieA{a}_m \) as 
the unique one corresponding to the Futaki functional 
\( F: \LieA{a}_m \ni \zeta \mapsto F_\zeta \in \R \) under 
\( \Xi_m \).
\begin{definition}
An element \( \zeta_m \in \LieA{a}_m \) is called 
\emphDef{extremal} if
\begin{equation}
	\label{eq:cartan:extremalElements:extremalCondition}
		F_\zeta = \Xi_m(\zeta_m, \zeta)
\end{equation}
for all \( \zeta \in \LieA{a}_m \).
\end{definition}
By definition, this condition is equivalent to
\begin{equation}
\kappa_\LieA{a}\bigl(J(\chi(p)), \theta_p (p \ldot \zeta)\bigr) = \Xi\bigl(\theta_p(p \ldot \zeta_m), \theta_p(p \ldot \zeta)\bigr)
\end{equation}
for some (and hence all) \( p \in P \). Thus, up to the pull-back 
by the map \( \zeta \mapsto \theta_p(p \ldot \zeta) \), an 
extremal element is the same as the projection of \( J(\chi(p)) \) 
onto \( \LieA{a}_m \).

\section{K\"ahler geometry}
\label{sec:kaehler}

Let \( (M, \omega) \) be a finite-dimensional symplectic manifold, 
assumed to be compact, for simplicity. Denote by 
\( \SectionSpaceAbb{I} \) the space of all almost complex 
structures on \( M \) compatible with \( \omega \).
Then \( \SectionSpaceAbb{I} \) is an open  contractible 
subset, in the sense of 
\parencite[Section~2.1]{DiezRatiuSymplecticConnections}, of the 
Fr\'echet space \( \sSectionSpace(M, \EndBundle(TM)) \), see 
\parencite[Equation~(5.7)]{DiezRatiuSymplecticConnections}.
For a given \( j \in \SectionSpaceAbb{I} \), define 
\( \SectionSpaceAbb{P} \) to be the space of all tuples 
\( (\sigma, \phi) \), where \( \sigma \) is a symplectic 
form on \( M \) compatible  with \( j \) and 
\( \phi \in \DiffGroup(M) \), satisfying 
\( \phi^* \omega = \sigma \). Then, if $j$ is integrable 
which we assume from now on, \( \SectionSpaceAbb{P} \) is 
a Fr\'echet principal bundle with structure group 
\( \DiffGroup(M, \omega) \) over the space 
\( \SectionSpaceAbb{K} \) of K\"ahler forms in the same 
class as \( \omega \). Since \( \sigma \) is uniquely 
determined by \( \phi \), we may regard \( \SectionSpaceAbb{P} \) 
as a submanifold of \(\DiffGroup(M)\) as follows:
\begin{equation}
\SectionSpaceAbb{P} = \set*{\phi \in \DiffGroup(M) \given 
\sigma_\phi \defeq \phi^* \omega \in \SectionSpaceAbb{K}}. 
\end{equation}
We also put 
\begin{equation}
j_\phi \defeq \phi_* j = 
\tangent\phi\circ j \circ \tangent\phi^{-1}, 
\qquad \dif^{c_\phi} f \defeq j_\phi \dif f.
\end{equation}
The action of \( \psi \in \DiffGroup(M, \omega) \) on 
$\SectionSpaceAbb{P}$ is given by $\phi \mapsto \psi \circ \phi$.
\Textcite{Donaldson1999a} endowed \( \SectionSpaceAbb{P} \) 
with a natural principal connection as follows.
First note that every \( \sigma \in \SectionSpaceAbb{K} \) 
is of the form \( \sigma = \omega + \dif \dif^c \varrho \) 
for some \( \varrho \in \sFunctionSpace(M) \), where 
\( \dif^c f \defeq j \dif f \).
This identifies the tangent space 
\( \TBundle_\sigma \SectionSpaceAbb{K} \) with 
\(\sFunctionSpace(M)/\R \isomorph \sFunctionSpace_0(M) \).
Under this identification, tangent vectors at 
\( \phi \in \SectionSpaceAbb{P} \) are of the form 
\( X \circ \phi \) with \( X \in \VectorFieldSpace(M) \) satisfying 
\begin{equation}\label{eq:kaehler:tangentP}
	\difFracAt{}{t}{0}
\left(\flow^X_t \circ\phi\right)^*\omega = 
\phi^* \difLie_X \omega =
\dif\dif^c \varrho_{\phi, X},
\end{equation}
for some \( \varrho_{\phi, X} \in \sFunctionSpace_0(M) \).
Note that \( \varrho_{\phi, X} \) is uniquely determined by 
this equation and, geometrically, it is the projection of 
\( X \circ \phi \in \TBundle_\phi \SectionSpaceAbb{P} \) down 
to a tangent vector in \( \TBundle_{\phi^* \omega} 
\SectionSpaceAbb{K} \isomorph \sFunctionSpace_0(M) \).
Hence, a tangent vector \( X \) to 
\( \SectionSpaceAbb{P} \) at \( \phi \) satisfies the 
constraint that \( X \contr \omega - \phi_* \dif^c \varrho_{\phi, X} \) 
is closed, where $\phi_* = (\phi^*)^{-1}$.
We can take the \( \omega \)-dual of this closed 
\( 1 \)-form and obtain a symplectic vector field.
It is now easy to verify that this is prescription 
defines a connection on \( \SectionSpaceAbb{P} \).
Our first basic observation is that this connection 
defined by Donaldson can be extended to a Cartan 
connection for the Klein pair 
\( \bigl(\VectorFieldSpace(M, \omega), 
\VectorFieldSpace(M, \omega) \oplus 
\I \, \HamVectorFields(M, \omega)\bigr) \) in the following way.

We begin by clarifying our conventions for the Lie algebra 
structure of this Klein pair. In the following, it will be 
convenient to regard the Lie algebra of vector fields as the left Lie algebra associated with the Lie 
group of diffeomorphisms. Thus, its Lie bracket 
is minus the usual Lie bracket of vector fields. Hence, 
on \( \VectorFieldSpace(M, \omega) \oplus 
\I \, \HamVectorFields(M, \omega) \) we define the bracket by
\begin{equation}
\label{eq:kaehler:bracket}
\commutator*{X + \I Y}{X' + \I Y'} = 
-\commutator{X}{X'} - \I \commutator{X}{Y'} 
- \I \commutator{Y}{X'} + \commutator{Y}{Y'}.
\end{equation}
This choice is consistent with having the adjoint action 
of \( \DiffGroup(M, \omega) \) on \( \VectorFieldSpace(M, \omega) 
\oplus \I \, \HamVectorFields(M, \omega) \) given by push-forward.
Next, since $j_\phi = \phi_* j$, we have
\begin{equation}\label{eq:kaehler:pushforwardDifc}
\phi_* (\dif^c f) = \phi_* \bigl(j\dif f \bigr) 
= \dif^{c_\phi} \bigl(\phi_* f \bigr)
\end{equation}
and the vector field $X$ has the decomposition
\begin{equation}\label{decompX}
X = \omega^\sharp (X \contr \omega - 
\phi_* \dif^c \varrho_{\phi, X}) +
\omega^\sharp j_\phi(\dif \phi_* \varrho_{\phi, X}).
\end{equation}
Motivated by this, the Cartan connection is given by 
the following lemma.
\begin{lemma}\label{lem1}
The assignment
\begin{equation}
\label{eq:kaehler:cartanConnection}
\theta_{\phi}(X \circ \phi) = 
\omega^\sharp \bigl( X \contr \omega - 
\phi_* \dif^c \varrho_{\phi, X} \bigr)\ 
- \ \I \omega^\sharp \bigl(\dif \phi_* \varrho_{\phi, X} \bigr)
\end{equation}
defines a Cartan connection on \( \SectionSpaceAbb{P} \) 
relative to the Klein pair given by the subalgebra 
\( \VectorFieldSpace (M, \omega) \) in 
\( \VectorFieldSpace (M, \omega) \oplus \I \, 
\HamVectorFields(M, \omega) \) with group 
\( \DiffGroup(M, \omega) \).
\end{lemma}
\begin{proof}
If \( X \) is a symplectic vector field, then 
\( \varrho_{\phi, X} = 0 \) and consequently 
\begin{equation}
\theta_{\phi}(X \circ \phi) = 
\omega^\sharp \bigl( X \contr \omega \bigr) = X.
\end{equation}
Moreover, for every \( \psi \in \DiffGroup(M, \omega) \), 
we have 
\begin{equation}
\dif \dif^c \varrho_{\psi \circ \phi, \psi_* X} 
= (\psi \circ \phi)^* \difLie_{\psi_* X} \omega 
= \phi^* \difLie_X \omega = \dif \dif^c \varrho_{\phi, X}.
\end{equation}
Hence, \( \varrho_{\psi \circ \phi, \psi_* X} = 
\varrho_{\phi, X} \) and
\begin{equation}
\theta_{\psi \circ \phi}\bigl(\tangent \psi \circ X \circ \phi\bigr)
= \theta_{\psi \circ \phi}\bigl((\psi_* X) \circ \psi \circ \phi\bigr)
= \psi_* \bigl(\theta_\phi (X \circ \phi) \bigr).
\end{equation}
This verifies that the form \( \theta \) is equivariant with 
respect to the action of \( \DiffGroup(M, \omega) \) on 
\( \SectionSpaceAbb{P} \) and \( \VectorFieldSpace(M) \).
Finally, the decomposition~\eqref{decompX} shows that 
\( \theta_\phi: \TBundle_\phi \SectionSpaceAbb{P} \to 
\VectorFieldSpace(M, \omega) \oplus \I \, 
\HamVectorFields(M, \omega) \) is an isomorphism.
\end{proof}

The isomorphism 
\( \SectionSpaceAbb{P} \times_{\DiffGroup(M, \omega)} 
\HamVectorFields(M, \omega) \isomorph 
\TBundle \SectionSpaceAbb{K} \) induced by the imaginary 
part of the Cartan connection is given by
\begin{equation}\label{eq:kaehler:tangentBundleIso}
\equivClass{\phi, X_f} \mapsto 
\dif \dif^c \phi^* f \in 
\TBundle_{\phi^* \omega} \SectionSpaceAbb{K}.
\end{equation}
Note that \( \SectionSpaceAbb{I} \) carries an almost 
complex structure given by \( \SectionMapAbb{j}_j(A) = -j A \).
The sign is chosen in such a way that \( \SectionMapAbb{j} \) 
is compatible with the symplectic forms on 
\( \SectionSpaceAbb{I} \) considered below, that is, 
the Riemannian metric associated is positive definite (it 
is the $L^2$-inner product), \cf 
\parencite[Equations~(5.1) and~(5.2)]{DiezRatiuSymplecticConnections}.

Now the choice of sign in the second term 
of~\eqref{eq:kaehler:cartanConnection} is motivated by the 
requirement that we want the infinitesimal 
action~\eqref{eq:cartan:infAction} associated with 
\( \theta \) to coincide with the natural action
\begin{equation}
(X + \I Y) \ldot j \defeq
X \ldot j + \SectionMapAbb{j}_j (Y \ldot j) 
= -\difLie_X j + j \difLie_Y j, \quad 
X \in \VectorFieldSpace(M, \omega), \;
Y \in \HamVectorFields(M, \omega).
\end{equation}
In fact, define \( \chi: \SectionSpaceAbb{P} \to \SectionSpaceAbb{I} \) 
by \( \chi(\phi) = \phi_* j = j_\phi \).
Then \( \chi \) is equivariant with respect to the action 
of \( \DiffGroup(M, \omega) \) on \( \SectionSpaceAbb{P} \)
and \( \SectionSpaceAbb{I} \). For integrable \( j \), we have
\begin{equation}
	\label{eq:kahler:complexOrbitModel}
	\tangent_{\phi} \chi(X \circ \phi) = - \difLie_X (\phi_* j) 
	= \theta_{\phi}(X) \ldot \chi(\phi),
\end{equation}
where the second equality follows from a straightforward 
calculation using \( j \difLie_Y j = \difLie_{jY} j \) 
(see \parencite[Lemma~1.1.1]{Gauduchon2017}).
The latter equation also shows that elements of 
\( \VectorFieldSpace(M, \omega)_\C \ldot j \) are of the form 
\( \difLie_Z j \) for some vector field \( Z \), 
\cf \parencite[Proposition~9.1.1.~(ii)]{Gauduchon2017}.
Hence \( \img \tangent_{\phi} \chi = 
\VectorFieldSpace(M, \omega)_\C \ldot \chi(\phi) \).
For every \( \varphi \in \DiffGroup(M) \) that preserves 
\( j \), we clearly have \( \chi(\phi \circ \varphi) = 
\chi(\phi) \). In other words, \( \chi \) is invariant 
under the natural right action of \( \DiffGroup(M,j) \).
Moreover, every real \( j \)-holomorphic vector field 
\( Z \) on a compact K\"ahler manifold can be uniquely 
written as a sum \( Z = X - j Y \) with 
\( X \in \VectorFieldSpace(M, \omega) \) and 
\( Y \in \HamVectorFields(M, \omega) \).
This realizes the space of \( j \)-holomorphic vector 
fields as subalgebra of \( \VectorFieldSpace(M, \omega) 
\oplus \I \, \HamVectorFields(M, \omega) \) forming the 
kernel of \( \tangent \chi \), \cf 
\parencite[Corollary~5.15]{DiezRatiuSymplecticConnections}.
In summary, this proves the following statement.
\begin{proposition}
	For every integrable \( j \in \SectionSpaceAbb{I} \), the Cartan bundle \( (\mathcal{P}, \theta) \) provides a model for the complex orbit through \( j \) in the sense of \cref{def:cartan:complexOrbitModel}.
\end{proposition} 
We wish to emphasize that we have a
rather interesting structure for the Lie algebra 
\( \VectorFieldSpace (M, \omega) \oplus 
\I \, \HamVectorFields(M, \omega) \)
in the Klein pair.  It is neither the complexification 
of \( \VectorFieldSpace (M, \omega) \) nor of 
\( \HamVectorFields(M, \omega) \). In particular, it is 
\emph{not natural} to view \( \SectionSpaceAbb{P} \) as 
a complexification of \( \DiffGroup(M, \omega) \) nor 
of \( \HamDiffGroup(M, \omega) \) unless the first Betti 
number vanishes. In particular, \( \SectionSpaceAbb{P} \) 
is not an "infinitesimal Lie group which complexifies 
\( \DiffGroup(M, \omega) \)" in the terminology 
of \parencite[p.~20]{Donaldson1999a}.


\begin{proposition}\label{torsion}
The Cartan connection \( \theta \) has zero curvature. 
Moreover, the Lie derivative of \( \theta \) with 
respect to any holomorphic vector field vanishes.
\end{proposition}
\begin{proof}
As described above, a tangent vector at 
\(\phi \in \SectionSpaceAbb{P}\) is given by a vector 
field \( X \) on \(M\) satisfying the relation
\begin{equation}
\difLie_X \omega = \phi_* \dif \dif^{c}\varrho_{\phi, X} 
= \dif j_\phi \dif \phi_*\varrho_{\phi, X}
\end{equation} 
for some smooth function \(\varrho_{\phi,X}\), where 
\(j_\phi = \phi_* j\). As in~\eqref{decompX}, \(X\) 
can be written as
\begin{equation}\label{decompX2}
X = \omega^\sharp \bigl( X \contr \omega - 
\phi_* \dif^c \varrho_{\phi, X} \bigr) +
\omega^\sharp j_\phi\bigl(\dif \phi_* \varrho_{\phi, X} \bigr).
\end{equation}
We put
\begin{equation}
U_{\phi, X} = \omega^\sharp \bigl( X \contr \omega 
- \phi_* \dif^c \varrho_{\phi, X} \bigr) 
\end{equation}
and 
\begin{equation}\label{decompX3}
V_{\phi, X} = 
- \omega^\sharp \bigl(\dif \phi_* \varrho_{\phi, X} \bigr).
\end{equation}
Then \( X = U_{\phi, X} - j_\phi V_{\phi, X} \) and, 
using~\eqref{eq:kaehler:cartanConnection}, we get
\begin{equation}\label{decompX4}
\theta_\phi(X \circ \phi) = 
U_{\phi, X} + \I V_{\phi, X} \in 
\VectorFieldSpace (M, \omega) \oplus \I \, 
\HamVectorFields(M, \omega).
\end{equation}
Conversely, if \( U_X \in \VectorFieldSpace (M, \omega) \) 
and \( V_X \in \HamVectorFields(M, \omega) \), then we 
can define a vector field \( \bar{X} \) on 
\( \SectionSpaceAbb{P} \) by
\begin{equation}
\bar{X}_\phi = \bigl(U_X - j_\phi V_X \bigr) \circ \phi.
\end{equation}
Since \( \theta_\phi(\bar{X}_\phi) = U_X + \I V_X \) is 
constant as a function of \( \phi \), we refer to 
\( \bar{X} \) as a \( \theta \)-constant vector field.

If \( U_X, U_Y \in \VectorFieldSpace (M, \omega) \) 
and \( V_X, V_Y \in \HamVectorFields(M, \omega) \), 
then the Lie bracket of the corresponding 
\( \theta \)-constant vector fields is given by
\begin{equation}\label{eq:kaehler:commutatorBarFields}
\commutator{\bar{X}}{\bar{Y}}_\phi \circ \phi^{-1}
= \commutator{U_X}{U_Y} - j_\phi \commutator{U_X}{V_Y} 
+ j_\phi \commutator{U_Y}{V_X} + \commutator{V_Y}{V_X}.
\end{equation}

For the proof of this identity, it is helpful to 
abbreviate \( X(\phi) = U_X - j_\phi V_X \) and 
\( Y(\phi) = U_Y - j_\phi V_Y \).
For every smooth function \( \SectionSpaceAbb{F} \) 
on \( \SectionSpaceAbb{P} \), we have
\begin{equation}\begin{split}
\bar{X}\bigl(\bar{Y}(\SectionSpaceAbb{F})\bigr) (\phi)
&= \difFracAt{}{t}{0} \difFracAt{}{s}{0} \SectionSpaceAbb{F}
\Bigl(\flow^{Y\bigl(\flow^{X(\phi)}_t \circ \phi\bigr)}_s 
\circ \flow^{X(\phi)}_t \circ \phi\Bigr) \\
&= \difFracAt{}{t}{0} \difFracAt{}{s}{0} \SectionSpaceAbb{F}
\Bigl(\flow^{Y\bigl(\flow^{X(\phi)}_t \circ \phi\bigr)}_s 
\circ \phi\Bigr) \\
&\quad+ \difFracAt{}{t}{0} \difFracAt{}{s}{0} 
\SectionSpaceAbb{F}\Bigl(\flow^{Y(\phi)}_s \circ 
\flow^{X(\phi)}_t \circ \phi\Bigr) \\
&= \difFracAt{}{t}{0} (\dif \SectionSpaceAbb{F})_\phi 
\Bigl(Y\bigl(\flow^{X(\phi)}_t \circ \phi\bigr) 
\circ \phi\Bigr) \\
&\quad+ \difFracAt{}{t}{0} \difFracAt{}{s}{0} 
\SectionSpaceAbb{F}\Bigl(\flow^{Y(\phi)}_s \circ 
\flow^{X(\phi)}_t \circ \phi\Bigr).
\end{split}\end{equation}
Using the Baker--Campbell--Hausdorff formula 
\begin{equation*}
\flow^{Y}_{-s}\circ \flow^{X}_{-t} \circ 
\flow^{Y}_s \circ \flow^{X}_t = 
\flow^{\commutator*{X}{Y}}_{st} + O \bigl((st)^2\bigr)
\end{equation*} 
for vector fields on  \( M \), see, \eg, 
\parencite[Theorem~1.38]{Biagi2019}, 
we thus find
\begin{equation}\label{eq:kaehler:commutatorOnP}
\commutator*{\bar{X}}{\bar{Y}}_\phi \circ \phi^{-1} 
= \commutator*{X(\phi)}{Y(\phi)} + 
\difFracAt{}{t}{0} Y\Bigl(\flow^{X(\phi)}_t \circ \phi\Bigr)
- \difFracAt{}{t}{0} X\Bigl(\flow^{Y(\phi)}_t \circ \phi\Bigr).
\end{equation}
Now,
\begin{equation}\begin{split}
\difFracAt{}{t}{0} Y\Bigl(\flow^{X(\phi)}_t \circ \phi\Bigr)
&= - \difFracAt{}{t}{0} \Bigl(\flow^{X(\phi)}_t 
\circ \phi\Bigr)_* j \, (V_Y) \\
&= \difLie_{X(\phi)} j_\phi \, (V_Y) \\
&= \commutator{U_X - j_\phi V_X}{j_\phi V_Y} - 
j_\phi \commutator{U_X - j_\phi V_X}{V_Y}.
\end{split}\end{equation}
Inserting this into~\eqref{eq:kaehler:commutatorOnP}, expanding 
terms, and using the vanishing of the Nijenhuis tensor of 
\( j_\phi \) yields~\eqref{eq:kaehler:commutatorBarFields}.
	
To conclude that the curvature of \( \theta \)
vanishes, we have to show that 
\( \theta\bigl(\commutator{\bar{X}}{\bar{Y}}\bigr) + 
\commutator*{\theta(\bar{X})}{\theta(\bar{Y})} = 0 \).
First, by~\eqref{eq:kaehler:bracket},
\begin{equation}
\commutator*{\theta(\bar{X})}{\theta(\bar{Y})} 
= - \commutator{U_X}{U_Y} + \commutator{V_X}{V_Y} 
- \I \bigl(\commutator{U_X}{V_Y} + 
\commutator{V_X}{U_Y}\bigr).
\end{equation}
On the other hand, we conclude from~\eqref{eq:kaehler:commutatorBarFields} that
\begin{equation}
\theta \bigl(\commutator*{\bar{X}}{\bar{Y}}\bigr) 
= \commutator{U_X}{U_Y} - \commutator{V_X}{V_Y} + 
\I \bigl(\commutator{U_X}{V_Y} + \commutator{V_X}{U_Y}\bigr).
\end{equation}
They clearly sum up to zero, which proves the 
vanishing of the curvature.
	
Next, let \(Z^*\) be the vector field on 
\(\SectionSpaceAbb{P}\) induced by the \( j \)-holomorphic 
vector field \( Z \in \VectorFieldSpace(M) \).
That is, \( \flow^{Z^*}_t (\phi) = \phi \circ \flow^{Z}_t \).	
Let \(\bar{X}\) be the \( \theta \)-constant vector 
field corresponding to \(U \in \VectorFieldSpace(M, \omega)\) 
and \(V \in \HamVectorFields(M, \omega)\), \ie, 
\( \bar{X}_\phi \circ \phi^{-1} = U - j_\phi V \).
Since $Z$ preserves $j$, we find
\begin{equation}\begin{split}
\commutator*{Z^*}{\bar{X}}_\phi 
&= \difFracAt{}{t}{0} \tangent \flow^{Z^*}_{-t} 
\Bigl(\bar{X}_{\flow^{Z^*}_t(\phi)} \Bigr)
= \difFracAt{}{t}{0} \bar{X}_{\phi \circ \flow^{Z}_t} 
\circ \flow^{Z}_{-t}
			\\
&= \difFracAt{}{t}{0} \bigl(U - j_{\phi \circ 
\flow^{Z}_t} V\bigr) \circ \phi
= 0.
\end{split}\end{equation}
Hence, 
\begin{equation}\begin{split}
\bigl(\difLie_{Z^*}\theta\bigr)(\bar{X}) &=
\difLie_{Z^*} \bigl(\theta(\bar{X})\bigr)
- \theta\bigl(\commutator{Z^*}{\bar{X}}\bigr) = 0,
\end{split}\end{equation}
where we used that \( \theta(\bar{X}) \) is a 
constant function on \( \SectionSpaceAbb{P} \).
This completes the proof.
\end{proof}

\begin{proof}[Alternative proof]
We give a second proof of \cref{torsion} which uses 
horizontal lifts of constant vector fields on 
\( \SectionSpaceAbb{K} \) instead of \( \theta \)-constant 
vector fields as above. This proof is more in the spirit 
of the original calculation of the curvature of the 
affine connection by \textcite{Donaldson1999a}, see also 
\cref{rem:kaehler:affineConnectionCurvature}.

Let \( \varphi \in \SectionSpaceAbb{P} \) and 
\( X \circ \varphi \in \TBundle_\varphi \SectionSpaceAbb{P} \).
In order to calculate the curvature of \( \theta \), it is 
convenient to extend the tangent vector \( X \circ \varphi \) 
to a vector field on \( \SectionSpaceAbb{P} \).
For this purpose, recall that there exists a unique function
\( \varrho_{\varphi, X} \in \sFunctionSpace_0(M) \) such that 
\( \phi^* \difLie_X \omega = \dif \dif^c \varrho_{\varphi, X} \).
In the following, we will keep \( \varphi \) fixed and 
abbreviate \( \varrho_{\varphi, X} \) by \( \varrho_X \).
For every \( \phi \in \SectionSpaceAbb{P} \), define a 
vector field \( X(\phi) \) on \( M \) by the relation
\begin{equation}
\label{def_of_X_phi}
	\phi^* \bigl(X(\phi) \contr \omega\bigr) = \dif^c \varrho_X.
\end{equation}
Note that \( \varrho_X \) does not depend on \( \phi \) 
but only on the initially chosen point \( \varphi \).
Then, a vector field \( \tilde{X} \) on \( \SectionSpaceAbb{P} \) 
is defined by \( \tilde{X}_\phi = X(\phi) \circ \phi \) 
for every \( \phi \in \SectionSpaceAbb{P} \).
This vector field is indeed tangent to \( \SectionSpaceAbb{P} \) 
because the associated function \( \varrho_{\phi, \tilde{X}(\phi)} \) 
exists and is determined by
\begin{equation}
	\dif \dif^c \varrho_{\phi, \tilde{X}(\phi)} = \phi^* \difLie_{\tilde{X}(\phi)} \omega = \dif \dif^c \varrho_X.
\end{equation}
Hence, \( \varrho_{\phi, \tilde{X}(\phi)} = \varrho_X \).
Recall that the function \( \varrho_X \) geometrically 
corresponds to the projection of \( X \circ \varphi \) 
to a tangent vector in 
\( \TBundle_{\varphi^* \omega} \SectionSpaceAbb{K} \).
Thus, the two tangent vectors \( X \circ \varphi \) 
and \( \tilde{X}_\varphi \) differ only by a vertical 
vector. So, while \( \tilde{X} \) being not quite an 
extension of \( X \circ \varphi \), this difference 
will not affect the calculation of the curvature 
because the curvature of \( \theta \) vanishes when 
one of the arguments is vertical.

Before we proceed with the curvature calculation, we 
collect some useful properties of the vector field \( \tilde{X} \).
Using \( \varrho_{\phi, \tilde{X}(\phi)} = \varrho_X \) and~\eqref{def_of_X_phi}, we find
\begin{equation}
\theta_\phi \bigl(\tilde{X}_\phi\bigr) = 
\omega^\sharp \bigl( X(\phi) \contr \omega - 
\phi_* \dif^c \varrho_X \bigr) - 
\I \omega^\sharp \bigl(\dif \phi_* \varrho_X \bigr) 
= - \I \omega^\sharp \bigl(\dif \phi_* \varrho_X \bigr).
\end{equation}
This shows that \( \tilde{X} \) is the horizontal lift 
of the constant vector field on \( \SectionSpaceAbb{K} \) 
determined by \( \varrho_X \).
Moreover, the defining relation~\eqref{def_of_X_phi} 
of \( X(\phi) \) can be rewritten as
\begin{equation}
X(\phi) \contr \omega = 
\phi_* \dif^c \varrho_X = 
\dif^{c_\phi} \bigl(\phi_* \varrho_X\bigr).
\end{equation}
Using \( \omega^\sharp (\dif^{c_\phi} f) = 
- j_\phi X_f \), where \( X_f \) is the Hamiltonian 
vector field of \( f \), we equivalently have 
\( X(\phi) = - j_\phi X_{\phi_* \varrho_X} \).
If \( Y \circ \varphi \in 
\TBundle_\varphi \SectionSpaceAbb{P} \) is 
another tangent vector, then we can construct a 
vector field \( \tilde{Y} \) on \( \SectionSpaceAbb{P} \) 
in the same way. Then, the Lie bracket of \( \tilde{X} \) 
and \( \tilde{Y} \) as vector fields on 
\( \SectionSpaceAbb{P} \) is given by
\begin{equation}
\label{eq:kaehler:commutatorTildeXtildeY}
\commutator*{\tilde{X}}{\tilde{Y}}_\phi = 
- X_{\poisson{\phi_* \varrho_X}{\phi_* \varrho_Y}} \circ \phi.
\end{equation}
For the proof of this identity, we need a bit of preparation.

First, note that we have
\begin{equation}
\label{eq:kaehler:XphiAppliedToY}
\begin{split}
X(\phi) (\phi_* \varrho_Y) &= 
X(\phi) \contr \dif (\phi_* \varrho_Y) = 
\omega\bigl(X(\phi), X_{\phi_* \varrho_Y}\bigr) \\ 
&= \omega\bigl(X_{\phi_* \varrho_X}, 
j_\phi X_{\phi_* \varrho_Y}\bigr)
= \omega\bigl(Y(\phi), X_{\phi_* \varrho_X}\bigr) 
= Y(\phi) (\phi_* \varrho_X).
\end{split}\end{equation}
Second, since \( j_\phi \) is integrable, the commutator 
of the two vector fields \( X(\phi) \) and \( Y(\phi) \) 
on \( M \) is given by
\begin{equation}
\label{intermediary_1}
\commutator*{X(\phi)}{Y(\phi)} = 
- j_\phi \commutator*{X(\phi)}{X_{\phi_* \varrho_Y}} 
- j_\phi \commutator*{X_{\phi_* \varrho_X}}{Y(\phi)} 
+ \commutator*{X_{\phi_* \varrho_X}}{X_{\phi_* \varrho_Y}}.
\end{equation}
Now, 
\begin{equation}\begin{split}
\commutator*{X(\phi)}{X_{\phi_* \varrho_Y}} \contr \omega
&= \difLie_{X(\phi)} \bigl(X_{\phi_* \varrho_Y} \contr \omega\bigr) 
- X_{\phi_* \varrho_Y} \contr \difLie_{X(\phi)} \omega \\
&= - \dif \bigl(X(\phi) (\phi_* \varrho_Y)\bigr) - 
X_{\phi_* \varrho_Y} \contr \dif \dif^{c_\phi} \phi_* \varrho_X.
\end{split}\end{equation}
Using~\eqref{eq:kaehler:XphiAppliedToY} and 
\( j_\phi (Z \contr \alpha) = 
(j_\phi Z) \contr (j_\phi \alpha) \), we thus find
\begin{equation}
\label{intermediary_2}
\begin{split}
j_\phi &\commutator*{X(\phi)}{X_{\phi_* \varrho_Y}} 
+ j_\phi \commutator*{X_{\phi_* \varrho_X}}{Y(\phi)} \\ 
&= j_\phi \omega^\sharp \bigl(X_{\phi_* \varrho_X} 
\contr \dif \dif^{c_\phi} \phi_* \varrho_Y - 
X_{\phi_* \varrho_Y} \contr \dif \dif^{c_\phi} 
\phi_* \varrho_X\bigr) \\
&= \omega^\sharp \bigl( (j_\phi X_{\phi_* \varrho_X}) 
\contr j_\phi \dif j_\phi \dif \phi_* \varrho_Y - 
(j_\phi X_{\phi_* \varrho_Y}) \contr 
j_\phi \dif j_\phi \dif \phi_* \varrho_X\bigr) \\	
&= \omega^\sharp \bigl( X(\phi) 
\contr \dif^{c_\phi} \dif \phi_* \varrho_Y - 
Y(\phi) \contr \dif^{c_\phi} \dif \phi_* \varrho_X\bigr).
\end{split}\end{equation}
Moreover, we have
\begin{equation}
\label{intermediary_3}
\begin{split}
\commutator*{X(\phi)}{Y(\phi)} \contr \omega
&= \difLie_{X(\phi)} \bigl(Y(\phi) \contr \omega\bigr) 
- Y(\phi) \contr \difLie_{X(\phi)} \omega \\
&= \difLie_{X(\phi)} \dif^{c_\phi} (\phi_* \varrho_Y) 
- Y(\phi) \contr \dif \dif^{c_\phi} (\phi_* \varrho_X).
\end{split}\end{equation}
Thus, using \( \dif^{c_\phi} \dif = - \dif \dif^{c_\phi} \),~\eqref{intermediary_1},~\eqref{intermediary_2}, 
and~\eqref{intermediary_3}, we conclude
\begin{equation}
\label{eq:kaehler:commutatorXphiYphi}
\begin{split}
\commutator*{X(\phi)}{Y(\phi)}
&= \omega^\sharp \bigl( Y(\phi) \contr 
\dif^{c_\phi} \dif \phi_* \varrho_X - 
X(\phi) \contr \dif^{c_\phi} \dif \phi_* \varrho_Y\bigr)
\\ &\quad + X_{\poisson{\phi_* \varrho_X}{\phi_* \varrho_Y}} \\
&= \omega^\sharp \difLie_{Y(\phi)} \dif^{c_\phi} (\phi_* \varrho_X) 
- \omega^\sharp \difLie_{X(\phi)} \dif^{c_\phi} (\phi_* \varrho_Y)
\\ &\quad + 2 \commutator*{X(\phi)}{Y(\phi)} + 
X_{\poisson{\phi_* \varrho_X}{\phi_* \varrho_Y}}.
\end{split}\end{equation}

With these preparations, we are ready to calculate the 
commutator of the vector fields \( \tilde{X} \) and 
\( \tilde{Y} \) on \( \SectionSpaceAbb{P} \) and verify 
the claimed identity~\eqref{eq:kaehler:commutatorTildeXtildeY}.
Using~\eqref{eq:kaehler:commutatorOnP} (applied to 
\( \tilde{X} \) instead of \( \bar{X} \)) and the identity
\begin{equation}
\difFracAt{}{t}{0} X \Bigl(\flow^{Y(\phi)}_t \circ \phi\Bigr)
= \difFracAt{}{t}{0} \omega^\sharp \bigl( 
(\flow^{Y(\phi)}_t)_* \phi_* \dif^c \varrho_X \bigr)
= - \omega^\sharp \bigl( \difLie_{Y(\phi)} 
\dif^{c_\phi} (\phi_* \varrho_X) \bigr),
\end{equation}
we find
\begin{equation}\label{infinitesimal1}
\begin{split}
\commutator*{\tilde{X}}{\tilde{Y}}_\phi \circ \phi^{-1}
&= \commutator*{X(\phi)}{Y(\phi)} \\ 
&\quad+ \omega^\sharp \bigl( \difLie_{Y(\phi)} 
\dif^{c_\phi} (\phi_* \varrho_X) \bigr) - 
\omega^\sharp \bigl( \difLie_{X(\phi)} 
\dif^{c_\phi} (\phi_* \varrho_Y) \bigr) \\
&= - X_{\poisson{\phi_* \varrho_X}{\phi_* \varrho_Y}}, 
\end{split}\end{equation}
where we used~\eqref{eq:kaehler:commutatorXphiYphi} 
in the last step. This completes the proof 
of~\eqref{eq:kaehler:commutatorTildeXtildeY}.

Now, we can calculate the curvature \( \Omega \) 
of \( \theta \):
\begin{equation}\label{infinitesimal2}
\begin{split}
\Omega_\varphi(X \circ \varphi, Y \circ \varphi)
&= \Omega_\varphi(\tilde{X}_\varphi, \tilde{Y}_\varphi) \\
&= \tilde{X}_\varphi \bigl(\theta(\tilde{Y})\bigr) 
- \tilde{Y}_\varphi \bigl(\theta(\tilde{X})\bigr) - 
\theta_\varphi\bigl(\commutator{\tilde{X}}{\tilde{Y}}\bigr)\\
&\quad- \commutator*{\theta_{\varphi} \bigl(
\tilde{X}\bigr)}{\theta_{\varphi} \bigl(\tilde{Y}\bigr)} \\
&= \I \omega^\sharp \Bigl(\dif \difLie_{X(\varphi)} (\varphi_* \varrho_Y)
- \dif \difLie_{Y(\varphi)} (\varphi_* \varrho_X)\Bigr) \\
&\quad+ X_{\poisson{\varphi_* \varrho_X}{\varphi_* \varrho_Y}} - 
\commutator{X_{\varphi_* \varrho_X}}{X_{\varphi_* \varrho_Y}},
\end{split}\end{equation}
where we recall the convention~\eqref{eq:kaehler:bracket} 
of the Lie bracket on the Klein pair (which is used in 
the last equality). The imaginary part vanishes due to~\eqref{eq:kaehler:XphiAppliedToY}, so we conclude 
that \( \theta \) is flat.

Next, we prove the second statement of the proposition.
Let \( Z^* \) be the vector field on \( \SectionSpaceAbb{P} \) 
induced by the \( j \)-holomorphic vector field \( Z \) 
on \( M \), \ie, \( Z^*_\phi = \tangent \phi \circ Z \).
First, note that if \( X \) is a symplectic vector field, then
\begin{equation}
\difLie_{Z^*} \theta (X^*) = 
\difLie_{Z^*} \bigl(\theta (X^*)\bigr) - 
\theta \bigl(\commutator{Z^*}{X^*}\bigr) = 0,
\end{equation}
because \( \theta(X^*) = X \) is constant on 
\( \SectionSpaceAbb{P} \) and because the left and 
right actions of \( \DiffGroup(M, \omega) \) and of 
\( \DiffGroup(M, j) \) on \( \SectionSpaceAbb{P} \) 
commute. Thus, in order to show that the Lie derivative 
of \( \theta \) with respect to a holomorphic vector 
field vanishes, it suffices again to test this only
against the horizontal vector field \( \tilde{X} \).
A similar calculation as above shows that
\begin{equation}
	\commutator*{Z^*}{\tilde{X}}_\phi \circ \phi^{-1} = \difFracAt{}{t}{0} X\bigl(\phi \circ \flow^{Z}_t\bigr)
		= - \omega^\sharp \phi_* \dif^c \bigl(\difLie_Z \varrho_X\bigr).
\end{equation}
Since \( \difLie_{\omega^\sharp (\dif^c f)} \omega = \dif \dif^c f \), we find \( \theta_\phi \bigl(\commutator{Z^*}{\tilde{X}}\bigr) = \I \omega^\sharp \bigl(\dif \phi_* \difLie_Z \varrho_X\bigr) \).
Thus,
\begin{equation}
	\difLie_{Z^*} \theta (\tilde{X}) = \difLie_{Z^*} \bigl(\theta (\tilde{X})\bigr) - \theta \bigl(\commutator{Z^*}{\tilde{X}}\bigr) = 0.
\end{equation}
This finishes the proof.
\end{proof}

In finite dimensions, every flat Cartan geometry is locally
isomorphic, as a Cartan geometry, to a Klein bundle 
\( A \to G \slash A \) for a Lie group \( A \), endowed 
with its Maurer--Cartan form, 
see \parencite[Theorem~5.5.1]{Sharpe1997}.
However, this integration result does not hold in infinite 
dimensions. The flat Cartan bundle 
\( (\SectionSpaceAbb{P}, \theta) \) is the best possible 
\textquote{integration} of the Lie algebra 
\( \VectorFieldSpace(M, \omega) \oplus \I \,
 \HamVectorFields(M, \omega) \).

\begin{remark}[Associated affine connection]
\label{rem:kaehler:affineConnectionCurvature}	
Since the curvature of the Cartan connection vanishes, 
the curvature of the associated affine connection (seen as 
a horizontal \( 2 \)-form on \( \SectionSpaceAbb{P} \) 
with values in \( \VectorFieldSpace(M, \omega) \)) 
is given by
\begin{equation}
R = \frac{1}{2} \wedgeLie*{\Im \theta}{\Im \theta}.
\end{equation}
In other words,
\begin{equation}\begin{split}
R_\phi(X \circ \phi, Y \circ \phi) \circ \phi^{-1}
= \commutator*{X_{\phi_* \varrho_{\phi, X}}}{Y_{\phi_* \varrho_{\phi,Y}}}
= X_{\poisson{\phi_*\varrho_{\phi, X}}{\phi_* \varrho_{\phi, Y}}}.
\end{split}\end{equation}
We may also view \( R \) as a \( (3,1) \)-tensor on 
\( \SectionSpaceAbb{K} \) and then under the 
isomorphism~\eqref{eq:kaehler:tangentBundleIso} obtain
\begin{equation}
R_{\sigma}(\rho_1, \rho_2, \rho_3) 
= - \poisson*{\poisson{\rho_1}{\rho_2}_{\sigma}}{\rho_3}_{\sigma},
\qquad \sigma \in \SectionSpaceAbb{K}, \rho_1, \rho_2, \rho_3 \in
\sFunctionSpace_0(M).
\end{equation}
This gives an alternate proof of the expression for the curvature of 
\parencite[Theorem~1]{Donaldson1999a}.
\end{remark}
\begin{remark}[Complex structure on 
\( \VectorFieldSpace (M, \omega) \oplus \I \, 
\HamVectorFields(M, \omega) \)] 
Since the situation is the same at any 
\(\phi \in \SectionSpaceAbb{P}\), it suffices to choose 
\( \phi = \id \). Let $X\in \VectorFieldSpace (M, \omega)$. 
As \( X \contr \omega \) is closed,
we can write, using the Hodge decomposition,
\begin{equation*}
X \contr\omega = \alpha_X + \dif\tau_X,
\end{equation*}
with \(\alpha_X\) a harmonic 1-form and \(\tau_X\) 
a smooth function. On a K\"ahler manifold, we have 
the decomposition for the complexification of the 
space \(\mathbb H^1\) of harmonic \( 1 \)-forms
\begin{equation}
\mathbb H^1\otimes \mathbb C = 
\mathbb H^{1,0} \oplus \overline{\mathbb H^{1,0}},
\end{equation}
and \(\mathbb H^1\) can be identified with 
\(\mathbb H^{1,0}\). Hence,
\begin{equation}
\VectorFieldSpace (M, \omega) \oplus 
\I \, \HamVectorFields(M, \omega) \cong 
\omega^\sharp\mathbb H^{0,1} \oplus 
(\HamVectorFields(M, \omega) \otimes\mathbb C).
\end{equation}
Thus, \(\VectorFieldSpace (M, \omega) \oplus 
\I \, \HamVectorFields(M, \omega)\) has a natural 
complex structure. Of course, the complex structure 
of \(\omega^\sharp\mathbb H^{0,1}\) is induced 
from \(j\) as its \((0,1)\)-forms.
\end{remark}

\begin{remark}\label{Rem0}
In the subsections below, we consider different 
symplectic structures on \(\mathcal I\) which 
are preserved by the natural \(\DiffGroup(M,\omega)\) 
action. Since \(\mathcal I\) is contractible 
there is a (non-equivariant) momentum map 
\( \SectionSpaceAbb{J} \) for the action 
of $\DiffGroup(M,\omega)$ on those symplectic 
structures on \(\mathcal I\), see,
\eg, \parencite[Proposition~2.1]{DiezRatiuSymplecticConnections}.
Due to the simple form of the Lie algebra 
\( \mathfrak{a} = \VectorFieldSpace (M, \omega) \oplus 
\I \, \HamVectorFields(M, \omega) \), the condition 
that the momentum map is \( \mathfrak{a} \)-invariant 
is equivalent to the equivariance of the momentum map 
for the action of the subgroup of Hamiltonian 
diffeomorphism, see 
\cref{prop:cartan:momentumMapEquivarianceComplexCase}. 
This equivariance property holds in the cases considered 
below. Thus, taking $M$ in the statement of 
\cref{prop:cartan:constFutaki} to be \(\mathcal I\), 
\cref{prop:cartan:constFutaki} can be applied in the 
subsections below.
\end{remark}
\begin{remark}[Futaki character]\label{Rem1} 
By the arguments in the previous section, 
the Futaki invariant descends from \(\mathcal P\)
to the base \(\mathcal K\). Furthermore, the tangent 
bundle \(\TBundle \mathcal P\) has a natural
horizontal distribution consisting of all
\(X \in \TBundle \mathcal P\) such 
that $\theta(X)$ is purely imaginary in
\( \VectorFieldSpace (M, \omega) \oplus \I \, 
\HamVectorFields(M, \omega) \). 
Thus, by~\eqref{eq:futaki:generalizedFutakiInvariant}, 
the Futaki invariant \(\widetilde{\mathcal F}_\zeta : 
\mathcal{P} \to \mathbb{R}\) 
on the connected component of 
\(\id_M\in \mathcal P\) can be expressed as
\begin{equation}
\widetilde{\mathcal F}_\zeta (\id_M) = 
- \langle \SectionSpaceAbb{J}(j),\omega^\sharp 
\dif \varrho_\zeta\rangle,
\end{equation}
where \(j=\chi(\id_M)\) and \( \dif\dif^c\varrho_\zeta 
= \difLie_\zeta \omega \). Note that the momentum map 
\( \SectionSpaceAbb{J} \) is paired with a 
\emph{Hamiltonian} vector field
and hence the Futaki invariant only depends on the 
momentum map for the subgroup of Hamiltonian 
diffeomorphisms.
\end{remark}
\begin{remark}
\Cref{auto} applies for the following reason. 
Suppose \(\chi(\id_M)=j\). Then \(\mathfrak a_j\) 
is the Lie algebra of all holomorphic vector fields 
with respect to \(j\). For every 
\( \phi \in \SectionSpaceAbb{P} \), we have
\( \ker \tangent_\phi \chi = \phi_\ast {\mathfrak a}_j \) 
because \( \phi \) is a biholomorphism  of \((M,j)\) 
to \((M,\phi_\ast j)\). On the other hand, 
\( \phi \ldot \LieA{a}_j \) consists of all tangent 
vectors to the curves \( \phi(c(t))\) at \(t=0\), where
the tangent vector of \(c(t)\) at \(t=0\) lies in 
\( {\mathfrak a}_j \). Thus \( \ker \tangent_\phi \chi \) 
and \( \phi \ldot \LieA{a}_j \) coincide.
\end{remark}
\begin{remark}[Geodesics]
Recall that the space \(\SectionSpaceAbb{K}\) of all 
K\"ahler forms in the cohomology class \([\omega]\) 
consists of K\"ahler forms that can be written as
\( \omega_\varrho \defeq \omega + \dif\dif^c\varrho \).
Tangent vectors are smooth functions \(u\) modulo constant 
functions, as discussed before~\eqref{eq:kaehler:tangentP}.
It is convenient to normalize \(u\) to satisfy 
\(\int_M u \omega_\varrho^n = 0\). The 
 norm-squared of the canonical Riemannian 
metric of \( \SectionSpaceAbb{K} \) is given by 
\(u \mapsto \int_M u^2 \omega_\varrho^n\). Then, a 
sufficiently smooth curve \(\varrho(t)\) in 
\( \SectionSpaceAbb{K} \), \(a\le t \le b\),  
is a metric geodesic if
\begin{equation}
\ddot{\varrho} - \norm{\bar{\partial}\dot{\varrho}}^2 = 0,
\end{equation}
see \parencite[p.~17]{Donaldson1999a}. By \cref{prop:cartan:riemannianConnection,prop:cartan:geodesicsAsProjections}, 
these geodesic curves arise exactly as projections of integral 
curves of certain \( \theta \)-constant vector fields on 
\( \SectionSpaceAbb{P} \). That is, metric geodesics 
on \( \SectionSpaceAbb{K} \) correspond to geodesics of 
the Cartan connection.
\end{remark}

\subsection{Constant scalar curvature 
K\"ahler metrics}
In this subsection, we consider the problem of finding 
constant scalar curvature metrics in the K\"ahler class 
$\mathcal K$. After the resolution of the existence 
problem of K\"ahler-Einstein metrics in the
negative case by \textcite{Aubin1976} and 
\textcite{Yau1978} and the zero case by \textcite{Yau1978}, 
it was conjectured by \textcite{Yau1993} that
the existence of general cscK metrics should be 
characterized by an algebraic condition
related to geometric invariant theory (GIT for 
short). A notion, called K-stability, inspired by 
ideas from GIT, was formulated by \textcite{Tian1997} 
for the study of the existence of positive K\"ahler-Einstein 
metrics on Fano manifolds, using special degenerations, and 
by \textcite{Donaldson2002} for the general cscK metrics
on polarized K\"ahler manifolds, using test configurations. 
Recall that a pair $(M,L)$ of a compact complex manifold $M$
and an ample line bundle $L$ is called a polarized manifold.
A test configuration (or special degeneration) is a 
$\mathbb C^\ast$-equivariant degeneration of $(M,L)$.
K-stability is tested by the sign of the 
Donaldson-Futaki invariant (or generalized Futaki invariant, 
\ie, an algebraic reformulation of the Futaki invariant 
valid for singular varieties) of the central fiber of 
the degeneration. The Yau--Tian--Donaldson conjecture 
asserts that the existence of cscK metrics on a polarized 
K\"ahler manifold should be equivalent to 
K-polystability of the polarized manifold. This conjecture 
has been confirmed by \textcite{ChenDonaldsonSun2015a,Tian2015} 
for the positive K\"ahler-Einstein 
case on Fano manifolds and later by 
\parencite{DatarSzekelyhidi2016,ChenSunWang2018,
BermanBoucksomJonsson2021,Li2022,LiuXuZhuang2022}
using other methods. The key tool of the proofs of 
\parencite{ChenDonaldsonSun2015a,Tian2015,
DatarSzekelyhidi2016,ChenSunWang2018} 
is the theory of Gromov--Hausdorff convergence.
On the other hand, the proof of 
\textcite{BermanBoucksomJonsson2021} is based 
on the variational method (not using 
Gromov--Hausdorff convergence) 
and the notion called uniform stability was assumed.
Uniform stability implies that the automorphism group 
is discrete. For a non-discrete automorphism group $G$, 
\textcite{Li2022} relaxed the condition to $G$-uniform stability.
In \parencite{LiuXuZhuang2022}, it is shown that when $G$ 
contains the maximal torus of the full automorphism group, 
then $G$-uniform stability is equivalent to K-stability 
and, as a conclusion, \parencite{BermanBoucksomJonsson2021,Li2022}
give an alternate proof for positive K\"ahler-Einstein metrics.

For an arbitrary K\"ahler manifold, \textcite{BermanDarvasLu2020} 
have proven that the existence of a cscK metric implies 
K-polystability. However, the converse is still 
open.

The existence problem of cscK metrics fits into the 
symplectic framework as shown by \textcite{Fujiki1992,Donaldson1997}.
In particular, their work showed that cscK metrics 
are the zeros of a momentum map on an infinite-dimensional 
symplectic manifold. This picture, combined with the usual 
finite-dimensional Kempf--Ness theory, gave a
philosophical motivation for the Yau--Tian--Donaldson 
conjecture, and, moreover, is generally regarded to 
give the right strategy for its proof.
However, until now, no proper infinite-dimensional 
framework has been available to put these ideas into 
a rigorous setting.

Now we apply \cref{prop:cartan:constFutaki} to 
the Fujiki--Donaldson picture. In this case, we consider a 
compact symplectic manifold \((M,\omega)\) and the infinite 
dimensional space \(\mathcal{I}\) of \( \omega \)-compatible 
complex structures endowed
with a suitable K\"ahler structure described below. 
The scalar curvature is the momentum map with respect 
to the action of the Hamiltonian diffeomorphism
group. The same conclusion is also obtained in 
\parencite[Corollary~5.7]{DiezRatiuSymplecticConnections} 
under the sign conventions of the momentum map~\eqref{momentum1} 
and the sign convention of the symplectic
form and the complex structure as described in 
\cref{prop:cartan:kempfNessConvexity}.
Thus, we can apply \cref{prop:cartan:constFutaki} and 
\cref{prop:cartan:kempfNess:positive} 
to obtain the Futaki invariant and the Mabuchi 
K-energy as follows. 

We identify the space of Hamiltonian vector fields with 
the space of average-zero Hamiltonian functions 
\(\sFunctionSpace_0(M)\). Using the symplectic structure 
of \(\mathcal I\) described below, the momentum map image
\(J(j)\) of \(j \in \mathcal I\) is expressed as the 
\(L^2\)-pairing of the scalar curvature \(S_j\) of the 
K\"ahler manifold \((M,\omega,j)\) with the Hamiltonian functions.
In \cref{Rem1}, the Hamiltonian vector field 
\( \omega^\sharp \dif\varrho_\zeta \)
is identified with \( \varrho_\zeta - 
\int_M \varrho_\zeta \, \mu_\omega/\int_M \mu_\omega \).                                                                                   
Thus, defining 
\begin{equation}
\label{eq:kahler:averagezero}
u_0 \defeq \int_M u \mu_\omega\Big/\!\int_M \mu_\omega 
\in \mathbb{R}
\end{equation}
for \(u \in \sFunctionSpace(M)\) and noting that 
\(\kappa_{\mathfrak a}\) is the negative of the 
\(L^2\)-pairing, \cref{prop:cartan:constFutaki} and  
\cref{Rem1} yield
\begin{equation}
 \widetilde{\mathcal F}_\zeta (\id_M) = 
 - \int_M S_j\, \left(\varrho_\zeta 
 - {(\varrho_\zeta)}_0\right)\mu_\omega
= - \int_M (S_j - {(S_j)}_0)\,\varrho_\zeta\ \mu_\omega. 
\label{obst1}
\end{equation}
This recovers the Futaki invariant as defined in 
\parencite{Futaki1983}.
If there is an extremal metric such that \(\grad S\) 
is a non-zero holomorphic vector field, then 
\begin{equation}
-\int_M (S-S_0)S \omega^n = 
-\int_M (S-S_0)^2 \omega^n < 0,
\end{equation}
where \(n=\dim M/2\). This indicates that if a cscK 
metric does not exist, then \((M,L)\) 
should be K-unstable. Of course, an extremal K\"ahler
metric may not exist, and this is not a rigorous proof.

To describe the symplectic structure of $\mathcal{I}$, 
assume that \(j \in \mathcal{I}\) acts on the cotangent bundle
rather than the tangent bundle. Fixing \(j \in \mathcal{I}\),
we decompose the complexified cotangent bundle into 
holomorphic and anti-holomorphic parts,
\ie, the \(\pm \sqrt{-1}\)-eigenspaces of \(j\):
\begin{equation}
\CotBundle M\otimes\bfC = \TBundle_j^{\ast\prime}M \oplus 
\TBundle_j^{\ast\prime\prime}M,  \qquad 
\TBundle_j^{\ast\prime\prime}M = \overline{\TBundle_j^{\ast\prime}M}.
\end{equation}
Taking an arbitrary \(j^{\prime} \in \mathcal I\), we 
also have the decomposition with respect to \(j^{\prime}\):
\begin{equation}
\CotBundle M\otimes\bfC = \TBundle_{j^{\prime}}^{\ast\prime}M \oplus 
\TBundle_{j^{\prime}}^{\ast\prime\prime}M,  \qquad 
\TBundle_{j^{\prime}}^{\ast\prime\prime}M = 
\overline{T_{j^{\prime}}^{\ast\prime}M}.
\end{equation}
If \(j^{\prime}\) is sufficiently close to \(j\), then 
\(\TBundle_{j^{\prime}}^{\ast\prime}M\) can be expressed as 
a graph over \(\TBundle_j^{\ast\prime}M\), namely, 
\begin{equation}
	\TBundle_{j^{\prime}}^{\ast\prime}M = \set*{\alpha + 
\mu(\alpha) \given \alpha \in \TBundle_j^{\ast\prime}M}
\end{equation}
for some homomorphism \( \mu \) of \(\TBundle_j^{\ast\prime}M\) 
into \(\TBundle_j^{\ast\prime\prime}M\), \ie, 
\begin{equation}
\mu \in \Gamma \bigl(\mathrm{Hom}(\TBundle_j^{\ast\prime}M, 
\TBundle_j^{\ast\prime\prime}M)
\bigr) 
\cong  \Gamma \bigl(\TBundle_j^{\prime}M \otimes 
\TBundle_j^{\ast\prime\prime}M\bigr) \cong 
\Gamma\bigl(\TBundle_j^{\prime}M\otimes \TBundle_j^{\prime}M\bigr),
\end{equation}
where the second isomorphism is given by the K\"ahler 
metric defined by the pair \((\omega, j)\). This can be
expressed in the notation of tensor calculus with indices as 
\begin{equation}
\mu^i{}_{\bark} \mapsto g^{j\bark}\mu^i{}_{\bark}=: \mu^{ij}, 
\end{equation}
where we chose a local holomorphic coordinate system
\((z^1, \ldots, z^n)\) on \((M,j)\) and \( \omega \) is written 
as \( \omega = \sqrt{-1}\ g_{i\barj} dz^i \wedge dz^\barj \).
Then one can see that \( \mu \) lies in the symmetric part 
\(\Gamma(\mathrm{Sym}(\TBundle_j^{\prime}M\otimes \TBundle_j^{\prime}M))\) 
of \(\Gamma(\TBundle_j^{\prime}M\otimes \TBundle_j^{\prime}M)\),
see, \eg, \parencite{Futaki2006}.
Hence, the tangent space \(\TBundle_j\mathcal{I}\) to 
\(\mathcal{I}\) at \(j\) is a subspace of 
\(\mathrm{Sym}(\TBundle_j^{\prime}M\otimes \TBundle_j^{\prime}M)\).
Then the \(L^2\)-inner product on 
\(\mathrm{Sym}(\TBundle_j^{\prime}M\otimes \TBundle_j^{\prime}M)\)
endows \(\mathcal{I}\) with a K\"ahler structure. 
Taking its imaginary part, we obtain a symplectic 
structure on \(\mathcal{I}\).

Since \(\kappa_{\mathfrak a}\) is minus the \(L^2\) 
inner product in this case, the Kempf--Ness functional 
defined in \cref{prop:cartan:kempfNess} takes the form:
\begin{equation}\label{Kempf1}
E(\varrho)=- \int_0^1 d t \int_M \dot{v}_t\left(S\left(
\omega_{v_t}\right)-{S}_0\right) \omega_{v_t}^n,
\qquad \varrho \in \sFunctionSpace_0(M),
\end{equation}
where \(v_{t}\), \(0 \leq t \leq 1\), is a 
smooth path between \(0\) and \( \varrho \)
such that \(\omega_{v_{t}} \defeq \omega + 
\I\partial\barpartial \, v_t >0\), \(S(\omega_{v_t})\) 
is the scalar curvature of the K\"ahler form \(\omega_{v_t}\),
and \(S_0\) is its average, \cf~\eqref{eq:kahler:averagezero}.
In this way, we recover the K-energy functional of 
\textcite{Mabuchi1987}, which is convex along geodesics.

\Textcite{Tian1994} gave a more explicit expression of 
the K-energy:
\begin{equation}\label{ChenTian}
E(\varrho)=\operatorname{Ent}(\varrho)+
S_0 \operatorname{AM}(\varrho)-
n \operatorname{AM}_{\operatorname{Ric}(\omega)} (\varrho), 
\end{equation}
where
\begin{equation}	
\operatorname{Ent}(\varrho)= 
\int_M \log \left(\omega_\varrho^n / \omega^n\right) 
\omega_\varrho^n,
\end{equation}
\begin{equation}
\operatorname{AM}_\omega (\varrho)=
\operatorname{AM}(\varrho)=\frac{1}{(n+1)} 
\sum_{j=0}^n \int_M \varrho \omega_\varrho^j \wedge \omega^{n-j}, 
\end{equation}
and 
\begin{equation}
\operatorname{AM}_\chi(\varrho)=\frac{1}{n V} 
\sum_{j=1}^{n} \int_M \varrho \, \chi \wedge 
\omega_\varrho^{j-1} \wedge \omega^{n-j},
\end{equation}
for closed \((1,1)\)-forms \( \chi \). 
\(\operatorname{Ent}(\varrho)\) is called the entropy 
and \(\operatorname{AM}(\varrho)\) is called the 
Monge--Amp\`ere energy. See also \parencite{FutakiNakagawa2001} 
for an intrinsic derivation of~\eqref{ChenTian}.

Mabuchi also introduced geodesics in \(\mathcal{K}\)
and showed the convexity of the K-energy along 
the geodesics (\parencite{Mabuchi1987a}, see also 
\parencite{Semmes1992,Donaldson1999a}). The slope of 
\(E(\varrho)\) was studied by \textcite{Donaldson1999a} 
and \textcite{ChenTang2007}. \Textcite{PhongSturm2007} 
showed that for each test configuration there corresponds 
a weak geodesic ray. \Textcite{BoucksomHisamotoJonsson2017} showed 
that the slopes of each term of~\eqref{ChenTian} is 
expressed by an algebraic invariant related to the 
minimal model program in algebraic geometry. The slopes
obtained in this way are called non-Archimedean 
functionals. The uniform stability assumed in
\parencite{BermanBoucksomJonsson2021,Li2022} is expressed 
in terms of non-Archimedean
functionals. It is known that the slope of the K-energy 
on a Fano manifold is the Donaldson-Futaki
invariant when the central fiber is reduced.

Let \(G\) be the maximal compact subgroup of the 
automorphism group of \(M\). Then \(\mathfrak a_j\) is equal to 
\(\mathfrak g\otimes \mathbb C\). Applying \cref{extremal1}, 
we obtain the following intrinsic bilinear form.

\begin{theorem}
Choose any \(G\)-invariant K\"ahler metric. For any 
\(X \in \mathfrak g\), we consider the Hamiltonian function 
\(u_X\) with respect to the K\"ahler form \(\omega\) with 
normalization \(\int_M u_X \omega^n = 0\). 
The \(L^2\)-inner product of the normalized Hamiltonian 
functions is independent of the choice of the \(G\)-invariant 
metric and also independent of the choice of the maximal 
compact subgroup \(G\).
\end{theorem}

This recovers the bilinear form obtained in 
\parencite{FutakiMabuchi1995}. Furthermore, the extremal 
element defined 
in~\eqref{eq:cartan:extremalElements:extremalCondition} 
is the extremal K\"ahler vector field, 
see \parencite[Theorem~3.3.3]{Futaki1988}, 
\parencite{FutakiMabuchi1995}.

\begin{remark}[Calabi operator]\label{momentum2}
The action \(\rho(\xi)\) is given by the infinitesimal action 
of \(\xi \in \mathfrak a\) on \( \SectionSpaceAbb{I} \) at 
\( j \in \SectionSpaceAbb{I} \), where we assume 
\( \rho(\xi_1 + \I \xi_2) = \rho(\xi_1) + j \rho(\xi_2) \).
Let us take \(\xi=X_u\), the Hamiltonian vector field defined 
by the smooth function \(u\). Then, writing 
\(X_u = X^{\prime} + X^{\prime\prime} \in 
\sSectionSpace(\TBundle^\prime M \oplus 
\TBundle^{\prime\prime}M)\),
the infinitesimal action \( \rho \) at \(j\) is expressed as
\begin{equation}
\difLie_X j = 2\sqrt{-1} \nabla_j^{\prime\prime} X^{\prime}
	- 2\sqrt{-1} \nabla_j^{\prime} X^{\prime\prime},
\end{equation}
where $\nabla^{\prime}_j$ and $\nabla^{\prime\prime}_j$ are,
respectively, the $(1,0)$-part and $(0,1)$-part of the 
covariant derivative $\nabla_j$,
see \parencite[Lemma~2.3]{Futaki2006}.
If \( j_\varepsilon \) is a curve in \( \SectionSpaceAbb{I} \)
such that \( \difFracAt{}{\varepsilon}{0} 
j_\varepsilon 
= \difLie_{X_u}j\), where \(X_u\contr\omega = - \dif u\) 
with \( u \in \sFunctionSpace(M, \R) \), then 
the Calabi operator defined in~\eqref{eq:cartan:calabiOperator} 
takes here the form
\begin{equation}
	\difFracAt{}{\varepsilon}{0} S(j_\varepsilon) 
= - \nabla^{\prime\prime\ast} 
\nabla^{\prime\prime\ast} \nabla^{\prime\prime} 
\nabla^{\prime\prime}u.
\end{equation}
The right-hand side is the operator used by Calabi in his 
study of extremal K\"ahler metrics \parencite{Calabi1985}.
However, in Calabi's paper, the variation of the scalar 
curvature was considered with the complex structure 
\(j\) fixed and the K\"ahler form \( \omega \) varying 
in the fixed cohomology class.
\end{remark}

\subsection{Perturbed scalar curvature}
In this subsection, we perturb the scalar curvature of 
compact K\"ahler manifolds by incorporating it with higher 
Chern forms, based on \parencite{Futaki2006}.
We show that the  perturbed scalar curvature becomes a 
momentum map, with respect to a perturbed symplectic structure, 
on the space $\mathcal I$ of all complex structures on a
fixed symplectic manifold \((M, \omega)\),
\(\dim M = 2n\). This extends results of 
Fujiki and Donaldson on the unperturbed case.

We modify the K\"ahler structure on $\mathcal I$ as follows. 
Denote by $c_n$ the $GL(n,\mathbb C)$-invariant polynomial 
corresponding to the $n$-th Chern class in the Chern-Weil theory,
namely, 
\begin{eqnarray*}
&&\det(t_1A_1 + \cdots + t_n A_n) = \\
&&\cdots + n!\ t_1 \cdots t_n c_n(A_1,\ldots,A_n) + \cdot \, ,
\end{eqnarray*}
where
\begin{equation}
c_n(A,\ \ldots\ ,A) = \det A. 
\end{equation}
Note also that the coefficient of $t^k$ in $\det(I + tA)$ 
corresponds to the $k$-th Chern class; in particular, 
the coefficient of $t$ is the trace.
Fix a small $t \in \bfR$. For $\mu, \nu \in 
\Gamma (\mathrm{Sym}(T'M \otimes T'M))\ \cong T_j \mathcal I$, 
we define
\begin{equation*}
(\mu, \nu)_t = 
n\int_M c_n \Bigl(\ \mu^i{}_{\barl}\ \overline{\nu}_{jk}\ 
\frac{\sqrt{-1}}{2\pi}dz^k\wedge dz^{\barl}, \omega\otimes I + 
t\frac{\sqrt{-1}}{2\pi}\Theta, \ldots, \omega\otimes I 
+ t\frac{\sqrt{-1}}{2\pi}\Theta\Bigr) 
\end{equation*}
where 
$\Theta = \barpartial(g^{-1}\partial g)$\ \ is the curvature 
matrix of the Levi-Civita connection of $(M,\omega,j)$.
Here, the objects that are plugged into \( c_n \) are 
\( 2 \)-forms with values in the endomorphism bundle and 
then \( c_n \) is applied to these endomorphisms.

By the remark above, when $t = 0$,
\begin{equation}
(\mu, \nu)_0 = 
L^2 \ \mathrm{inner\ product\ of\ } \mu\ \mathrm{and}\ \nu; 
\end{equation}
this is used in the previous subsection to define the 
symplectic structure on $\mathcal I$. Thus, for small $t$, 
the imaginary part of $(\mu, \nu)_t$ defines a perturbed 
symplectic structure on $\mathcal I$.
\begin{definition}
The $t$-perturbed scalar curvature (or simply perturbed 
scalar curvature) $S(j,t)$ is defined by
\begin{eqnarray*}
S(j,t)\ \omega^n &= & 
c_1(j) \wedge \omega^{n-1} + tc_2(j)\wedge \omega^{n-2}
+ \cdots + t^{n-1} c_n(j) \\
&=& \frac1t \Bigl(\det \bigl(\omega \otimes I + 
t\frac{\sqrt{-1}}{2\pi}\Theta\bigr) - \omega^n\Bigr),
\end{eqnarray*}
where $c_i(j)$ is the $i$-th Chern form of $(M, \omega, j)$. 
\end{definition}
It is shown in \parencite{Futaki2006} that the action of 
the Hamiltonian symplectomorphism group $\HamDiffGroup(M,\omega)$
on the K\"ahler manifold $(\mathcal I, (\cdot,\cdot)_t)$
admits the equivariant momentum map
\begin{equation}
J_t : j \ni \mathcal I \mapsto -S(j,t) \in  C^{\infty}(M)/\bfR
\end{equation}  
with the sign convention of~\eqref{momentum1}. 
Obviously,
\begin{equation}
J_t^{-1}(0) = \set*{j \in \mathcal I\given S(j,t)\ \mathrm{is\ constant}}.
\end{equation}
In other words, the zeros of the momentum map are those
$j$'s for which
\begin{equation}
c_1(j)\wedge \omega^{n-1} + t c_2(j) \wedge \omega^{n-2} 
+ \cdots + t^{n-1} c_n(j) \end{equation}
is a harmonic form, \ie, a multiple  of $\omega^n$. 
Using \cref{prop:cartan:constFutaki} as well as 
the arguments in \parencite{Bando2006},
we obtain the Futaki invariant in this perturbed case
\begin{equation}
\widetilde{\mathcal F}_\zeta (\omega, 1) = 
- \int_M (S(j,t) - {S(j,t)}_0)\,\varrho_\zeta\ \mu_\omega. 
\label{obst2}
\end{equation}
This recovers the result of \parencite{Futaki2006}.
The coefficient of $t^k$ in 
\( \widetilde{\mathcal F}_\zeta \) of~\eqref{obst2} 
can be re-written as
\begin{equation}
f_k(\zeta) = \int_M \difLie_\zeta F_k \wedge \omega^{n-k+1}
\end{equation}
where, writing $Hc_k(\omega)$ for the harmonic part 
of $c_k(\omega)$, we have
\begin{equation}
c_k(\omega) - Hc_k(\omega) = 
\sqrt{-1} \partial \barpartial F_k 
\end{equation}
with $ F_k \in \Omega^{k-1,k-1}(M)$, see 
\parencite{Futaki2005,Futaki2006} and 
\parencite{Bando2006} for details.
This $f_k$ is an obstruction 
to the existence of $\omega \in \mathcal{K}$
for which the $k$-th Chern form $c_k(\omega)$ is
harmonic. Note that $f_1$ coincides with the cscK 
obstruction in the previous subsection. This result 
is due to \textcite{Bando2006}.  
In \parencite{Futaki2008}, the perturbed version of
extremal K\"ahler metrics and the
Calabi-Lichnerowicz-Matsushima type decomposition theorem
has been studied. The Kempf--Ness functional can be defined 
as in~\eqref{Kempf1}, but not much has been
done so far in this perturbed case.


\subsection{Deformation quantization}

A deformation quantization is a formal associative 
deformation of a Poisson algebra 
$(\sFunctionSpace(M), \poissonDot)$
into the space $\sFunctionSpace(M)[[\nu]]$ of 
formal power series in $\nu$ with a composition 
law $\ast$, called the {\it star product},
with the following property. The constant function 
$1$ is a unit and if 
we write for $f,\ g  \in \sFunctionSpace(M)$ 
\begin{eqnarray*}
f \ast g = \sum_{r=0}^\infty C_r(f,g) \nu^r,
\end{eqnarray*}
then $\ast$ is required to satisfy
\begin{eqnarray*}
C_0(f,g) = f g, \qquad C_1(f,g) - C_1(g,f)  = \{f,g\},
\end{eqnarray*}
where the $C_r$'s are required to be bidifferential operators.

For symplectic manifolds, the existence of star products 
was shown by \textcite{DeWildeLecomte1983}, \textcite{Fedosov1994}, 
and \textcite{OmoriMaedaYoshioka1991}. For general Poisson 
manifolds, the existence of star products was shown by 
\textcite{Kontsevitch2003}. In this subsection, we are concerned 
with the star product constructed by Fedosov.

A star product on a compact symplectic manifold 
$(M,\omega)$ of dimension $2n$ 
is called \emph{closed} (in the sense of 
\textcite{ConnesFlatoSternheimer1992}) if 
\begin{eqnarray*}
\int_M F\ast H\ \mu_\omega = \int_M H \ast F\ \mu_\omega
\end{eqnarray*}
for all $F,\ H \in \sFunctionSpace(M)[[\nu]]$.






A \emph{symplectic connection} $\nabla$ on a compact 
symplectic manifold $(M,\omega)$ is a torsion free 
affine connection such that $\nabla\omega = 0$.
There always exists a symplectic connection on 
any symplectic manifold. Unlike the Levi-Civita 
connection on a Riemannian manifold, a symplectic 
connection is not unique on a symplectic manifold. 
Given two symplectic connections $\nabla$ and $\nabla^\prime$, let
\begin{equation}
S(X,Y) \defeq \nabla_X Y - \nabla^\prime_X Y.
\end{equation}
Then $\omega(S(X,Y),Z)$ is totally symmetric in $X$, $Y$, $Z$. 
Conversely, if $\nabla$ is a symplectic connection
and $\omega(S(X,Y),Z)$ is totally symmetric, then 
$\nabla^\prime \defeq \nabla + S$ is a 
symplectic connection.
Thus, on a symplectic manifold $(M,\omega)$, the 
space of symplectic connections, denoted by 
$\mathcal E (M,\omega)$, is an affine space modeled 
on the set of all smooth sections $\Gamma(S^3(T^\ast M))$ 
of symmetric covariant 3-tensors. Hence, we may identify 
$\mathcal E (M,\omega)$ with 
\begin{equation}
\mathcal E (M,\omega) \cong \nabla + \Gamma(S^3(T^\ast M)).
\end{equation}
On $\mathcal E (M,\omega)$ there is a natural symplectic 
structure $\Omega^{\mathcal E}$
whose value at $\nabla$ is given by
\begin{equation}\label{symp str}
 \Omega^{\mathcal E}_\nabla(\underline{A},\underline{B}) = 
\int_M \omega^{i_1j_1}\omega^{i_2j_2}\omega^{i_3j_3}
\underline{A}_{i_1i_2i_3}\,\underline{B}_{j_1j_2j_3}\ 
\mu_\omega \nonumber
\end{equation}
for $\underline{A},\ \underline{B} \in 
T_\nabla \mathcal E(M,\omega) \cong \Gamma(S^3(T^\ast M))$, 
see \parencite{CahenGutt2005}.
Since
$\Omega^{\mathcal E}_\nabla$ is independent of $\nabla$ 
we may omit $\nabla$ and write $\Omega^{\mathcal E}$. 
There is a natural action of the group 
$\DiffGroup(M,\omega)$ of symplectomorphisms 
on $\mathcal E(M,\omega)$, which is given for a 
symplectomorphism $\varphi$ by
\begin{equation}
(\varphi(\nabla))_X Y = 
\varphi_\ast \bigl(\nabla_{\varphi_\ast^{-1} X} 
(\varphi_\ast^{-1} Y)\bigr)
\end{equation}
for any $\nabla \in \mathcal E(M,\omega)$ and any 
smooth vector fields  $X$ and $Y$ on $M$.
This action preserves the symplectic structure 
$\Omega^{\mathcal E}$ on $\mathcal E(M,\omega)$. 
In particular, the group $\HamDiffGroup(M,\omega)$ 
of Hamiltonian diffeomorphisms acts on
$\mathcal E(M,\omega)$ as symplectomorphisms. 

Let $X_f$ be the Hamiltonian vector field on 
$M$ defined by a smooth function $f$ on $M$. 
Then the induced infinitesimal action of 
$- X_f$ on $\mathcal E(M,\omega)$ is computed as
\begin{equation}\label{tangent}
\begin{split}
(\difLie_{X_f}\nabla)_Y Z
&= [X_f,\nabla_Y Z] - \nabla_{[X_f, Y]} Z - \nabla_Y [X_f,Z] \\
&= R^\nabla(X_f,Y)Z + (\nabla\nabla X_f )(Y,Z).
\end{split} 
\end{equation}

For $\nabla \in \mathcal E(M,\omega)$, we define 
the \emph{Cahen--Gutt momentum map} $J(\nabla)$ by
\begin{equation}  \label{CGmoment}
J(\nabla) = \nabla_p\nabla_q \mathrm{Ric}(\nabla)^{pq} 
 - 
\frac12 \mathrm{Ric}(\nabla)_{pq}\mathrm{Ric}(\nabla)^{pq}
+ \frac14 \mathrm{R}(\nabla,\omega)_{pqrs}
\mathrm{R}(\nabla,\omega)^{pqrs},
\end{equation}
where 
\begin{equation}\label{curvature}
\mathrm{R}(\nabla,\omega)(X,Y,Z,W) = \omega (R(X,Y)Z, W)
\end{equation}
and 
\begin{equation}
\mathrm{Ric}(X,Y) = - \tr ( Z \mapsto R(X,Z)Y).
\end{equation}

\begin{theorem}[\textcite{CahenGutt2005,DiezRatiuSymplecticConnections}] 
The function $J$ on $\mathcal E(M,\omega)$ given by~\eqref{CGmoment} is an equivariant momentum map
for the action of $\HamDiffGroup(M,\omega)$.
Moreover, the action of the group of symplectomorphisms 
on $\mathcal E(M,\omega)$ has a non-equivariant momentum map.
\end{theorem}
\noindent
Thus, the zeros of the momentum map are the symplectic 
connections $\nabla$ for which $J(\nabla)$ is constant.
\begin{remark}\label{Fact1} An important observation 
is that if the Fedosov star product is closed, then 
the Cahen--Gutt momentum $J(\nabla)$ is constant, see 
\parencite{Fuente-Gravy2015,Fuente-Gravy2016}.
\end{remark}

Now we assume that $M$ is a compact K\"ahler manifold 
and that $\omega$ is a fixed symplectic form. We set, 
as before, $\mathcal I$ to be the set of all integrable 
complex structures $j$ such that $(M,\omega, j)$ is a K\"ahler 
manifold. Let $lv : \mathcal I \to \mathcal E(M,\omega)$ 
be the \emph{Levi-Civita map} sending $j$ to the Levi-Civita 
connection $\nabla^j$ of the K\"ahler manifold $(M,\omega,j)$. 
Then $lv^\ast \Omega^{\mathcal E}$ gives a new symplectic 
structure on $\mathcal I$ if it is non-degenerate.
A condition for non-degeneracy, under suitable assumptions 
on the Ricci curvature, is given in 
\parencite[Proposition~17]{Fuente-Gravy2016}. In the following, 
we assume that $lv^\ast \Omega^{\mathcal E}$ is symplectic.
Then, by \cref{Rem0}, we can apply
\cref{prop:cartan:constFutaki} and recover the following result.
\begin{theorem}[{\parencites[Theorem~1]{Fuente-Gravy2016}{FutakiOno2018}}]
Let $(M, \omega)$ be a compact K\"ahler manifold. Then 
\begin{equation}
 \widetilde{\mathcal F}_\zeta (\omega, 1) = - \int_M (J(\nabla^j) - J(\nabla^j)_0)\,\varrho_\zeta\ \mu_\omega, \qquad j=\chi(\omega,1) \label{obst3}
\end{equation}
is independent of the choice of $\omega \in \mathcal K$. In particular, if $ \widetilde{\mathcal F}_\zeta (\omega, 1) $
is not identically zero as a function of $\zeta$, then there is no closed Fedosov star-product for
$\omega \in \mathcal K$.
\end{theorem}
\noindent
The last statement follows from \cref{Fact1}.

In \parencite{FutakiOno2018} it has been shown that 
for a real smooth function $f$, 
$\difLie_{X_f}\nabla^j = 0$ if and only if 
$\difLie_{X_f}j = 0$. In this case, $X_f$ is a holomorphic 
Killing vector field. Note also that the momentum map $J(\nabla)$ 
defined in~\eqref{CGmoment} is written in the K\"ahler case as
\begin{equation*}
J(\nabla) = \Delta S - \frac12 |\mathrm{Ric}|^2 
+ \frac14 |\mathrm{R}|^2.
\end{equation*}

If a vector field $Y$ generates a Hamiltonian isometric 
$S^1$-action, \ie, \( Y \contr \omega=-\dif v_Y \)
for the Hamiltonian function $v_Y$, then
adapting the cohomology formula of \textcite{Odaka2013} 
and \textcite{Wang2012a} to our context, we obtain 
the version of the Donaldson-Futaki invariant 
$\mathrm{Fut}$ for a test configuration 
$(\mathcal M, \mathcal L)$
\begin{eqnarray}
&&\frac{1}{(2\pi)^n}\mathrm{Fut}\bigl(\mathrm{grad}^{(1,0)}v_Y\bigr)
\label{intersection1}\\
&&\quad = \frac{-2}{n+1}\kappa(M,L)\cup c_1(\mathcal{L})^{n+1}
+2n\left( c_2(\mathcal{M})- \frac{1}{2}c_1^2\Bigl(
\mathcal{K}^{-1}_{\mathcal{M}/\bfC\bfP^1}\Bigr)
\right)\cup c_1(\mathcal{L})^{n-1},\nonumber
\end{eqnarray}
where $\kappa(M,L)$ is the average of the Cahen--Gutt momentum
\begin{equation}\label{intersection2}
\kappa(M,L) \defeq 
n(n-1)\frac{\left(c_2-\frac{1}{2}c_1^2
\right)(M)\cup c_1(L)^{n-2}}{c_1(L)^n},
\end{equation}
and the cup products \( \cup \) are taken on the total 
space of $\mathcal M$, which has dimension $n+1$ in~\eqref{intersection1}, and on $M$, which has dimension 
$n$ in~\eqref{intersection2}, see \parencite{FutakiFuenteGravy2019}.
This suggests that one could define $K$-stability related 
to the study of a K\"ahler metric with constant Cahen--Gutt 
momentum, at least if one can restrict to smooth test 
configurations. The Kempf--Ness functional can be also 
defined as in~\eqref{Kempf1}, but nothing has been
done so far in the literature in this Cahen--Gutt 
momentum map case.

Replacing the role of the scalar curvature by $J(\nabla)$, 
many results in K\"ahler geometry can be
carried over to the geometry of the Cahen--Gutt momentum 
map, see, e.g., \parencite{Fuente-Gravy2016,FutakiOno2018}. 
For example, one can define a Cahen--Gutt version of extremal 
K\"ahler metrics and prove the same structure theorem as 
the Calabi extremal K\"ahler metrics, see 
\parencite[Theorem~4.7]{FutakiOno2018} and 
\parencite[Theorem~6.9]{DiezRatiuSymplecticConnections}.

\subsection{\texorpdfstring{$Z$}{Z}-critical K\"ahler metrics}
In this subsection, we connect Dervan's 
work~\parencite{Dervan2023} to our general results of \cref{sec:general}.
Let $(M,L)$ be a smooth polarized variety of dimension $n$, 
with $L$ an ample line bundle. Let $\mathcal K$ be the set 
of K\"ahler forms in the class $c_1(L)$.
Suppose we are given
$\rho = (\rho_0, \dotsc, \rho_n) \in \mathbb C^{n+1}$
with $\rho_n = \sqrt{-1}$, and 
$\Theta \in \oplus_{j=0}^n H^{j,j}(M,\mathbb C)$
of the form $\Theta = 1 + \Theta'$ where 
$\Theta' \in \oplus_{j>0} H^{j,j}(M,\mathbb C)$.
We define the 
\emph{central charge} $Z(M,L) : \mathbb N \to \mathbb C$ by 
\begin{equation}
Z_k(M,L) = 
\sum_{\ell=1}^n \rho_\ell k^\ell 
\int_M L^\ell \cup \sum_{j=0}^n a_j (-K_M)^j \cup \Theta
\end{equation}
where \( a_j \in \C \) with $a_0=a_1=1$.
Here, \( K_M \) is the canonical bundle of \( M \) 
and we use the convention $K_M^0 = 1$. 
We define the \emph{phase} $\varphi_k (M,L)$ to be 
the argument of $Z_k(M,L)$. Hereafter we often omit $k$ 
and write $Z(M,L)$ and $\varphi(M,L)$, or simply $Z$ and $\varphi$.

We fix a representative $\theta \in \Theta$. 
For $\omega \in \mathcal K$, we define
\begin{equation}
\tilde Z(\omega) = 
\sum_{\ell=1}^n \rho_\ell k^\ell 
\sum_{j=0}^n a_j\tilde Z^{\ell,j} \in 
\sFunctionSpace(M,\mathbb C)\end{equation}
where
\begin{equation}
\tilde Z^{\ell,j} = 
\frac{\omega^\ell \wedge \mathrm{Ric}\, \omega^j 
\wedge \theta}{\omega^n}
- \frac{j}{\ell +1} \Delta 
\left( \frac{\omega^{\ell + 1} \wedge \mathrm{Ric}\, 
\omega^{j-1} \wedge \theta}{\omega^n}  \right).
\end{equation}
Here $\mathrm{Ric}\, \omega$ denotes the Ricci form 
of $\omega$ and $\Delta$ the Laplacian.
Note that
\begin{equation}
\int_M \tilde Z(\omega) \omega^n = Z(M,L).
\end{equation}
We say that $\omega \in \mathcal K$ is a $Z$-critical K\"ahler metric if
\begin{equation}
\Im\bigl(e^{-\I\varphi(M,L)}\tilde Z(\omega)\bigr) = 0
\quad \text{and} \quad 
\mathrm{Re}\bigl(e^{-\I\varphi(M,L)}\tilde Z(\omega)\bigr) > 0.
\end{equation}
Notice that the first condition is equivalent to  
$\arg \tilde Z(\omega) = \arg Z(M,L) \mod \pi$. 

Fix $\omega \in \mathcal K$ and consider it as 
a symplectic form. As before, let  $\mathcal I$ be
the set of all $\omega$-compatible complex structures 
$j$ such that $(M,\omega, j)$ is a K\"ahler
manifold. Dervan constructs a K\"ahler structure 
$\Omega_\epsilon$, $\epsilon = \frac 1 k$,  on 
$\mathcal I$  when $k$ is large, \ie, for the 
large volume limit. Unlike the $L^2$-inner product 
in the previous subsections, the construction of 
$\Omega_\epsilon$ is more nonlinear and we will 
not reproduce it here, see 
\parencite[Equation~(3.11)]{Dervan2023}. Dervan 
shows that $\Im (e^{-i\varphi_\epsilon}\tilde Z)$ 
is a momentum map for the action of the group of 
Hamiltonian diffeomorphisms. The Futaki invariant 
of \cref{prop:cartan:constFutaki} (and \cref{Rem1}) 
in this case is given by
\begin{equation}
\widetilde{\mathcal F}_\zeta (\omega, 1) =
- \int_M (\Im (e^{-\I\varphi_\epsilon}\tilde Z)
- \Im (e^{-\I\varphi_\epsilon}\tilde Z)_0)\,
\varrho_\zeta\ \mu_\omega. 
\label{obst4}
\end{equation}
The vanishing of $\widetilde{\mathcal F}_\zeta (\omega, 1)$ 
is a necessary condition for the existence
of $Z$-critical K\"ahler metrics; in this way, we 
recover \parencite[Corollary~3.7]{Dervan2023}. 
A special form of $\widetilde{\mathcal{F}}_\zeta (\omega, 1)$ 
has been considered by \textcite{Leung1998a}.
The Kempf--Ness functional can be defined again as 
in~\eqref{Kempf1}. 

Instead of treating only $K_M$, the higher Chern classes
can also be treated, see \parencite[Section 4]{Dervan2023} 
and \parencite{DervanHallam2023}. 


\section{Gauge theory}
\label{sec:gauge}

Let $E \to M$ be a $\sFunctionSpace$ complex vector bundle
over a compact K\"ahler manifold $(M,\omega)$, 
$h_0$ a fixed Hermitian metric of $E$, and $\mathcal G$ 
the group of unitary gauge transformations of $E$ with 
respect to $h_0$. Let $\mathcal A(h_0)$ be the space of 
$h_0$-unitary connections $A$ of $E$ 
such that the curvature $F_A$ is type $(1,1)$.
Thus, $F_A^{0,2} = \barpartial_A\circ \barpartial_A = 0$ 
and $A$ defines a holomorphic structure on $E$, 
see \parencite[Theorem~2.1.53]{DonaldsonKronheimer1997}. 
Here we denote by $\dif_A$ the covariant derivative with 
respect to $A \in \mathcal A(h_0)$;
$\barpartial_A$ is the type $(0,1)$-part of $\dif_A$.
Then $\mathcal G$ acts on $\mathcal A(h_0)$ by 
$\dif_{g(A)} = g\circ \dif_A \circ g^{-1}$ so that 
$g(A) = A - \dif_A g \cdot g^{-1}$ for \( g \in \mathcal{G} \). 
Here we used the same letter $A$ to denote the connection 
form with respect to a chosen unitary frame. 
We also denote by $\mathcal G^c$ the space of complex gauge 
transformations of $E$; thus $\mathcal G^c$ is the 
complexification of $\mathcal G$.

We take the Klein pair $(\mathfrak{g}^c, \mathfrak{g})$, where 
$\mathfrak{g}$ is the Lie algebra of $\mathcal{G}$ and 
$\mathfrak{g}^c$ is the complexification of $\mathfrak{g}$. 
In this case, $\mathfrak{g}^c$ integrates to
the complex gauge group $\mathcal{G}^c$ of $E$. 
We take as the Cartan bundle $\SectionSpaceAbb{P}$ modeled 
on the Klein pair $(\mathfrak{g}^c, \mathfrak{g})$ with 
group $\mathcal{G}$ the principal bundle $\SectionSpaceAbb{P}
=\mathcal{G}^c \to \mathcal{G}\backslash \mathcal{G}^c$ 
with the left $\mathcal{G}$-action. With the right Maurer--Cartan
form $\theta$ on \( \SectionSpaceAbb{G}^c \), the pair 
$(\SectionSpaceAbb{P},\theta)$ is a Cartan geometry modeled 
on the Klein pair $(\mathfrak{g}^c, \mathfrak{g})$.
Alternatively, one can consider the space of pairs $(h, g)$,
where $h$ is a Hermitian metric on $E$ and $g \in \mathcal G^c$,
satisfying $h(x,y) = h_0(gx,gy)$. The left action of $\mathcal{G}$ 
on $\SectionSpaceAbb{P}$ is given by  $k\cdot (h,g) = (h,kg)$. 
Then $\SectionSpaceAbb{P}$ is a principal $\mathcal{G}$-bundle 
over the homogenous space $\mathcal{G}\backslash \mathcal{G}^c$
and $\SectionSpaceAbb{P}$ is diffeomorphic to $\mathcal{G}^c$.
Note the formal similarity of the latter definition of 
$\SectionSpaceAbb{P}$ to the Cartan bundle in K\"ahler 
geometry introduced in \cref{sec:kaehler}.
Since any two Hermitian metrics are related by a complex 
gauge transformation, the base 
\( \mathcal G\backslash \mathcal G^c \) is identified 
with the space of Hermitian metrics on \( E \).
In particular, \( \mathcal G\backslash \mathcal G^c \) 
is simply connected.

Fix $A \in \mathcal A(h_0)$. There is a smooth map 
$\chi : \SectionSpaceAbb{P} \to \mathcal A(h_0)$ induced
by the action of $\mathcal G^c$ on $\mathcal A(h_0)$.
To describe this action, let $A = A^{1,0} + A^{0,1}$ be 
the decomposition of $A \in \mathcal A(h_0)$ into its
$(1,0)$ and $(0,1)$-components so that $A^{1,0} = 
- (A^{0,1})^\ast$, where ${}^\ast$ denotes the transpose-conjugate 
with respect to $h_0$. For $g \in \mathcal G^c$, define the 
partial connection by $\barpartial_{g(A)}\defeq
g\circ \barpartial_A \circ g^{-1}$ or, equivalently, by
\begin{equation}
g(A)^{0,1} = A^{0,1} - \barpartial_A^\ast g \cdot g^{-1}.
\end{equation}
In these formulas, $\circ$ is the composition of operators 
and $\cdot$ denotes matrix multiplication. 
A partial connection defines a unique unitary connection 
$g(A)$ by
\begin{equation}
g(A)^{1,0} \defeq -(g(A)^{0,1})^\ast 
= A^{1,0} + (\barpartial_A^\ast g \cdot g^{-1})^\ast,
\end{equation}
see \parencite[Lemma~2.1.54]{DonaldsonKronheimer1997}. 
If $g$ is a unitary gauge transformation in $\mathcal G$, 
then $\dif_{g(A)} = g\circ \dif_A \circ g^{-1}$ or, equivalently, 
\begin{equation}
g(A) = A + g\dif_A(g^{-1}) = A - \dif_A g\cdot g^{-1}.
\end{equation}
Thus, the $\mathcal{G}^c$-action on $\mathcal A(h_0)$ 
is a natural extension of the $\mathcal{G}$-action.

Since a unitary connection is uniquely determined by a 
partial connection, the tangent space $T_A\mathcal A(h_0)$ 
can be identified with $\Omega^{0,1}(\End(E))$ and there 
is an almost complex structure on $\mathcal A(h_0)$ defined 
by the multiplication by $\I$ on $\Omega^{0,1}(\End(E))$.
It is easy to see that the Cartan bundle 
\( (\mathcal{P}, \theta) \) provides a model for the 
complex orbit through \( A \) in the sense of 
\cref{def:cartan:orbitModel}. Note that $\theta$ is torsion free.

\subsection{Hermitian Yang--Mills connections}
Keeping the same notations, let $h_0$ be a Hermitian metric 
of a holomorphic vector bundle $E \to M$ over a compact 
\(2n\)-dimensional K\"ahler manifold $(M,\omega)$ and 
$\mathcal A(h_0)$ the space of $h_0$-unitary connections. 
There is a symplectic form $\Omega$ on $\mathcal A(h_0)$
defined by
\begin{equation}
\Omega(\xi,\eta) = - \int_M \tr(\xi\wedge\eta) \,\omega^{n-1}
\end{equation}
for $ \xi , \eta \in \TBundle_A \mathcal A(h_0)=\Omega^{0,1}(\End(E))$.
It is well known 
\parencite{AtiyahBott1983,Donaldson1985,DonaldsonKronheimer1997} 
and easy to show that
\begin{equation}
J(A) = F_A\wedge \omega^{n-1} -  \lambda \, \id\, \omega^n
\end{equation}
is a momentum map on the symplectic manifold 
$\bigl(\mathcal A(h_0), \Omega\bigr)$ for the action 
of the unitary gauge transformations $\mathcal{G}$ for 
any purely imaginary constant $\lambda$.
Thus, by \cref{prop:cartan:constFutaki}, the Futaki 
functional $\widetilde{\mathcal{F}}$ on the stabilizer 
\( (\mathfrak{g}^c)_A \) of $A= \chi(h_0,1)$ of 
the action of $\mathfrak{g}^c$, defined by
\begin{equation}\label{obst5}
\widetilde{\mathcal F}(\zeta) =  
- \Im \int_M \tr\bigl(\zeta (F_A \wedge \omega^{n-1} 
-  \lambda \, \id \, \omega^n)\bigr), \qquad 
\zeta \in (\mathfrak g^c)_A 
\end{equation}
is independent of the Hermitian metric $h_0$. Note that 
the unitary connection $A$ determines a holomorphic
structure $\barpartial_A$ on $E$ and thus the infinitesimal 
complex gauge transformation $\zeta$ in the
stabilizer of $A$ is just a holomorphic endomorphism of 
$E$ with respect to $\barpartial_A$.

Choosing $\lambda$ to be the topological invariant determined 
by 
\begin{equation}
-2\pi \I r\lambda = c_1(E)[\omega]^{n-1}
\end{equation} 
with $r=\mathrm{rank}(E)$, the Chern connection $A$ of 
a Hermitian metric $h$ is called a \emph{Hermitian 
Yang--Mills connection} (or $h$ is called a 
\emph{Hermitian-Einstein metric}) if its curvature 
$F_A$ satisfies
\begin{equation}
F_A\, \omega^{n-1} =  \lambda \, \id \, \omega^n.
\end{equation}
Thus $\widetilde{\mathcal F}$ is an obstruction to the 
existence of Hermitian Yang--Mills connections. 
However, it is well known that if a Hermitian Yang--Mills 
connection exists on an irreducible holomorphic
vector bundle $E$, then $E$ is simple, meaning that 
any holomorphic section $\zeta$ of the endomorphism 
bundle of $E$ is a multiple of the identity. Therefore, 
when trying to prove existence of Hermitian Yang--Mills 
solutions, it is necessary to assume that $E$ is simple 
(and even slope stable). With this natural assumption, 
the obstruction~\eqref{obst5} always vanishes.

Note, in passing, that we have the following criterion 
for the vanishing of the Futaki character which follows from
Kodaira--Serre duality.
\begin{proposition}
The Futaki character~\eqref{obst5} vanishes if and only 
if $F_A \omega^{n-1} - \lambda \omega^n$ is trivial in 
$H^{n,n}(M) = H^{2n}(M, \C)$.
\end{proposition}

The Kempf--Ness functional defined in \cref{sec:general}
recovers in this case the \emph{Donaldson functional}, seen as a functional 
on the space of Hermitian metrics.
The Donaldson functional has been effectively used by
\textcite{Donaldson1985} for the existence problem of Hermitian 
Yang--Mills connections on algebraic surfaces 
and was the precursor of the Mabuchi K-energy.
See \parencite{UhlenbeckYau1986} for the higher dimensional result.

We remark that the above arguments still work even if 
the Hermitian metric is indefinite. 
\Textcite{GarciaFernandezMolina2023} used~\eqref{obst5} 
for indefinite metrics as an obstruction to the existence 
of solutions of the Hull--Strominger system. See also 
\parencite{GarciaFernandezRubioTipler2020,GarciaFernandezRubioTipler2024}.

\subsection{\texorpdfstring{$Z$}{Z}-critical connections}
In this subsection, we apply our results of \cref{sec:general} 
to the setting of \textcite{DervanMcCarthySektnan2020}.
Let $E \to M$ be a holomorphic vector bundle over a compact 
K\"ahler manifold $(M,\omega)$.
A polynomial central charge $Z_k(E)$ of $E$ is
\begin{equation}
Z_k(E) = \int_M \sum_{d=0}^n\, \rho_d \,k^d 
\,[\omega]^d \cup \chernClass(E)\cup U
\end{equation}
where $(\rho_0, \rho_1, \dotsc, \rho_n) \in 
(\mathbb C^\ast)^{n+1}$ with $\Im \rho_n > 0$, 
$\Re\rho_{n-1} <0$, and $U = 1 + N$ with 
$N \in H^{>0}(M,\mathbb R)$ are given data.
Here, \( \chernClass(E) \in H^*(M) \) denotes  the Chern 
character of \( E \). For each $E$ with $Z_k(E) \ne 0$, 
we define the phase
\begin{equation}
\varphi_k(E) \defeq \arg Z_k(E).
\end{equation}
We often omit $k$, and write \( Z(E) \) and 
\( \varphi(E) \) when no confusion is likely to occur.
Let $\tilde U \in \Omega^\ast(M, \mathbb R)$ be a 
closed form representing $U \in H^\ast(M,\mathbb R)$.
For a Hermitian metric $h$ of $E$ with Chern connection 
$A$, we denote by $\widetilde{\chernClass}(A) = 
\exp\left(\frac{i}{2\pi}\,F_A\right)$ the endomorphism-valued 
differential form whose trace represents the 
Chern character $\chernClass(E)$.
We define an endomorphism-valued
$(n,n)$-form $\tilde Z_k(A)$ by
\begin{equation}
\tilde Z_k(A) = \left[\sum_{d=0}^n \rho_d\,k^d\,
\omega^d\wedge\widetilde{\chernClass}(A)\wedge
\tilde{U}\right]_{n,n},
\end{equation}
where on the right-hand side above we only select the 
$(n,n)$-form component of the differential 
form in the brackets with various mixed degrees.  Obviously,
\begin{equation}\label{average1}
 0 = \Im\bigl(e^{-\I\varphi_k(E)}\tilde Z_k(E)\bigr) 
 = \int_M\,\tr \, \Im
 \bigl(e^{-\I\varphi_k(E)}\tilde Z_k(A)\bigr).
 \end{equation}
We say that $A$ is a \emphDef{$Z$-critical (or $Z_k$-critical) 
connection} if
\begin{equation}
\Im\bigl(e^{-\I\varphi_k(E)}\tilde Z_k(A)\bigr) = 0,
\end{equation}
where the imaginary part means $-\I$-times the $h$-skew 
Hermitian part. 

Fix a Hermitian metric $h_0$ and let $\mathcal A(h_0)$ be 
the space of all $h_0$-unitary connections.
Define a Hermitian pairing on $\TBundle_A \mathcal A(h_0) = 
\Omega^{0,1}(\End(E))$ by
\begin{equation}
\langle \xi,\eta \rangle_A = 
-\I \int_M \tr \, 
\left[ \Im( e^{-\I\varphi(E)}\,\tilde Z^\prime(A)
\wedge \xi \wedge \eta^\ast)\right]_{\mathrm{sym}}
\end{equation}
where $\tilde Z^\prime(A)$ denotes the derivative of 
$\tilde Z(A)$ with respect to $\frac{\I}{2\pi}\,F_A$,
considered as a variable, and where
$\mathrm{sym}$ denotes the graded symmetric product 
of endomorphism valued forms
\begin{equation}
[B_1 \wedge \cdots \wedge B_j]_{\mathrm{sym}}
= \frac 1{j!} \sum_{\sigma \in S_j} 
(-1)^{\mathrm{gradsgn}\,\sigma}
B_{\sigma(1)}\wedge \cdots \wedge B_{\sigma(j)};
\end{equation}
here, gradsgn means the sign resulting from the 
permutation of the differential forms $B_i$'s.
Then the imaginary part 
\begin{equation}
\Omega_A(\xi,\eta) = \Im \langle \xi,\eta \rangle_A
\end{equation}
is a symplectic form on $\mathcal A(h_0)$ provided 
$ \langle \cdot \, ,\,\cdot \rangle_A $ is positive
definite. When this latter condition is satisfied, 
$A$ is called a \emph{subsolution}. 
When $A$ is a subsolution, then the PDE for $A$ being 
a $Z$-critical connection becomes elliptic,
see \parencite[Lemma 2.36 and Proposition 2.43]{DervanMcCarthySektnan2020}.
By \parencite[Theorem~2.45]{DervanMcCarthySektnan2020}, the map
\begin{equation}
A \mapsto 2\pi \I \Im \bigl(e^{-\I\varphi(E)}\tilde Z(A)\bigr)
\end{equation}
is a momentum map for the $\mathcal G$-action on 
$\bigl(\mathcal A(h_0),\Omega\bigr)$. Thus, by 
\cref{prop:cartan:constFutaki}, the Futaki functional 
$\widetilde{\mathcal F}$ on $H^0(M,\End(E))$, defined 
 in this case by
\begin{equation}\label{obst6}
\widetilde{\mathcal F}(\zeta) \defeq  
2\pi \int_M \tr\Bigl(\zeta \, \Im \bigl(
e^{-\I\varphi(E)}\tilde Z(A)\bigr)\Bigr),
\end{equation}
is independent of the Hermitian metric $h_0$.
Clearly, the non-vanishing of $\widetilde{\mathcal F}$
is an obstruction to the existence of a Hermitian metric 
$h$ whose Chern connection $A$ is a $Z$-critical connection. 
By \cref{i:cartan:kempfNess:existence}, we can also obtain the Kempf--Ness functional just as
the Donaldson functional, \parencite{CollinsYau2021}.

A holomorphic vector bundle $E \to (M,\omega)$ with 
$Z_k(E) \ne 0$ for all $k \gg 0$ is said to be 
\emph{asymptotically $Z$-stable} if for all proper non-zero 
subsheaves $F \subset E$ there exists some $k_0$ such that 
$ \varphi_k(F) < \varphi_k(E)$ for all $k \ge k_0$. By 
\parencite[Lemma 2.18]{DervanMcCarthySektnan2020}, 
any asymptotically $Z$-stable 
bundle is simple. Thus, by~\eqref{average1}, 
$\widetilde{\mathcal F}$ vanishes on any asymptotically 
$Z$-stable bundle.


In the special case where \( E \) is a line bundle \( L \), 
the central charge is
\begin{equation}
Z_k(L) = \I^{n+1}\int_M e^{-\I k[\omega]} \, \chernClass(L).
\end{equation}
Moreover, in this case, a $Z$-critical connection 
is a solution of
\begin{equation}
\Im\bigl(e^{-\I\varphi}\tilde Z_k(A)\bigr) = 
\frac1{n!} \Im 
\left( e^{-\I\varphi}\left( \omega - 
\frac{F_A}{2\pi} \right) \right) = 0.
\end{equation}
This equation is called the LYZ equation after 
\parencite{LeungYauZaslow2000}, which was formulated as the mirror
to the special Lagrangian equation; its solutions 
are called the deformed Yang--Mills connections. 
This subject has been well-studied and there are
many interesting papers, but most relevant to  
our setting are \parencite{CollinsYau2021,CollinsShi2022}.
Since the endomorphism bundle of a line bundle is trivial, 
its holomorphic sections are constant functions. Thus, 
the obstruction~\eqref{obst6} vanishes. The Kempf--Ness 
functional is explicitly described using the Calabi-Yau 
functional and is studied in terms of geodesics in 
\parencite{CollinsYau2021}. 

\begin{refcontext}[sorting=nyt]{}
	\printbibliography
\end{refcontext}

\end{document}